\documentclass[10pt]{article}

\usepackage{xspace}
\usepackage{url}
\usepackage{mathtools}
\usepackage{amssymb}
\usepackage{amsthm}
\usepackage{empheq}
\usepackage{latexsym}
\usepackage{enumitem}
\usepackage{eurosym}
\usepackage{dsfont}
\usepackage{appendix}
\usepackage{color} 
\usepackage[bookmarks=false]{hyperref}
\usepackage{frcursive}
\usepackage[utf8]{inputenc}
\usepackage[T1]{fontenc}
\usepackage{geometry}
\usepackage{multirow}
\usepackage{comment}
\usepackage{todonotes}
\usepackage{lmodern}
\usepackage{anyfontsize}
\usepackage{stmaryrd}
\usepackage{natbib}
\usepackage{cleveref}
\usepackage[english]{babel}
\usepackage[english=british]{csquotes}
\usepackage{accents}
\usepackage{bm}
\usepackage{subfig}
\usepackage{aliascnt}

\setcitestyle{numbers,open={[},close={]}}

\definecolor{red}{rgb}{0.7,0.15,0.15}
\definecolor{green}{rgb}{0,0.5,0}
\definecolor{blue}{rgb}{0,0,0.7}
\hypersetup{colorlinks, linkcolor={red}, citecolor={green}, urlcolor={blue}}
			
\makeatletter \@addtoreset{equation}{section}

\newtheorem{theorem}{Theorem}[section]

\newaliascnt{assumption}{theorem}
\newtheorem{assumption}[assumption]{Assumption}
\aliascntresetthe{assumption}

\newaliascnt{proposition}{theorem}
\newtheorem{proposition}[proposition]{Proposition}
\aliascntresetthe{proposition}

\newaliascnt{definition}{theorem}
\newtheorem{definition}[definition]{Definition}
\aliascntresetthe{definition}

\newaliascnt{lemma}{theorem}
\newtheorem{lemma}[lemma]{Lemma}
\aliascntresetthe{lemma}

\newaliascnt{example}{theorem}
\newtheorem{example}[example]{Example}
\aliascntresetthe{example}

\newaliascnt{corollary}{theorem}
\newtheorem{corollary}[corollary]{Corollary}
\aliascntresetthe{corollary}

\newaliascnt{remark}{theorem}
\newtheorem{remark}[remark]{Remark}
\aliascntresetthe{remark}

\newaliascnt{condition}{theorem}

\aliascntresetthe{condition}

\crefname{theorem}{theorem}{theorems}
\Crefname{theorem}{Theorem}{Theorems}

\crefname{assumption}{assumption}{assumptions}
\Crefname{assumption}{Assumption}{Assumptions}

\crefname{proposition}{proposition}{propositions}
\Crefname{proposition}{Proposition}{Propositions}

\crefname{definition}{definition}{definitions}
\Crefname{definition}{Definition}{Definitions}

\crefname{lemma}{lemma}{lemmas}
\Crefname{lemma}{Lemma}{Lemmas}

\crefname{example}{example}{examples}
\Crefname{example}{Example}{Examples}

\crefname{corollary}{corollary}{corollaries}
\Crefname{corollary}{corollary}{Corollaries}

\crefname{remark}{remark}{remarks}
\Crefname{remark}{Remark}{Remarks}

\crefname{equation}{condition}{conditions}
\Crefname{equation}{Condition}{Conditions}

\DeclareUnicodeCharacter{014D}{\=o}
\newcommand{\smallertext}[1]{\text{\fontsize{5}{5}\selectfont$#1$}}
\newcommand{\smalltext}[1]{\text{\fontsize{4}{4}\selectfont$#1$}}
\newcommand{\tinytext}[1]{\text{\fontsize{3}{3}\selectfont$#1$}}

\newcommand{\diff}{\mathrm{d}}

\def\eps{\varepsilon}
\def\d{{\mathrm{d}}}
\def\1{\mathbf{1}}

\def \E{\mathbb{E}}
\def \F{\mathbb{F}}
\def \G{\mathbb{G}}

\def \L{\mathbb{L}}

\def \N{\mathbb{N}}

\def \P{\mathbb{P}}
\def \Q{\mathbb{Q}}
\def \R{\mathbb{R}}

\def\Ac{\mathcal{A}}

\def\Cc{\mathcal{C}}

\def\Ec{\mathcal{E}}
\def\Fc{\mathcal{F}}
\def\Gc{\mathcal{G}}

\def\Lc{\mathcal{L}}
\def\Mc{\mathcal{M}}
\def\Nc{\mathcal{N}}

\def\Tc{\mathcal{T}}

\def\eps{\varepsilon}
\def\d{\mathrm{d}}

\DeclareMathOperator*{\Prob}{Prob} 
 
\let\ae\relax
\DeclareMathOperator*{\ae}{\text{\rm a.e.}}

\DeclareMathOperator*{\Tr}{Tr}
\newcommand{\al}{\alpha}

\renewcommand{\|}{\Vert}

\makeatletter
\def\moverlay{\mathpalette\mov@rlay}
\def\mov@rlay#1#2{\leavevmode\vtop{%
   \baselineskip\z@skip \lineskiplimit-\maxdimen
   \ialign{\hfil$\m@th#1##$\hfil\cr#2\crcr}}}
\newcommand{\charfusion}[3][\mathord]{
    #1{\ifx#1\mathop\vphantom{#2}\fi
        \mathpalette\mov@rlay{#2\cr#3}
      }
    \ifx#1\mathop\expandafter\displaylimits\fi}
\makeatother

\begin{document}
\title{Here, there and everywhere: state-dependent time-inconsistent stochastic control}
\author{Dylan \textsc{Possama\"i}\footnote{ETH Z\"urich, Mathematics department, Switzerland, dylan.possamai@math.ethz.ch. This author gratefully acknowledges partial support by the SNF project MINT 205121-219818.} \and Mateo \textsc{Rodriguez Polo}\footnote{ETH Z\"urich, Mathematics department, Switzerland, mateo.rodriguezpolo@math.ethz.ch. This author gratefully acknowledges partial support by the SNF project MINT 205121-219818.}}

\date{\today}

\maketitle

\begin{abstract}

This paper addresses the challenge of time-inconsistent stochastic control within a continuous-time framework. Its primary focus lies in uncovering a probabilistic representation, specifically in the shape of a system of backward stochastic differential equations (BSDEs). These equations encapsulate the equilibrium value function essential for resolving cases where the present state affecting the target functional triggers the inconsistency. Additionally, the paper offers an application exemplifying this theory through the time-inconsistent linear--quadratic regulator.

\end{abstract}
\setlength{\parindent}{0pt}

\allowdisplaybreaks

\section{Introduction}\label{sec:intro}

Classical stochastic control is largely built around an intertemporal consistency principle: the policy that is optimal when the problem is posed at time $0$ remains optimal when the same optimisation is reconsidered at any later time $t$, conditional on the information available at $t$. This property is the backbone of Bellman's dynamic programming principle (DPP). It allows one to propagate value functions through conditioning and concatenation, and it leads to tractable characterisations of optimal feedback controls via Hamilton--Jacobi--Bellman (HJB) equations and verification arguments; see, for instance, \citeauthor*{fleming2006controlled} \cite{fleming2006controlled}, or \citeauthor*{yong1999stochastic} \cite{yong1999stochastic}.

\medskip
A large and important family of economically and financially motivated objectives violates this principle. In a \emph{time-inconsistent} control problem, the continuation criterion used by the agent at time $t$ differs from the criterion that will be used at a later date $s>t$. As a consequence, a plan designed at time $0$ is typically not self-enforcing: when time $t$ arrives, the agent re-optimises and may deviate from the original plan even when the underlying dynamics have not changed. Time inconsistency therefore fundamentally alters the nature of the problem. Since a global optimum in the classical sense is no longer necessarily meaningful, the relevant solution concept must be reconsidered, and one needs new analytical tools to replace the missing DPP.

\medskip
A natural resolution, going back to \citeauthor*{strotz1955myopia} \cite{strotz1955myopia}, is to interpret time inconsistency as an \emph{intrapersonal dynamic game} in which the `players' are the agent’s successive selves. This viewpoint clarifies three canonical behavioural benchmarks. A \emph{pre-committed} agent computes an optimum at time~$0$ and follows it regardless of future incentives. A \emph{naive} agent re-optimises over time as if the current plan would never be revised again. The \emph{sophisticated} (game-theoretic) agent studied in this paper instead seeks a self-enforcing, subgame-perfect strategy: no self has an incentive to deviate, given that later selves will also behave optimally from their own perspective. In discrete time, this `consistent planning' paradigm is classical \citeauthor*{phelps1968second} \cite{phelps1968second}, \citeauthor*{pollak1968consistent} \cite{pollak1968consistent}, \citeauthor*{peleg1973existence} \cite{peleg1973existence}, and it also provides behavioural foundations for quasi-hyperbolic and more general forms of discounting \citeauthor*{laibson1997golden} \cite{laibson1997golden}, \citeauthor*{odonoghue1999doing} \cite{odonoghue1999doing}. We will illustrate the quantitative gap between precommitment, naivety and sophistication in our linear--quadratic example in \Cref{sec:LQR}, and remark that analogous two-layer game-theoretic structures also arise when time inconsistency interacts with strategic considerations in multi-player games \citep{possamai2025variance}.
\medskip

In continuous time, equilibrium notions are necessarily local. A `current self' is allowed to deviate only on a short time interval, while taking the continuation behaviour of future selves as fixed, so that equilibrium controls are locally optimal in the sense of an infinitesimal deviation analysis. Several equilibrium concepts coexist in the stochastic control literature, reflecting both modelling choices (open-loop versus feedback strategies) and analytical requirements (how deviations are measured, and what regularity is imposed on the candidate strategy). The strong/weak equilibrium distinction of \citeauthor*{huang2021strong} \cite{huang2021strong} and the subsequent analysis of equilibrium notions in \citeauthor*{he2021equilibrium} \cite{he2021equilibrium} make this particularly transparent. A related, widely used notion is that of \emph{regular equilibrium}, which is tailored to the extended HJB approach and is closely connected to the solvability of equilibrium PDE systems \citeauthor*{lindensjo2019regular} \cite{lindensjo2019regular}, \citeauthor*{bjork2021time} \cite{bjork2021time}. In this paper we focus on feedback equilibria in the sense of local deviations, as this is the natural notion for dynamic programming.

\medskip
Time inconsistency can be generated by several conceptually distinct mechanisms, and the continuous-time literature reflects this diversity. First, and perhaps most prominently, \emph{non-exponential discounting} destroys stationarity: the discount factor depends on the evaluation time and induces a re-weighting of future payoffs as time passes. In continuous time this mechanism motivated the pioneering equilibrium analysis of \citeauthor*{ekeland2006being} \cite{ekeland2006being,ekeland2010golden}, \citeauthor*{ekeland2008investment} \cite{ekeland2008investment}. It remains a benchmark class and has been revisited in general Markovian settings; see, for instance, \citeauthor*{bjork2017time} \cite{bjork2017time,bjork2021time}.

\medskip
Second, \emph{nonlinear dependence on conditional expectations} breaks the DPP even when discounting is exponential. The paradigmatic example is the mean--variance criterion, which introduces a variance term (a nonlinear function of an expectation) into the objective and is central in dynamic Markowitz portfolio selection. Equilibrium formulations for mean--variance and related deviation--risk criteria have been developed in, among many others, \citeauthor*{basak2010dynamic} \cite{basak2010dynamic}, \citeauthor*{bjork2014mean} \cite{bjork2014mean}, \citeauthor*{gu2020constrained} \cite{gu2020constrained}. This line of work has also motivated robust and ambiguity-averse formulations, where time inconsistency and model uncertainty interact; see, \emph{e.g.}, \citeauthor*{pun2018robust} \cite{pun2018robust}. Time inconsistency also interacts with additional modelling features such as regime switching and discrete interventions; equilibrium analyses of time-inconsistent stochastic switching problems can be found in, for instance, \citeauthor*{mei2019equilibrium} \cite{mei2019equilibrium}.

\medskip
Third, and this is the focus of the present paper, time inconsistency may stem from \emph{state-dependent preference parameters}. In many models the criterion depends on a parameter that is updated as the state evolves---wealth-dependent risk aversion, moving targets, relative-performance benchmarks, or endogenous reference points. When this parameter is \emph{recalibrated} by each future self, different selves effectively face different objective functionals even if discounting is exponential and the reward structure is otherwise time-homogeneous. State-dependent risk aversion in deviation--risk criteria provides one family of examples \citep{bjork2014mean,gu2020constrained,pun2018robust}, but the mechanism is broader: the preference parameter may itself be the state used as a reference point, as in the criterion considered in~\eqref{eq:intro-criterion} below.

\medskip
A further important class, closely related to nonlinear expectation criteria, arises in \emph{recursive} (BSDE-type) objectives: time inconsistency can emerge from a lack of flow property in the backward component and from non-separable aggregation. This has led naturally to equilibrium characterisations in terms of flows of forward--backward SDEs and, more generally, backward stochastic Volterra integral equations (BSVIEs); see \citeauthor*{wei2017time} \cite{wei2017time}, \citeauthor*{hamaguchi2020extended} \cite{hamaguchi2020extended}, \citeauthor*{wang2021time} \cite{wang2021time}, \citeauthor*{mastrogiacomo2023subgame} \cite{mastrogiacomo2023subgame}.

\medskip
Finally, time-inconsistent \emph{stopping} (and mixed control--stopping) problems form a parallel and active strand of the literature, where the game-theoretic equilibrium concept takes a different form but shares the same conceptual origin. We refer to \citeauthor*{christensen2018finding} \cite{christensen2018finding,christensen2020time}, \citeauthor*{bayraktar2021notions} \cite{bayraktar2021notions}, \citeauthor*{bodnariu2022local} \cite{bodnariu2022local} for representative recent works and for further references.

\medskip
We concentrate on a Markovian controlled diffusion in weak formulation and on objective functionals of the form
\begin{equation}\label{eq:intro-criterion}
    J(t,x,\alpha)\coloneqq \mathbb{E}^{\P^{\smalltext{t}\smalltext{,}\smalltext{x}\smalltext{,}\smalltext{\alpha}}}\bigg[\int_t^T f\big(s,x,X_s,\alpha_s\big)\,\mathrm{d}s+\xi(x,X_T)\bigg],
    \; (t,x,\alpha)\in[0,T]\times\R^n\times\mathcal{A},
\end{equation}
where $X$ denotes the controlled state, $\alpha$ is the control, and the crucial feature is the appearance of the \emph{current state} $x$ as an additional argument in both the running and terminal payoff. When the same problem is re-evaluated at time $s>t$, the parameter $x$ is updated to $X_s$, so the continuation criterion differs from~\eqref{eq:intro-criterion} even if the control law is kept fixed. Such state-dependent updating is natural whenever payoffs are formulated relative to a moving target or a reference point that evolves with the system, rather than being fixed at time~$0$.

\medskip
At a formal level, criteria of the form~\eqref{eq:intro-criterion} are encompassed by the general Markovian equilibrium frameworks of \cite{bjork2010general,bjork2017time,bjork2021time}. The key insight in these frameworks is that equilibrium behaviour is described not by a single value function but by an extended object (an `equilibrium value function' together with auxiliary functions) whose diagonal captures the continuation values faced by each self. However, the existing Markovian literature at this level of generality proceeds primarily via \emph{verification-type} results: one postulates an extended HJB system (a coupled system of nonlinear PDEs in multiple variables) and proves that any sufficiently smooth solution yields an equilibrium control. This approach was pioneered and systematised in \cite{bjork2010general,bjork2017time} and remains central in the monograph \cite{bjork2021time}. Parallel approaches based on Pontryagin-type maximum principles lead to equilibrium characterisations in terms of flows of forward--backward SDEs, especially in linear--quadratic settings; see \citeauthor*{hu2012time} \cite{hu2012time}, \citeauthor*{hu2017time} \cite{hu2017time} and the references therein. There are also contributions focusing on the existence of closed-loop equilibria in more general models and on the relationship between different equilibrium notions; see, \emph{e.g.}, \citeauthor*{yong2012time} \cite{yong2012time}, \citeauthor*{huang2021strong} \cite{huang2021strong}, \citeauthor*{he2021equilibrium} \cite{he2021equilibrium}, \citeauthor*{wang2021closed} \cite{wang2021closed}.

\medskip
Despite this substantial progress, genuinely state-dependent time inconsistency raises conceptual and technical obstacles that, in our view, have not been fully resolved at the level of dynamic programming. The key difficulty is that the preference parameter driving the inconsistency becomes stochastic once it is updated to the current state. From a dynamic programming viewpoint, the equilibrium object is therefore not a single scalar value function: one must keep track of a \emph{family} of continuation values indexed by a reference parameter (the `reference state'), together with a consistent mechanism that selects the correct diagonal when the parameter is updated along the state process. In smooth PDE approaches this manifests in the need to solve an extended HJB system on an enlarged state space and to evaluate the solution along a diagonal. Outside smooth settings, however, it is not \emph{a priori} clear how to interpret this diagonal, how it evolves along the diffusion, and how it interacts with the equilibrium definition based on local deviations.

\medskip
By contrast, the most complete rigorous dynamic programming foundations currently available in the time-inconsistent literature focus on mechanisms where the preference parameter is either deterministic (as in non-exponential discounting) or enters through conditional expectations (as in mean--variance and deviation--risk criteria). In these cases one can often set up a flow of value functions indexed by the initial time or by auxiliary expectation variables and derive extended HJB systems, FBSDE flows, and/or BSVIE characterisations \citep{basak2010dynamic,ekeland2006being,ekeland2010golden,ekeland2008investment,hu2012time,wang2021time}. Recent works have also developed dynamic programming and viscosity-solution methods for the resulting extended HJB systems in specific settings \citeauthor*{karnam2016dynamic} \cite{karnam2016dynamic}, \citeauthor*{xu2022dynamic} \cite{xu2022dynamic}. The non-Markovian theory of \citeauthor*{hernandez2023me} \cite{hernandez2023me} provides a very general equilibrium DPP and BSDE representation for sophisticated agents, but does not cover the Markovian specialisation required for state-dependent reference parameters.

\medskip
To the best of our knowledge, a \emph{fully rigorous} dynamic programming treatment of time inconsistency stemming from \emph{state-dependent preference updating} of the form~\eqref{eq:intro-criterion} has been missing. While state dependence is present in the general Markovian frameworks above, existing results in that direction are predominantly verification-type. They do not derive a dynamic programming principle that is both necessary and sufficient and that explicitly propagates the state-dependent preference parameter through time. Providing such a dynamic programming principle, and turning it into a concrete probabilistic representation, is the central objective of the present paper.

\medskip
We develop a rigorous and operational dynamic programming theory for state-dependent time-inconsistent stochastic control in continuous time. We work in weak formulation for a controlled diffusion with uncontrolled volatility, and we seek feedback equilibrium controls. The analysis is probabilistic throughout, and the main output is an equilibrium DPP together with a Markovian system of backward stochastic differential equations (BSDEs) characterising the equilibrium value.

\medskip
The starting point is the non-Markovian equilibrium DPP of \cite{hernandez2023me}. In the state-dependent Markovian setting, this suggests that the equilibrium value at $(t,x)$ should be understood as the diagonal of a flow of continuation values indexed by a reference parameter. Turning this into a tractable Markovian object requires a way to evaluate such a flow \emph{along} the random curve given by the state process when the reference parameter is updated. The key tool enabling this step is the It\^o--Kunita--Wentzell formula \citeauthor*{kunita1981some} \cite{kunita1981some}. Roughly speaking, the It\^o--Kunita--Wentzell formula allows us to compute the semimartingale decomposition of a random field evaluated along a stochastic flow. In our context, it provides a clean and explicit `diagonal calculus' for the equilibrium flow and makes the additional drift terms generated by state dependence transparent.

\medskip
The resulting BSDE system yields a probabilistic counterpart to extended HJB systems that is compatible with low regularity. It also clarifies the role of diagonal objects that appear throughout the equilibrium literature (both in PDE and FBSDE formulations) and that are intimately connected to the local deviation structure of equilibrium definitions \citep{he2021equilibrium,hu2012time,huang2021strong}. For completeness, we recall that BSDE methods play a central role in stochastic control, both as a probabilistic representation of PDEs and as a natural language for recursive criteria; see, \emph{e.g.}, \citeauthor*{pardoux1990stochastic} \cite{pardoux1990stochastic}, \citeauthor*{el1997backward} \cite{el1997backward}.

\medskip
A second theme of the paper is a unification of \emph{time-dependent} and \emph{state-dependent} time inconsistency. In standard (time-consistent) optimal control, explicit time dependence can always be reduced to state dependence by augmenting the state with a clock variable \citep{fleming2006controlled,yong1999stochastic}. While this observation is classical, it has not been systematically exploited at the level of equilibrium dynamic programming for sophisticated agents. The reason is that, without a complete treatment of state-dependent preference updating, the reduction is essentially formal: one may embed time into an enlarged state space, but one still needs to understand how the equilibrium flow and its diagonal behave when the preference parameter becomes a component of the state.

\medskip
Our probabilistic approach, and in particular the It\^o--Kunita--Wentzell based diagonal calculus, makes this reduction transparent and explicit in the equilibrium setting. It shows that non-exponential discounting can be viewed as a special instance of state-dependent preference updating (with the `reference' being the augmented state, i.e.\ the clock), and it clarifies how the BSDE systems appearing in the discounting literature are recovered as a degenerate case of the general state-dependent theory. In that sense, the present work does more than recall the classical state-augmentation trick: it provides the missing state-dependent equilibrium theory that makes the reduction operational.

\medskip
The present paper provides, to our knowledge, the first complete Markovian dynamic programming theory for time inconsistency driven by state-dependent preference updating. Concretely, our contributions can be summarised as follows.

\smallskip
\noindent\textit{$(i)$ Equilibrium \emph{DPP} and Markovian \emph{BSDE} characterisation for state dependence.}
We establish an equilibrium DPP for the criterion~\eqref{eq:intro-criterion} and derive a Markovian system of BSDEs whose solution characterises both the equilibrium value and the equilibrium feedback control. This yields a probabilistic analogue of the extended HJB approach which does not require smooth PDE solutions and which makes the diagonal structure explicit.

\smallskip
\noindent\textit{$(ii)$ A transparent diagonal calculus via the It\^o--Kunita--Wentzell formula.}
We show that the It\^o--Kunita--Wentzell formula provides the correct probabilistic mechanism behind the diagonal terms that appear in equilibrium conditions. This clarifies and complements the extended HJB viewpoint of \cite{bjork2010general,bjork2017time,bjork2021time}, and it connects the Markovian state-dependent setting to the general non-Markovian equilibrium DPP of \cite{hernandez2023me}.

\smallskip
\noindent\textit{$(iii)$ Reduction of time dependence to state dependence in the equilibrium setting.}
We make explicit how time-dependent mechanisms such as non-exponential discounting can be embedded into the state-dependent framework via state augmentation. We then show how the corresponding equilibrium BSDE systems arise as a degenerate case of our general theory. To our knowledge, this ``time as state'' reduction has not previously been pointed out and exploited in a dynamic programming framework for sophisticated equilibrium controls.

\smallskip
\noindent\textit{$(iv)$ A tractable illustration: a time-inconsistent linear--quadratic regulator.}
We apply the general results to a time-inconsistent linear--quadratic regulator, where we obtain existence and characterisation results in a concrete class and provide numerical experiments comparing equilibrium, naive, and precommitted controls.

\medskip
The theory developed here fits naturally within the growing probabilistic approach to time-inconsistent control. On the one hand, it complements the general non-Markovian equilibrium theory of \cite{hernandez2023me} by providing an explicit Markovian specialisation adapted to state-dependent preference parameters, and by connecting it to the extended HJB paradigm through a concrete BSDE system. On the other hand, it provides a rigorous dynamic programming underpinning for Markovian state-dependent models that have previously been handled mainly through smooth verification arguments.

\medskip
Time-inconsistent preferences also arise in other domains, including contracting problems with sophisticated agents, where the failure of commitment interacts with moral hazard. We refer to \cite{hernandez2024time} for recent developments in that direction and note that, while our focus is on Markovian diffusion control, the present results strengthen the conceptual bridge between Markovian state-dependent models and the general non-Markovian probabilistic theory.

\medskip
The rest of the paper is organised as follows. \Cref{sec:pres} introduces the time-inconsistent control problem and the equilibrium concept.
\Cref{sec:mainresults} states the main results, including the equilibrium DPP and the BSDE characterisation.
\Cref{sec:LQR} studies the linear--quadratic regulator example and compares equilibrium and naive controls.
Finally, \Cref{sec:time-dependency} discusses the reduction of time dependence to state dependence and its implications for non-exponential discounting.

\medskip
{\footnotesize{\bf Notations:} Throughout this paper we take the convention $\infty-\infty\coloneqq -\infty$, and we fix a time horizon $T>0$.
$\R_\smallertext{+}$ and $\R_\smallertext{+}^\star$ denote the sets of non-negative and positive real numbers, respectively.
Given $(E,\|\cdot \|)$ a Banach space, a positive integer $p$, and a non-negative integer $q$, $\Cc^{p}_{q}(E)$ (resp. $\Cc^{p}_{q,b}(E)$) will denote the space of functions from $E$ to $\R^{p}$ which are at least $q$ times continuously differentiable (resp. and bounded with bounded derivatives).
Whenever $E=[0,T]$ (resp. $q=0$ or $b$ is not specified), we suppress the dependence on $E$ (resp. on $q$ or $b$), \emph{e.g.} $\Cc^p$ denotes the space of continuous functions from $[0,T]$ to $\R^{p}$. For any $(x,y)\in\Cc_k\times\Cc_k$, we write $\|x-y\|_\smallertext{\infty}\coloneqq \sup_{t\in[0,T]}\|x(t)-y(t)\|$. {For any dimension $k \in \N^\star$ and radius $R > 0$, we denote by $\bar{B}_\smallertext{R}$ the closed ball of radius $R$ centred at the origin in $\R^k$. That is
\[
    \bar{B}_\smallertext{R} \coloneqq \big\{ y \in \R^k : \|y\| \le R \big\}.
\]}\medskip
Given $(x,\tilde x)\in \Cc^p\times\Cc^p$ and $t\in[0,T]$, we define their concatenation $x\otimes_t  \tilde x\in\Cc^p$ by 
\[
(x\otimes_t  \tilde x)(r)\coloneqq  x(r)\mathbf{1}_{\{r\leq t\}}+(x(t)+\tilde x(r)-\tilde x(t))\mathbf{1}_{\{r\geq t\}},\; r\in[0,T].
\]
For $\varphi \in \Cc^{p}_q(E)$ with $q\geq 2$, $\partial_{xx}^2 \varphi$ will denote its Hessian matrix.
For $(u,v) \in \R^p\times\R^p$, $u\cdot v$ will denote their usual inner product, and $\|u\|$ the corresponding norm.
For positive integers $m$ and $n$, we denote by $\Mc_{m,n}(\R)$ the space of $m\times n$ matrices with real entries, and we simplify notations by setting $\Mc_n(\R)\coloneqq \Mc_{n,n}(\R)$. $\Tr [M]$ denotes the trace of a matrix $M\in  \Mc_{n}(\R)$.

\medskip
For $(\Omega, \Gc)$ a measurable space, $\Prob(\Omega)$ denotes the collection of all probability measures on $(\Omega, \Gc)$.
For $\P\in \Prob(\Omega)$ and a filtration $\G$, $\G^\smallertext{\P}\coloneqq (\Gc_t^\smallertext{\P})_{t\in[0,T]},$ denotes the $\P$-completion of $\G$.
We recall that for any $t\in [0,T]$, $\Gc^\smallertext{\P}_t\coloneqq \Gc_t\vee \sigma(\Nc^\smallertext{\P})$, where 
\[
\Nc^\smallertext{\P}\coloneqq \{N\subseteq \Omega: \exists B \in \Gc,\; N \subseteq B, \;\text{and}\; \P[B]=0\}.
\]
$\G^\smallertext{\P}_\smallertext{+}$ denotes the right limit of $\G^\smallertext{\P}$, \emph{i.e.} $\Gc_{t\smallertext{+}}^\smallertext{\P}\coloneqq \bigcap_{\eps>0} \Gc_{t\smallertext{+}\eps}^\smallertext{\P}$, $t\in[0,T)$, and $\Gc_{T\smallertext{+}}^\smallertext{\P}\coloneqq \Gc_T^\smallertext{\P}$.

\medskip
For $(s,t)\in [0,T]^2$, with $s\leq t$, $\Tc_{s,t}(\F)$ denotes the collection of $[s,t]$-valued $\F$--stopping times.

}

\section{Time-inconsistent stochastic control}\label{sec:pres}

We fix two positive integers $n$ and $d$, which represent respectively the dimension of the process controlled by the agent, and the dimension of the Brownian motion driving this controlled process. We fix a time horizon $T>0$, and consider the canonical space $\Omega\coloneqq \Cc([0,T],\R^{n})$, with canonical process $X$, and whose generic elements we denote $\omega$.

\medskip

We let $\Fc$ be the Borel $\sigma$-algebra on $\Omega$ (for the topology of uniform convergence), and we denote by $\F^\smallertext{X}\coloneqq (\Fc^\smallertext{X}_t)_{t\in[0,T]}$ the natural filtration of $X$. We let $A$ be a closed subset of $\R^k$ for some positive integer $k$, where the controls will take values. 

 \begin{remark} Note that we do not assume that $A$ is compact. This will allow the case treated in {\rm\Cref{sec:LQR}} to be included in our theory. However, we will later assume that the Hamiltonian in  \eqref{eq:Hamiltonian_Def} is attained, either due to compactness of $A$ or coercivity of the coefficients.\end{remark}

\begin{remark}
    We restrict our attention to Euclidean action spaces primarily to facilitate the heuristic derivations in {\rm\Cref{sec:mainresults}}, which rely on differentiation with respect to the control variable. However, the rigorous results of this paper $($specifically the necessity and verification theorems$)$ rely solely on measurable selection arguments. Consequently, our theory extends straightforwardly to the case where $A$ is a closed subset of an arbitrary Polish space.
\end{remark}

\subsection{Probabilistic setting}\label{sec:setting}

We will follow a similar setting to the one in \citeauthor*{hernandez2023me} \cite{hernandez2023me} restricting to a Markovian framework, and working exclusively under the weak formulation. We fix a bounded Borel measurable map $\sigma:[0,T]\times\R^n \longrightarrow \R^{n\times d}$, an initial condition $x_0\in\R^n$, and assume that there is a unique solution, denoted by $\P$, to the martingale problem for which $X$ is an $(\F^\smallertext{X},\P)$--local martingale, such that $X_0=x_0$ with $\P$-probability $1$, and $\mathrm{d}[ X]_t=\sigma(t, X_{t})\sigma^\top(t, X_{ t})\mathrm{d}t$, $\P$--a.s.. Enlarging the original probability space if necessary (see \citet*[Theorem 4.5.2]{stroock1997multidimensional}), we can find an $\R^{d}$-valued Brownian motion $W$ such that
\[
X_t=x_0+\int_0^t\sigma(r, X_{r})\mathrm{d}W_r,\; t\in[0,T].
\]

We now let $\F\coloneqq (\Fc_t)_{t\in[0,T]}$ be the $\P$--augmentation of $\F^\smallertext{X}$. We recall that uniqueness of the solution to the martingale problem implies that the predictable martingale representation property holds for $(\F,\P)$-martingales, which can be represented as stochastic integrals with respect to $X$ (see \citet*[Theorem III.4.29]{jacod2003limit}). We also mention that the right-continuity of $\F$ guarantees that $(\F,\P )$ satisfies the Blumenthal zero--one law and, in particular, all $\Fc_0$-measurable random variables are deterministic.

\medskip
We can then introduce our drift functional $b:[0,T]\times\Omega \times A\longrightarrow \R^d$, which is assumed to be Borel-measurable with respect to all its arguments. Let us recall that for any $A$-valued, $\F$-predictable process $\alpha$ such that
\begin{equation}\label{eq:integalpha}
\E^{\P}\bigg[\exp\bigg(\int_0^T b(r, X_{ r},\al_r)\cdot\mathrm{d}W_r-\frac12\int_0^T\big\|b(r, X_{ r},\al_r)\big\|^2\mathrm{d}r\bigg)\bigg]<\infty,
\end{equation} 
we can define the probability measure $\P^{\alpha}$ on $(\Omega,\Fc_T)$, whose density with respect to $\P$ is given by
\[
\frac{\mathrm{d}\P^{\alpha}}{\mathrm{d}\P}\coloneqq \exp\bigg(\int_0^T b(r, X_{ r},\al_r)\cdot\mathrm{d}W_r-\frac12\int_0^T\big\|b(r, X_{ r},\al_r)\big\|^2\d r\bigg).
\]
Moreover, by Girsanov's theorem, the process $W^{\alpha}\coloneqq W-\int_0^{\cdot} b(r, X_{ r},\al_r)\mathrm{d}r$ is an $\R^d$-valued, $(\F,\P^{\alpha})$--Brownian motion and we have
\[
    X_t=x_0+\int_0^t\sigma(r, X_{ r})b(r, X_{ r},\al_r)\d r+\int_0^t\sigma(r, X_{ r})\mathrm{d}W^{\alpha}_r,\; t\in[0,T],\; \P\text{\rm --a.s.}
\]

We define $\mathcal{A}$ to be the set of all continuous processes such that condition \eqref{eq:integalpha} holds. Let us emphasise that we are working under the so-called weak formulation of the problem. This means that the state process $X$ is fixed and, in contrast to the typical strong formulation, the Brownian motion, and the probability measure are not fixed. Indeed, the choice of $\alpha$ corresponds to the choice of probability measure $\P^{\alpha}$ and thus impacts the distribution of process $X$.

\medskip
Let us now recall the celebrated result on the existence of a well-behaved $\omega$-by-$\omega$ versions of the conditional expectation. We also introduce the concatenation of a measure and a stochastic kernel. Recall $\Omega$ is a Polish space and $\Fc$ is a countably generated $\sigma$-algebra. For $\P\in \Prob(\Omega)$ and $\tau \in \Tc_{0,T}(\F)$, $\Fc_\tau$ is also countably generated, so there exists an associated regular conditional probability distribution (r.c.p.d. for short) $(\P_x^\tau)_{x \in \Omega}$, see \citeauthor{stroock1997multidimensional} \cite[Theorem 1.3.4]{stroock1997multidimensional}, satisfying
\begin{enumerate}[label=$(\roman*)$, ref=.$(\roman*)$,wide, labelindent=0pt]
\item for every $x \in \Omega$, $\P^\tau_x$ is a probability measure on $(\Omega,\Fc)$;
\item for every $E\in \Fc$, the mapping $x\longmapsto \P^\tau_x[E]$ is $\Fc_\tau$-measurable;
\item the family $(\P_x^\tau)_{x\in \Omega}$ is a version of the conditional probability measure of $\P$ given $\Fc_\tau$, that is to say that for every $\P$-integrable, $\Fc$-measurable random variable $\xi$, we have $\E^\P[\xi|\Fc_\tau](x)=\E^{\P^\tau_x}[\xi]$, for $\P\text{--}\ae\ x \in \Omega$;
\item for every $x \in \Omega$, $\P^\tau_x [\Omega^x_\tau]=1$, where $\Omega^x_\tau\coloneqq \{ x^\prime \in \Omega : x^\prime(r)=x(r),\; 0\leq r\leq \tau(x)\}$.
\end{enumerate}

Moreover, for $\P\in \Prob(\Omega)$ and an $\Fc_\tau$-measurable stochastic kernel $(\Q_x^\tau)_{x\in\Omega}$ such that $\Q^\tau_x [\Omega^x_\tau]=1$, for every $x \in \Omega$, the concatenated probability measure is defined by
\begin{align}\label{Def:ConcMeasure}
\P \otimes_\tau \Q_\cdot [A]\coloneqq \int_\Omega \P(\d x) \int_\Omega \1_A(x \otimes_{\tau(x)}\tilde x) \Q_x(\d \tilde x),\; \forall A \in \Fc.
\end{align}
The following result, see \cite[Theorem 6.1.2]{stroock1997multidimensional}, gives a rigorous characterisation of the concatenation procedure.
\begin{theorem}[Concatenated measure]\label{Thm:Concatenated:M}
Consider a stochastic kernel $(\Q_\omega)_{\omega\in\Omega}$,  and let $\tau\in\Tc_{0,T}(\F)$. Suppose the map $\omega \longmapsto \Q_\omega$ is $\Fc_\tau$-measurable and $\Q_\omega[\Omega_\tau^\omega]=1$ for all $\omega\in \Omega$. Given $\P \in \Prob(\Omega)$, there is a unique probability measure $\P\otimes_{\tau(\cdot)} \Q_\cdot$ on $(\Omega,\Fc)$ such that $\P\otimes_{\tau(\cdot)} \Q_\cdot$ equals $\P$ on $(\Omega,\Fc_\tau)$ and $(\delta_{\omega} \otimes_{\tau(\omega)} \Q_\omega)_{\omega \in \Omega}$ is an {\rm r.c.p.d}. of $\P\otimes_{\tau(\cdot)}\Q_\cdot | \Fc_\tau$. For some $t\in[0,T]$, suppose that $\tau \geq t$, that $M:[t,T]\times \Omega \longrightarrow \R$ is a right-continuous, $\F$--progressively measurable function after $t$, such that $M_t$ is $\P\otimes_{\tau(\cdot)} \Q_\cdot$-integrable, that for all $r\in[t,T]$, $(M_{r\wedge \tau})_{r\in[t,T]}$ is an $(\F,\P)$-martingale, and that $(M_r- M_{r\wedge\tau(\omega)})_{r\in[ t,T]}$ is an $(\F,\Q_\omega)$-martingale, for all $\omega \in \Omega$. Then $(M_r)_{r\in [t,T]}$ is an $(\F,\P\otimes_{\tau(\cdot)}\Q_\cdot)$-martingale.
\end{theorem} 

In particular, for an $\Fc$-measurable function $\xi$, $\E^{\P\otimes_\smalltext{\tau} \P^\smalltext{\tau}_\smalltext{\cdot}}[\xi]=\E^\P[\E^\P[\xi|\Fc_\tau]]=\E^\P[\xi]$. This is the classical tower property. Additionally, the reverse implication in the last statement in \Cref{Thm:Concatenated:M} holds by \cite[Theorem 1.2.10]{stroock1997multidimensional}.

\medskip
In particular, the exposition above means that we can ensure the existence of probability measures indexed by $(t, x, \alpha) \in [0,T] \times \R^{n} \times \mathcal{A}$ under which the state process satisfies, for $s\in [t,T]$

\[
X_s=x+\int_t^s\sigma(r, X_{ r})b(r, X_{ r},\al_r)\d r+\int_t^s\sigma(r, X_r)\mathrm{d}W^{\alpha}_r,\; t\in[0,T],\; \P^{t, x, \alpha}\text{\rm--a.s.},
\]
where $W^{\alpha}$ is a Brownian motion with respect to $\P^{t, x, \alpha} \coloneqq  (\P^{\alpha})^t_x.$ 

\subsection{Target functional}\label{Section:targetfunctional}

Let us introduce the running and terminal payoff functionals
\begin{align}\label{eq:target_functional_weak}
    J(t, x, \alpha) \coloneqq  \E^{\P^{\smalltext{t}, \smalltext{x}, \smalltext{\alpha}}}\bigg[\int_{t}^T f(s,    x, X_s, \al_s)\diff s + \xi(x, X_T)\bigg],
\end{align}
where $f:[0,T]\times \R^n \times \R^n \times A \longrightarrow \R$ and $\xi:\R^n\times\R^n \longrightarrow \R$ are Borel-measurable functions. We will refer to $f$ as the running payoff function and $\xi$ as the terminal payoff function.

\medskip
 We will sometimes refer to a more generic payoff functional of the form
\begin{align*}
 J(t, x, y, \alpha) \coloneqq  \E^{\mathbb{P}^{\smalltext{t}, \smalltext{x}, \smalltext{\alpha}}}\bigg[\int_{t}^T f(s,    y, X_s, \al_s)\mathrm{d}s + \xi(y, X_T)\bigg], \; (t,x, \alpha)\in [0,T]\times \R^n \times  \Ac.
\end{align*}

Note that we have that $ J(t, x, x, \alpha) =  J(t, x, \alpha)$, justifying our nomenclature. As introduced earlier, we remark that the appearance of $x$ in both functions in the reward functional creates the time-inconsistency. The goal of the controller will be, roughly speaking, to choose $\alpha$ to maximise \eqref{eq:target_functional_weak}. However, since their preferences change over time, it is not clear what we mean mathematically by this. In the next subsection, we introduce the precise notion of controls that we will be interested in.

\subsection{Game formulation}\label{Section:gameformulation}

We recall that a strategy profile is \textit{sub-game perfect} if it prescribes a Nash equilibrium in any sub-game.  In our framework, every player together with a past trajectory define a new sub-game. This motivates the idea behind the definition of an equilibrium model, see among others \citeauthor{bjork2010general} \cite{bjork2010general}, \citeauthor{ekeland2006being} \cite{ekeland2006being} and \citeauthor{strotz1955myopia} \cite{strotz1955myopia}. The intuition behind this consideration is that at each point in time a different player stands (which can be thought of different versions of one-self), and we intuitively try to achieve a sub-game perfect strategy.

\medskip
Let $\alpha \in \Ac$ be an action, $(t,x) \in [0,T]\times \R^n$ an arbitrary initial condition, and $\ell\in (0, T-t]$. We recall that $\alpha\otimes_{\tau}\alpha^\star\coloneqq \alpha\1_{[ t, \tau)}+
\alpha^\star\1_{[ \tau ,T]}$.

\begin{definition}[\textit{Equilibrium control}]\label{Def:equilibrium}
Let $\alpha^\star \in \mathcal{A}$ be an admissible control. We say that $\alpha^\star$ is an equilibrium control, if for any $\varepsilon >0$, we have that $\ell_\varepsilon >0$, where
\[
\ell_\varepsilon: = \inf\big\{\ell>0 : \exists\alpha \in \Ac,\; \mathbb{P}[\{ \exists t\in [0,T], J(t, X_t, \alpha^\star) < J(t, X_t, \alpha\otimes_{\ell}\alpha^\star) - \varepsilon\ell\}]> 0\big\}.
\]

In this case, we write $\alpha^\star \in \Ec$.
\end{definition}

\begin{remark}
    We can show that one can recover the essence of the classical definition in {\rm\cite{bjork2014theory}} in the following sense: assume that $\alpha^\star$ is an equilibrium control as in the previous definition, and let $\varepsilon>0$. Then, there exists some $\ell_\varepsilon>0$ and a set $\Tilde{\Omega}$ with $\mathbb{P}[\Tilde{\Omega}] = 1$ with
    \[
    J(t, X_t, \alpha^\star) - J(t, X_t, \alpha\otimes_{\ell}\alpha^\star) \geq -\varepsilon\ell, \;\forall (\ell, X_t, \alpha) \in (0,  \ell_\varepsilon)\times \Tilde{\Omega}\times A.
    \]

Now, as $\varepsilon$ was arbitrary, we can take a sequence $\varepsilon_n = 1/n$, $n\in\N^\star$, with their corresponding sets $\tilde{\Omega}_n$, and on $\Omega^\star \coloneqq \bigcap_{n\in\N^\smalltext{\star}} \tilde{\Omega}_n$ we have that
\[
\liminf_{\ell \downarrow 0}\frac{J(t, X_t, \alpha^\star) - J(t, X_t, \alpha\otimes_{\ell}\alpha^\star)}{\ell} \geq 0.
\]

\end{remark}

In the rest of the document we fix some $(t, x) \in [0,T] \times \R^n$ and study the problem 
\begin{align*}\label{P}\tag{P}
v(t,x)\coloneqq J(t,x,\alpha^\star),\; (t,x)\in[0,T]\times \R^n,\; \alpha^\star \in \Ec.
\end{align*}

Thanks to the weak uniqueness assumption, $v$ is well-defined for all $(t,x)\in [0,T]\times  \R^n$ and Borel-measurable. 

\begin{remark}\label{Remark:valuefunction}
\eqref{P} is fundamentally different from the problem of maximising $\Ac \ni \alpha \longmapsto J(t,x,\alpha)$. In \eqref{P}, one finds $\alpha^\star\in \Ac$ first and then defines the value function. This contrasts with the classical formulation of optimal control problems. Second, the previous maximisation will find player $t$'s so-called pre-committed strategy.
\end{remark}

\subsection{Functional spaces}\label{sec:functional_spaces}

In this section, we introduce the spaces of processes that we will be using throughout this paper. We first recall the standard spaces of square-integrable processes
\begin{itemize}
    \item $\mathbb{S}^2(\R^n,\F,\P)$: the space of $\F$--progressively measurable, càdlàg processes $Y$ taking values in $\R^n$ such that
    \[ \|Y\|_{\mathbb{S}^\smalltext{2}(\R^\smalltext{n},\F,\P)}^2 \coloneqq  \E^\P \bigg[ \sup_{t \in [0, T]} \|Y_t\|^2 \bigg] < \infty. \]
    \item $\mathbb{H}^2(\R^d,\F,\P)$: the space of $\F$-predictable processes $Z$ taking values in $\R^d$ such that
    \[ \|Z\|_{\mathbb{H}^\smalltext{2}(\R^\smalltext{d},\F,\P)}^2 \coloneqq  \E^\P \bigg[ \int_0^T \|Z_t\|^2 \d t \bigg] < \infty. \]
\end{itemize}

For the derivative processes, which depend on the parameter $y \in \R^n$, we require well-posedness uniform on compact sets. Toward this purpose, we introduce the spaces of locally square-integrable random fields. 
\begin{definition}[Locally uniform random fields] \label{def:locallyuniformrandomfields}
    Let $\mathcal{U} = (\mathcal{U}^y)_{y \in \R^\smalltext{n}}$ and $\mathcal{V} = (\mathcal{V}^y)_{y \in \R^\smalltext{n}}$ be two families of stochastic processes indexed by $y$.
    \begin{itemize}
        \item We say $\mathcal{U} \in \mathfrak{S}^2_{\mathrm{loc}}(\R^n,\F,\P)$ if the map $y \longmapsto \mathcal{U}^y$ is continuous from $\R^n$ to $\mathbb{S}^2(\R^n,\F,\P)$, and bounded on compact sets. That is, for any compact set $K \subset \R^n$
        \[ \sup_{y \in K} \|\mathcal{U}^y\|_{\mathbb{S}^\smalltext{2}(\R^\smalltext{n},\F,\P)} < \infty. \]
        
        \item We say $\mathcal{V} \in \mathfrak{H}^2_{\mathrm{loc}}(\R^d,\F,\P)$ if the map $y \longmapsto \mathcal{V}^y$ is continuous from $\R^d$ to $\mathbb{H}^2(\R^d,\F,\P)$, and for any compact set $K \subset \R^d$
        \[ \sup_{y \in K} \|\mathcal{V}^y\|_{\mathbb{H}^\smalltext{2}(\R^\smalltext{d},\F,\P)} < \infty. \]
    \end{itemize}
    The spaces are equipped with the topology induced by the family of semi-norms $\{\sup_{y \in \bar{B}_\smalltext{R}} \|\cdot\|\}_{R > 0}$.
  
\end{definition}

\subsubsection{Auxiliary weighted functional spaces and norms}

To carry out the proof of well-posedness, we introduce the specific polynomial weight function $\rho: \R^n \longrightarrow \R_+$ defined by 
\[
    \rho(y) \coloneqq (1 + \|y\|^2)^{-k},
\]
where $k \ge 1$ is a fixed integer chosen sufficiently large relative to the growth rate $m$ appearing in \Cref{ass:poly_growth}. Specifically, we require $2k \ge m$, as we will see later.

\medskip

\begin{remark}[General growth conditions]
    The choice of the weight function $\rho$ has been made for presentation purposes and to directly encompass the {\rm LQR} example that we will present in {\rm \Cref{sec:LQR}}. See also {\rm \Cref{rem:weight}}.
\end{remark}

For any $\beta > 0$ and dimension $d\in \N^\star$, we define the following Banach spaces for processes on $[0, T]$.
    \begin{itemize}
        \item $\mathbb{H}^2_\beta(\R^d,\F,\P)$ is the space of $\R^d$-valued, $\F$-predictable processes $Z$ such that 
        \[
        \|Z\|^2_{\mathbb{H}^\smalltext{2}_\smalltext{\beta}(\R^\smalltext{d},\F,\P)} \coloneqq  \E^\P \bigg[ \int_0^T \mathrm{e}^{\beta t} \|Z_t\|^2 \d t \bigg] < \infty.
        \]
        \item $\mathbb{S}^2_\beta(\R^d,\F,\P)$ is the space of $\R^d$-valued, $\F$-optional càdlàg processes $Y$ such that 
        \[
        \|Y\|^2_{\mathbb{S}^\smalltext{2}_\smalltext{\beta}(\R^\smalltext{d},\F,\P)} \coloneqq  \E^\P \bigg[ \sup_{t \in [0, T]} \mathrm{e}^{\beta t} \|Y_t\|^2 \bigg] < \infty.
        \]
    \end{itemize}

    Note that the norms are equivalent for all values of $\beta$ since $[0, T]$ is compact. Let $U = (U_t^y)_{y \in \R^\smalltext{n}}$ be a random field where, for each $y$, $U^y$ is a process. We define the weighted spaces:
    \begin{itemize}
        \item $\mathbb{S}^{2,2}_{\beta, \rho}(\R^d,\F,\P)$ is the space of random fields $U$ such that $U^y \in \mathbb{S}^2_\beta(\R^{d},\F,\P)$ for all $y$, the map $y \longmapsto U^y$ is continuous from $\R^n$ to $\mathbb{S}^2_\beta(\R^d,\F,\P)$, and
        \[ \|U\|^2_{\mathbb{S}^{\smalltext{2}\smalltext{,}\smalltext{2}}_{\smalltext{\beta}\smalltext{,} \smalltext{\rho}}(\R^\smalltext{d},\F,\P)} \coloneqq  \sup_{y \in \R^\smalltext{n}} \Big\{ \rho(y)  \|U^y\|^2_{\mathbb{S}^\smalltext{2}_\smalltext{\beta}(\R^\smalltext{d},\F,\P)} \Big\} < \infty.
        \]
        \item $\mathbb{H}^{2,2}_{\beta, \rho}(\R^d,\F,\P)$ is the space of random fields $V$ such that $V^y \in \mathbb{H}^2_\beta(\R^{d},\F,\P)$ for all $y$, the map $y \longmapsto V^y$ is continuous from $\R^n$ to $\mathbb{H}^2_\beta(\R^d)$, and
        \[ \|V\|^2_{\mathbb{H}^{\smalltext{2}\smalltext{,}\smalltext{2}}_{\smalltext{\beta}\smalltext{,} \smalltext{\rho}}(\R^\smalltext{d},\F,\P)} \coloneqq  \sup_{y \in \R^\smalltext{n}} \Big\{ \rho(y)  \|V^y\|^2_{\mathbb{H}^\smalltext{2}_\smalltext{\beta}(\R^\smalltext{d},\F,\P)} \Big\} < \infty.
        \]
    \end{itemize}

We define the global product space $\mathcal{K}_{\beta}^{n,d}$ for the tuple $(Y, Z, \partial Y, \partial Z, \partial\partial Y, \partial\partial Z)$, which will solve the BSDE system \eqref{BSDE Weak Formulation}, to be introduced in \Cref{sec:mainresults}:
\medskip

\begin{equation}\label{eq:K_beta_def}
    \mathcal{K}_{\beta}^{n,d}(\F,\P) \coloneqq 
    \underbrace{ \mathbb{S}^2_\beta(\R,\F,\P) \times \mathbb{H}^2_\beta(\R^d,\F,\P) }_{\text{Value process } (Y, Z)} 
    \times 
    \underbrace{\mathbb{S}^{2,2}_{\beta, \rho}(\R^n,\F,\P) \times \mathbb{H}^{2,2}_{\beta, \rho}(\R^{n \times d},\F,\P) }_{\text{Gradient process } (\partial Y, \partial Z)} 
    \times 
    \underbrace{ \mathbb{S}^{2,2}_{\beta, \rho}(\R^{n \times n},\F,\P) \times \mathbb{H}^{2,2}_{\beta, \rho}(\R^{n \times n \times d},\F,\P)}_{\text{Hessian process } (\partial\partial Y, \partial\partial Z)}.
\end{equation}

\begin{proposition}[Banach structure]\label{prop:BanachSpace}
    The space $\mathcal{K}_{\beta}^{n,d}(\F,\P)$ is a Banach space.
\end{proposition}

\begin{proof}
    The spaces $\mathbb{S}^2_\beta(\R,\F,\P)$ and $\mathbb{H}^2_\beta(\R^d,\F,\P)$ are standard spaces of square-integrable processes and are well-known to be Banach spaces (actually Hilbert spaces). The weighted spaces $\mathbb{S}^{2,2}_{\beta, \rho}(\R^n,\F,\P)$ (resp. $\mathbb{S}^{2,2}_{\beta, \rho}(\R^{n\times n},\F,\P)$) and $\mathbb{H}^{2,2}_{\beta, \rho}(\R^{n\times d},\F,\P)$ (resp. $\mathbb{H}^{2,2}_{\beta, \rho}(\R^{n\times n\times d},\F,\P)$) are defined as spaces of continuous functions $y \longmapsto U^y$ from $\R^n$ (resp. $\R^{n\times n}$) into the Banach spaces $\mathbb{S}^2_\beta(\R^n,\F,\P)$ (resp.  $\mathbb{S}^2_\beta(\R^{n\times n},\F,\P)$) and $\mathbb{H}^2_\beta(\R^{n\times d},\F,\P)$ (resp. $\mathbb{H}^2_\beta(\R^{n\times n\times d},\F,\P)$), equipped with a supremum norm weighted by $\rho(y)^{1/2}$. Since $\rho$ is strictly positive, these are weighted spaces of bounded continuous functions taking values in a Banach space. By standard functional analysis results, the space of bounded continuous functions from a topological space into a Banach space is itself a Banach space under the supremum norm. Since $\mathcal{K}_\beta(\F,\P)$ is a finite Cartesian product of Banach spaces, it is itself a Banach space.
\end{proof}

\medskip

To further motivate these spaces at this point, let us present the following lemma, that asserts that they hold \Cref{def:locallyuniformrandomfields}.
\begin{lemma}[Embedding of weighted spaces]\label{lemma:embedding_spaces}
    Let $\beta > 0$ and $k \ge 0$. Let $\mathcal{U} = (\mathcal{U}^y)_{y \in \R^\smalltext{n}}$ be a random field belonging to the weighted space $\mathbb{S}^{2,2}_{\beta, \rho}(\R^{n},\F,\P)$. Then, $\mathcal{U}$ belongs to the locally uniform space $\mathfrak{S}^2_{\rm loc}(\R^n,\F,\P)$. Similarly, $\mathbb{H}^{2,2}_{\beta, \rho}(\R^{n},\F,\P) \subset \mathfrak{H}^2_{\rm loc}(\R^n,\F,\P)$.
\end{lemma}

\begin{proof}
    Let $\mathcal{U} \in \mathbb{S}^{2,2}_{\beta,\rho}(\R^{n},\F,\P)$. By definition, there exists a constant $C_{\mathcal{U}} < \infty$ such that
    \begin{equation}\label{eq:weighted_bound}
        \sup_{z \in \R^\smalltext{n}} \frac{\|\mathcal{U}^z\|^2_{\mathbb{S}^\smalltext{2}_\smalltext{\beta}(\R^\smalltext{n},\F,\P)}}{(1+\|z\|^2)^k} = C_{\mathcal{U}}.
    \end{equation}
    We must show that for any compact set $K \subset \R^n$, the standard $\mathbb{S}^2(\R^{n},\F,\P)$ norm is uniformly bounded. Let $K$ be an arbitrary compact subset of $\R^n$. Since $K$ is bounded, there exists a radius $R > 0$ such that $\|y\| \le R$ for all $y \in K$. First, we relate the $\beta$-weighted time norm to the standard norm. Since $t \in [0, T]$, we have $\mathrm{e}^{\beta t} \ge 1$. Thus, for any process $Y$
    \[ \|Y\|^2_{\mathbb{S}^\smalltext{2}(\R^\smalltext{n},\F,\P)} = \E^\P\bigg[\sup_{t \in [0, T]} \|Y_t\|^2\bigg] \le \E^\P\bigg[\sup_{t \in [0, T]} \mathrm{e}^{\beta t} \|Y_t\|^2\bigg] = \|Y\|^2_{\mathbb{S}^\smalltext{2}_\smalltext{\beta}(\R^\smalltext{n},\F,\P)}. \]
    
    Next, we handle the parameter weight. For any $y \in K$
    \begin{align*}
        \|\mathcal{U}^y\|^2_{\mathbb{S}^\smalltext{2}(\R^\smalltext{n},\F,\P)} \le \|\mathcal{U}^y\|^2_{\mathbb{S}^\smalltext{2}_\smalltext{\beta}(\R^\smalltext{n},\F,\P)} = (1+\|y\|^2)^k  \frac{\|\mathcal{U}^y\|^2_{\mathbb{S}^2_\beta(\R^\smalltext{n},\F,\P)}}{(1+\|y\|^2)^k} \le (1+R^2)^k  \sup_{z \in \R^\smalltext{n}} \bigg\{ \frac{\|\mathcal{U}^z\|^2_{\mathbb{S}^\smalltext{2}_\smalltext{\beta}(\R^\smalltext{n},\F,\P)}}{(1+\|z\|^2)^k} \bigg\} = (1+R^2)^k C_{\mathcal{U}}.
    \end{align*}
    The right-hand side is a finite constant independent of $y \in K$. Thus, $\sup_{y \in K} \|\mathcal{U}^y\|_{\mathbb{S}^\smalltext{2}(\R^\smalltext{n},\F,\P)} < \infty$. Continuity of $y \longmapsto \mathcal{U}^y$ in the standard norm follows immediately from the continuity in the weighted norm, as the weight function $(1+\|y\|^2)^{-k}$ is smooth and bounded away from zero on compacts. Therefore, $\mathcal{U} \in \mathfrak{S}^2_{\rm loc}(\R^n,\F,\P)$.
    
\medskip
The remaining result is proved in an analogous way.
\end{proof}

\section{Main results}\label{sec:mainresults}

In this section, we present the core theoretical contributions of this paper. We characterise the equilibrium strategies for state-dependent time-inconsistent control problems through a probabilistic approach. The roadmap will be as follows:
\begin{enumerate}
    \item[$(i)$] we first provide an informal derivation of the system of backward stochastic differential equations (BSDEs) that characterises the equilibrium, building intuition from the extended HJB equation;
    \item[$(ii)$] we then establish an extended dynamic programming principle (DPP), which generalises the Bellman principle by accounting for the changing preferences of the agent;
    \item[$(iii)$] we derive the BSDE system (as a necessary condition for equilibria) and prove a verification theorem (the sufficiency counterpart);
    \item[$(iv)$] we prove the well-posedness (existence and uniqueness) of this system.
\end{enumerate}

\subsection{An informal derivation of the BSDE system}

The purpose of this section is to informally justify the BSDE system that will be at the heart of this work. This derivation will be based on the extended HJB equation \cite[Definition 15.4]{bjork2021time}, and thus we will remain in the Markovian, feedback control (meaning we look for an equilibrium control $\alpha^\star$ that is a deterministic feedback function of the time and state, \emph{i.e.}, $\alpha^\star_t = \alpha^\star(t, X_t)$ for some Borel-measurable map $\alpha^\star$), and we will use the weak formulation all along. 

\medskip
For simplicity in this derivation, let $n=d=1$ and let the dynamics of the state process $(X_t)_{t\geq0}$ under $\P^{\alpha}$ be given by
\begin{equation}\label{strong formulation state process chapter 6}
    X_t = x_0 + \int_0^t \sigma(r, X_{ r})b(r, X_{ r}, \al_r)\mathrm{d}r + \int_0^t \sigma(r, X_{r})\mathrm{d}W^{\alpha}_r,\; t\in[0,T].
\end{equation}

Once again, the payoff functional is given by
\[
J(t, x, \alpha) \coloneqq  \E^{\P^{\smalltext{t},\smalltext{x},\smalltext{\alpha}}}\bigg[\int_{t}^T f(s, x, X_s, \al_s)\d s + \xi(x, X_T)\bigg],\; (t,x)\in[0,T]\times\R.
\]

For a fixed control $\alpha^\star$, we let $V(t, x) \coloneqq J(t, x, \alpha^\star)$ denote the equilibrium value function and $\mathcal{J}(t, x, y)$ denote the auxiliary value function with fixed preference parameter $y$, defined as
\[
    \mathcal{J}(t, x, y) \coloneqq \E^{\P^{\smalltext{t},\smalltext{x},\smalltext{\alpha}^\star}}\bigg[\int_{t}^T f(s, y, X_s, \alpha^\star_s)\d s + \xi(y, X_T)\bigg],\; (t,x,y)\in[0,T]\times\R\times\R.
\]

\medskip

According to the theory developed in \citeauthor{bjork2014theory} \cite{bjork2014theory}, the pair $(V, \mathcal{J})$ must satisfy the extended HJB system, which we present now particularised for our case.

\medskip

For any $(t, x) \in [0, T] \times \R^n$ and action $a \in A$, we define the infinitesimal generator $\mathcal{L}_t^a$ acting on smooth functions $\phi \in C^2(\R^n)$ by
\[
    \mathcal{L}_t^a \phi(x) \coloneqq b(t, x, a)  \sigma(t,x) \nabla_x \phi(x) + \frac{1}{2}\Tr\big[\sigma(t, x)\sigma(t, x)^\top \nabla_{xx}^2 \phi(x)\big].
\]

For $(t, x, y) \in [0, T) \times \R \times \R$, the system is
 \begin{equation}\label{eq:extended_HJB_system}
    \begin{cases}
   \displaystyle \partial_t V(t,x) + \sup_{a \in A} \Big\{ f(t, x, x, a) + b(t, x, a)\sigma(t,x)  \big( \partial_x V(t, x) - \partial_y \mathcal{J}(t, x, x) \big) \Big\}  +\frac{1}{2}\sigma^2(t, x)\partial_{xx}^2 V(t, x)\\[1em]
   \displaystyle \quad  - \sigma^2(t, x) \partial^2_{xy} \mathcal{J}(t, x, x) - \frac{1}{2}\sigma^2(t, x) \partial^2_{yy} \mathcal{J}(t, x, x) = 0,\\[1em]
 \displaystyle  \partial_t\mathcal{J}(t, x, y) + \mathcal{L}_t^{\alpha^\smalltext{\star}(t, x)} \mathcal{J}(t, x, y) + f(t, y, x, \alpha^\star(t, x))= 0,\\[0.5em]
 \displaystyle   V(T, x) = \xi(x,x),\; \mathcal{J}(T, x, y) = \xi(y,x).
\end{cases}
\end{equation}
The equilibrium control $\alpha^\star(t, x)$ is defined as the argument attaining the supremum in the first equation. Note that in the second equation, the generator $\mathcal{L}_t^{\alpha^\star(t, x)}$ acts on the variable $x$ with $y$ fixed.

\medskip
Note that the equilibrium control, which maximises the supremum in the first equation, appears in the second equation. Simultaneously, the function $\mathcal{J}(t, x, y)$ is part of the first equation. Hence, the system is very entangled and it is hard to determine its well-posedness using analytical techniques.

\medskip
Within the supremum in \eqref{eq:extended_HJB_system}, the effective gradient acting on the drift is not the standard $\partial_x V$, but the difference $\partial_x V - \partial_y \mathcal{J}$. This specific structure motivates the definition of our Hamiltonian below. The diffusion part includes the standard Hessian $\partial_{xx} V$ corrected by the mixed derivative $\sigma^2 \partial_{xy} \mathcal{J}$ and the parameter Hessian $\frac{1}{2}\sigma^2 \partial_{yy} \mathcal{J}$.

\medskip

To derive the BSDE system, we differentiate the second equation in \eqref{eq:extended_HJB_system} with respect to $y$ to find the dynamics of the derivatives $\mathcal{J}^y(t, x) \coloneqq \mathcal{J}(t, x, y)$. For $(t, x, y) \in [0, T) \times \R \times \R$
 \begin{equation}\label{eq:derivatives_J}
    \begin{cases}
 \displaystyle  \big(\partial_t + \mathcal{L}_t^{\alpha^\smalltext{\star}(t, x)}\big) \partial_y \mathcal{J}^y(t, x) + \partial_y f(t, y, x, \alpha^\star(t, x)) = 0, \\[0.5em]
 \displaystyle  \big(\partial_t + \mathcal{L}_t^{\alpha^\smalltext{\star}(t, x)}\big) \partial_{yy}^2 \mathcal{J}^y(t, x) + \partial_{yy}^2 f(t, y, x, \alpha^\star(t, x)) = 0.
\end{cases}
\end{equation}

We now define the stochastic processes corresponding to these quantities along the equilibrium trajectory $X_t$
 \begin{gather*}
 Y_t = V(t, X_t), \; Z_t = \sigma(t, X_t) \partial_x V(t, X_t),\; t\in[0,T],\\
 \partial Y^y_t = \partial_y \mathcal{J}(t, y, X_t), \; \partial Z^y_t = \sigma(t, X_t) \partial_{xy}^2 \mathcal{J}(t, y, X_t),\; t\in[0,T],\; y\in\R,\\
 \partial\partial Y^y_t = \partial_{yy}^2 \mathcal{J}(t, y, X_t), \; \partial\partial Z^y_t = \sigma(t, X_t) \partial_{xyy}^3 \mathcal{J}(t, y, X_t),\; t\in[0,T], \; y\in\R.
 \end{gather*}

 Applying It\^o's formula to $Y_t$, the drift is given by $(\partial_t + \Lc^{\alpha^\smalltext{\star}(t,\smallertext{X}_{\smalltext{t}})}) V(t,X_t)$. By rearranging the first equation of the extended HJB system, we can express this operator as
\begin{align*}
    \big(\partial_t + \Lc^{\alpha^\smalltext{\star}(t,X_\smalltext{t})}) V (t, X_t) &=  -f\big(t,X_t,X_t,\alpha^\star(t,X_t)\big) + b\big(t,X_t,\alpha^\star(t,X_t)\big) \sigma(t, X_t) \partial_y \mathcal{J}(t,X_t,X_t) +\sigma(t, X_t) \partial_{xy}^2 \mathcal{J}(t,X_t,X_t) \\
    &\quad+ \frac{1}{2}\sigma^2(t, X_t) \partial_{yy}^2 \mathcal{J}(t,X_t,X_t).
\end{align*}
Substituting the process definitions (\emph{e.g.}, $\sigma \partial_y \mathcal{J} = \sigma \partial Y^{\smallertext{X}_\smalltext{t}}$), the driver for $Y_t$ becomes
\[
    f(t,X_t,X_t,\alpha^\star(t,X_t)) - b(t,X_t,\alpha^\star(t,X_t)) (\sigma(t,X_t) \partial Y^{\smallertext{X}_\smalltext{t}}_t) - \sigma(t,X_t) \partial Z^{\smallertext{X}_\smalltext{t}}_t - \frac{1}{2}\sigma^2(t,X_t) \partial\partial Y^{\smallertext{X}_\smalltext{t}}_t.
\]
We define the extended Hamiltonian $H$ to encapsulate the maximisation problem. For arguments $(t,x,z,\gamma,\eta,\rho)$ representing $(t, X_t, Z_t, \partial Y_t, \partial\partial Y_t, \partial Z_t)$ in $\R$
\begin{equation}\label{eq:Hamiltonian_Def_Informal}
    H(t, x, z, \gamma, \eta, \rho) \coloneqq \sup_{a\in A} \big\{f(t,x, x, a) + b(t,x,a) (z - \sigma(t,x) \gamma)\big\} - \sigma(t,x) \rho -\frac{1}{2}\sigma^2(t,x) \eta.
\end{equation}
We assume, for simplicity in this expository section, that there exists a unique $A$-valued, Borel-measurable map $\mathcal{V}^\star$ satisfying the maximisation condition. The resulting BSDE system, under the reference measure $\P$, is
\begin{equation}\label{eq:bsde_weak_1d}
    \begin{cases}
   \displaystyle  Y_t = \xi(X_T, X_T) + \int_t^T H\big(r, X_r, Z_r, \partial Y_r^{\smallertext{X}_\smalltext{r}}, \partial\partial Y_r^{\smallertext{X}_\smalltext{r}}, \partial Z_r^{\smallertext{X}_\smalltext{r}}\big)\mathrm{d}r -\int_t^T Z_r \mathrm{d}W_r,\; t\in[0,T],\\[0.5em]
   \displaystyle   \partial Y^y_t = \partial_y \xi(y, X_T) + \int_t^T\Big( \partial_y f\big(r, y, X_r, \alpha^\star_r\big)  + \partial Z^y_r b(r, X_r, \alpha^\star_r)\Big) \mathrm{d}r  -\int_t^T \partial Z^y_r \mathrm{d}W_r, \; t\in[0,T],\; y\in\R^n,\\[0.5em]
   \displaystyle  \partial\partial Y^y_t = \partial^2_{yy} \xi(y, X_T) + \int_t^T \Big(\partial^2_{yy} f\big(r, y, X_r, \alpha^\star_r\big) + 
\partial\partial Z^y_r b(r, X_r, \alpha^\star_r) \Big)\mathrm{d}r 
 -\int_t^T \partial\partial Z^y_r \mathrm{d}W_r,\; t\in[0,T],\; y\in\R^n.
\end{cases}
\end{equation}

One might ask why the system requires three equations $($including the Hessian $\partial\partial Y)$ when the original problem is characterised by $V$ and $\mathcal{J}$. The reason lies in the second-order adjustment terms that appear in the equation for $V$. In the context of {\rm BSDEs}, the process $\partial Z^y$ carries the information of the mixed derivative $($specifically $\sigma \partial_{xy}^2 \mathcal{J})$. To write our system, we need the dynamics of the gradient $\partial Y^y$. However, as seen in \eqref{eq:derivatives_J}, the dynamics of the first derivative depends on the second derivatives, such as $\partial_{yy}^2 \mathcal{J}$. Therefore, to determine the evolution of the gradient, we must simultaneously use the Hessian process $\partial\partial Y^y$.

\subsection{Assumptions}

We require the following regularity assumptions for the validity of our main results.

\begin{assumption}[Regularity and growth of the coefficients]\label{Assumptions}
    We assume the following conditions on the problem data
    \begin{enumerate}
        \item[$(i)$] \emph{continuity:} the functions $b, \sigma, f,$ and $\xi$ are continuous in all their arguments$;$
        \item[$(ii)$] \emph{regularity of the state dynamics:} the drift $b: [0,T] \times \R^n \times A \to \R^d$ is Lipschitz-continuous with respect to the state variable $x$, uniformly in $(t, a)$. That is, there exists $K > 0$ such that for all $t \in [0,T]$, $a \in A$, and $(x, x^\prime) \in \R^n\times\R^n$
        \[
            \|b(t, x, a) - b(t, x', a)\| \le K \|x - x^\prime\|;
        \]
        \item[$(iii)$] \emph{regularity and growth of the cost:} 
        for every fixed $(t, x, a)$, the cost functions $y \longmapsto f(t, x, y, a)$ and $y \longmapsto \xi(y, x)$ belong to $C^2(\R^n)$. Moreover, the functions and their partial derivatives satisfy a polynomial growth condition. There exist constants $C > 0$ and $m \ge 1$ such that for all $(t, x, y, a)\in[0,T]\times\R^n\times\R^n\times A$
        \[
            |f(t, x, y, a)| + \|\nabla_y f(t, x, y, a)\| + \|\nabla_{yy}^2 f(t, x, y, a)\| + |\xi(y,x)| + \|\nabla_y \xi(y,x)\| + \|\nabla_{yy}^2 \xi(y,x)\| \le C\big(1 + \|x\|^m + \|y\|^m + \|a\|^m\big);
        \]
        \item[$(iv)$] \emph{integrability of the state:} for any admissible control $\alpha \in \mathcal{A}$ and any $p \ge 1$, the controlled state process $X$ admits finite moments of order $p$, uniformly in time
        \[
             \E^{\P^\smalltext{\alpha}}\bigg[ \sup_{t \in [0,T]} \|X_t\|^p \bigg] < \infty;
        \]

        \item[$(v)$] \emph{non-degeneracy:}  the diffusion matrix $\sigma: [0,T] \times \R^n \to \R^{n \times d}$ is bounded and full rank.
    \end{enumerate}
\end{assumption}
The Lipschitz-continuity of the coefficients ensures that the state process remains well-behaved under reasonable controls. We formalise this in the following lemma, which justifies the integrability of the polynomial costs.

\begin{lemma}[Moment estimates for the state process]\label{lem:state_moments}
    Let {\rm\Cref{Assumptions}}.$(ii)$ hold. Let $\alpha \in \mathcal{A}$ be an admissible control such that the  drift $b^\alpha(t, x) \coloneqq b(t, x, \alpha_t)$ satisfies the linear growth condition
    \[
        \|b^\alpha(t, x)\| \le C(1 + \|x\|), \; \forall (t, x) \in [0,T] \times \R^n.
    \]
    This holds, for instance, if $\alpha$ is bounded or is a linear feedback control as in the {\rm LQR} case. Then, for any $p \ge 1$, the state process $X$ admits finite moments of order $p$ under the controlled measure $\P^\alpha$, uniformly in time
    \[
        \E^{\P^\smalltext{\alpha}}\bigg[ \sup_{t \in [0,T]} \|X_t\|^p \bigg] < \infty.
    \]
\end{lemma}

\begin{proof}
    This is a standard result in the theory of stochastic differential equations. Under the linear growth condition on the drift $b^\alpha$ and the diffusion $\sigma$ (implied by \Cref{Assumptions}.$(ii)$), the existence of moments of all orders follows from standard estimates, such as those in \cite[Theorem 5.2.2.9]{karatzas1991brownian}.
\end{proof}

\subsection{The extended dynamic programming principle}

As with all time-inconsistent problems, the classical Bellman principle fails because the cost functional changes with the state as time advances. However, we manage to prove an equality we call extended dynamic programming principle that resembles a classical DPP, and in fact implies it in the absence of $x$ in the reward functional.

\begin{theorem}[Extended dynamic programming principle]\label{thm: extendedDPP}
 Let {\rm\Cref{Assumptions}} hold and let $\alpha^\star\in \Ec$ be an equilibrium control.
 Then, for any $t\in[0,T]$, for all $s \in[0,t]$ and $x\in \mathbb{R}^n$, we have
 \begin{equation} \label{eq:extended_dpp}
 \begin{aligned}
       v(s, x) &= \sup_{\alpha \in \Ac} \E^{\mathbb{P}^{\smalltext{s},\smalltext{x}, \smalltext{\alpha}}}\Bigg[v(t, X_t) + \int_s^t \Bigg(f(r,  X_r, X_r, \al_r) \\
       &\quad -b(r, X_r, \al_r) \cdot \sigma(r, X_r)^\top \E^{\mathbb{P}^{\smalltext{r}, \smalltext{X}_\tinytext{r},\smalltext{\alpha}^\tinytext{\star}}}\bigg[\nabla_{y}\xi(X_r, X_T) + \int_r^T\nabla_{y} f\big(u, X_r, X_u, \alpha^\star_u\big) \d u\bigg] \\
       &\quad -\Tr\Bigg[\sigma(r, X_r)\sigma(r, X_r)^\top \E^{\mathbb{P}^{\smalltext{r}, \smalltext{X}_\tinytext{r}, \smalltext{\alpha}^\tinytext{\star}}}\bigg[\nabla^2_{yx}\xi(X_r, X_T) + \int_r^T\nabla^2_{yx} f\big(u, X_r, X_u, \alpha^\star_u\big) \d u\bigg]\Bigg] \\
       &\quad -\frac{1}{2}\Tr\Bigg[\sigma(r, X_r)\sigma(r, X_r)^\top \E^{\mathbb{P}^{\smalltext{r}, \smalltext{X}_\tinytext{r},\smalltext{\alpha}^\tinytext{\star}}}\bigg[\nabla^2_{yy}\xi(X_r, X_T) + \int_r^T\nabla^2_{yy} f\big(u, X_r, X_u, \alpha^\star_u\big) \d u\bigg]\Bigg]\Bigg)\d r\Bigg].
 \end{aligned}
 \end{equation}
Furthermore, the equilibrium control $\alpha^\star$ attains the supremum in \eqref{eq:extended_dpp}.
\end{theorem}

The three last rows represent the cost of time-inconsistency: the drift in value caused solely by the updating of preferences along the path. This result is the main building block for the rest of the theory developed in this paper. See \Cref{Appendix:ddp} for the proof.

\subsection{A necessity result}

We recall that, for a fixed equilibrium control $\alpha^\star$, we will very often use the following notation
\[
\mathcal{J}(t, x, y) \coloneqq  \E^{\P^{\smalltext{t}, \smalltext{x}, \smalltext{\alpha}^\tinytext{\star}}} \bigg[ \int_t^T f(u,    y, X_u, \alpha^\star_u) \d u + \xi(y, X_T) \bigg],\; t\in[0,T],\; y\in\R^n.
\]

In other words, $\mathcal{J}(t, x, y)$ represents the payoff under the equilibrium control if we were to freeze the parameter $y$. The next theorem guarantees that smooth equilibrium controls implicitly define solutions to \eqref{BSDE Weak Formulation}.

\medskip

Using the extended DPP, we can formally characterise the equilibrium via the system of BSDEs \eqref{BSDE Weak Formulation}. We identify the scalar value process $Y_t = v(t, X_t)$, the gradient vector process $\partial Y^y_t = \partial_y \mathcal{J}(t, X_t, y)$, and the Hessian matrix process $\partial\partial Y^y_t = \partial^2_{yy} \mathcal{J}(t, X_t, y)$.
\begin{equation}\label{BSDE Weak Formulation}
    \begin{cases}
 
    \displaystyle Y_t = \xi(X_T, X_T) + \int_t^T H\big(r, X_r, Z_r, \partial Y_r^{\smallertext{X}_\smalltext{r}}, \partial\partial Y_r^{\smallertext{X}_\smalltext{r}}, \partial Z_r^{\smallertext{X}_\smalltext{r}}\big)\d r -\int_t^T Z_r  \d W_r, \; t\in[0,T].\\
    \displaystyle \partial Y^y_t = \nabla_y \xi(y, X_T) + \int_t^T \Big( \nabla_y f(r,   y, X_r, \alpha^\star_r)  +   \partial Z^y_r b(r, X_r, \alpha^\star_r)\Big)\d r - \int_t^T \partial Z^y_r  \d W_r, \; t\in[0,T],\; y\in\R^n,\\
    \displaystyle 
    \partial\partial Y^y_t = \nabla^2_{yy} \xi(y, X_T) + \int_t^T \Big( \nabla^2_{yy} f(r,   y, X_r, \alpha^\star_r) +  \partial\partial Z^y_r b(r, X_r, \alpha^\star_r) \Big)\d r - \int_t^T \partial\partial Z^y_r  \d W_r,\; t\in[0,T],\; y\in\R^n.
    \end{cases}
\end{equation}

\medskip

Here, the extended Hamiltonian $H$ is defined to match the variables introduced in the informal derivation. For a state $x \in \R^n$, it takes as arguments the co-state $z \in \R^d$, the parameter gradient $\gamma \in \R^n$, the parameter Hessian $\eta \in \Mc_{n}(\R)$, and the mixed consistency term $\rho \in \Mc_{n, d}(\R)$

\begin{equation}\label{eq:Hamiltonian_Def}
    H(t, x, z, \gamma, \eta, \rho) \coloneqq  \sup_{a \in A} \big\{ f(t, x, x, a) + b(t, x, a) \cdot (z - \sigma(t, x)^\top \gamma) \big\} - \frac{1}{2}\Tr\big[\sigma(t, x)\sigma(t, x)^\top \eta\big] - \Tr\big[\sigma(t, x) \rho^\top \big].
\end{equation}

\begin{remark}[Dimensionality of the adjoint processes]
    Let us clarify the dimensions of the processes appearing in the system \eqref{BSDE Weak Formulation}. Let us recall that the state process $X$ takes values in $\R^n$ and the Brownian motion $W$ in $\R^d$.
    \begin{itemize}
        \item \emph{Value process:} $Y$ is scalar-valued in $\R$. Its volatility $Z$ takes values in $\R^d$.
        \item \emph{Gradient process:} $\partial Y$ takes values in $\R^n$ (representing $\nabla_y \mathcal{J}$). Its volatility $\partial Z$ is defined as a matrix in $\R^{n \times d}$.
        This specific dimension is required by the Hamiltonian term $\Tr[\sigma \rho]$ in \eqref{eq:Hamiltonian_Def}. Since $\sigma \in \R^{n \times d}$, the variable $\rho$ $($identified with $\partial Z)$ must be in $\R^{n \times d}$ for the product $\sigma \rho^\top$ to be a square matrix in $\R^{n \times n}$.
        \item \emph{Hessian process:} $\partial\partial Y$ takes values in $\R^{n \times n}$ $($representing $\nabla^2_{yy} \mathcal{J})$. Consequently, its volatility $\partial\partial Z$ is a rank-3 tensor in $\R^{n \times n \times d}$, representing the sensitivity of each entry of the Hessian matrix to the $d$ components of the Brownian motion.
    \end{itemize}
\end{remark}

\begin{remark}[Consistency with the classical theory]
    The Hamiltonian defined in \eqref{eq:Hamiltonian_Def} includes the terms involving $\gamma$, $\eta$, and $\rho$, which differ from the standard Hamiltonian in time-consistent stochastic control. These terms represent the \emph{inconsistency adjustment}. Indeed, consider a standard time-consistent problem where the cost functions $f$ and $\xi$ do not depend on the parameter $y$. In this case, the auxiliary value function $\mathcal{J}(t, x, y)$ is independent of $y$, implying that the derivatives $\nabla_y \mathcal{J}$, $\nabla^2_{yy} \mathcal{J}$, and $\nabla^2_{xy} \mathcal{J}$ vanish. Consequently, the inputs $\gamma$, $\eta$, and $\rho$ are zero, and the Hamiltonian reduces to
    \[
        H(t, x, z, 0, 0, 0) = \sup_{a \in A} \big\{ f(t, x, a) + b(t, x, a) \cdot z \big\}.
    \]
    Thus, we recover the standard Hamiltonian from the classical stochastic control theory.
\end{remark}

Let us define what we mean by the solution to such a system.
\begin{definition}\label{def:bsde_solution}
    We say that $(Y, Z, \partial Y, \partial Z, \partial\partial Y, \partial\partial Z)$ is a solution to the system \eqref{BSDE Weak Formulation} if
    \begin{enumerate}
        \item[$(i)$] the system of equations \eqref{BSDE Weak Formulation} holds $\P${\rm--a.s.}$;$
        \item[$(ii)$] the value process and its control satisfy the standard integrability
        \[ 
        Y \in \mathbb{S}^2(\R,\mathbb{F}, \P),\; Z \in \mathbb{H}^2(\R^d,\mathbb{F}, \P); 
        \]
        \item[$(iii)$] the derivative random fields belong to the locally uniform spaces. That is, for any $\psi \in \{ \partial Y, \partial\partial Y\}$ and $\phi \in \{ \partial Z, \partial\partial Z\}$
        \[ \psi \in \mathfrak{S}^2_{\mathrm{loc}}(\R^{k_1},\F,\P),\; \phi \in \mathfrak{H}^2_{\mathrm{loc}}(\R^{k_2},\F,\P), 
        \]
        where $k_1$ and $k_2$ represent the right dimensions of the derivative random fields.
    \end{enumerate}
\end{definition}

In other words, we ask the processes to be in the classical spaces for the solution of BSDEs, but we additionally ask that the norms of the families indexed by the parameter $y$ are uniformly bounded in the sense of the norm of convergence over compact subsets. Compared with the definition of solution given in \citeauthor{hernandez2023me} \cite{hernandez2023me}, where the space in which the uni-parametric family took values was already compact, we need to consider a weaker norm. 
\begin{theorem}[Necessity]\label{thm:necessity}
    Let {\rm\Cref{Assumptions}} hold and let $\alpha^\star \in \Ac$ be an equilibrium control in the sense of {\rm\Cref{Def:equilibrium}}. Assume that the equilibrium value function $V(t, x) \coloneqq  J(t, x, \alpha^\star)$ belongs to $C^{1,2}([0,T) \times \R^n) \cap C^0([0,T] \times \R^n)$ and the parametric function 
    \[
    \mathcal{J}(t, x, y) \coloneqq  \E^{\P^{\smalltext{t}, \smalltext{x}, \smalltext{\alpha}^\tinytext{\star}}}\bigg[\int_t^T f(s,   y, X_s, \alpha^\star_s) \d s + \xi(y, X_T)\bigg],
    \]
    belongs to $C^{1,2, 2}([0,T) \times \R^n \times \R^n) \cap C^0([0,T] \times \R^n \times \R^n)$. Then, the processes $(Y, Z, \partial Y, \partial Z, \partial\partial Y, \partial\partial Z)$ defined by
  \begin{gather*}
    Y_t \coloneqq  V(t, X_t), \; Z_t \coloneqq   \nabla_x V(t, X_t)\sigma(t, X_t),\; t\in[0,T],\\
    \partial Y^y_t \coloneqq  \nabla_y \mathcal{J}(t, X_t, y), \; \partial Z^y_t \coloneqq   \nabla_{yx}^2 \mathcal{J}(t, X_t, y)\sigma(t, X_t),\; t\in[0,T],\; y\in\R^n,\\
    \partial\partial Y^y_t \coloneqq  \nabla^2_{yy} \mathcal{J}(t, X_t, y), \; \partial\partial Z^y_t \coloneqq   \nabla^3_{xyy} \mathcal{J}(t, X_t, y)\sigma(t, X_t), \; t\in[0,T],\; y\in\R^n,
    \end{gather*}
    provided they belong to the suitable spaces stated in {\rm\Cref{def:bsde_solution}}, solve the {\rm BSDE} system \eqref{BSDE Weak Formulation}. Furthermore, $\alpha^\star$ satisfies the optimality condition
    \begin{equation}\label{eq:optimality_condition}
        \alpha^\star_t \in \underset{a \in A}{\mathrm{argmax}} \Big\{ f(t, X_t, X_t, a) + b(t, X_t, a) \cdot \big( Z_t - \sigma(t, X_t)^\top \partial Y^{\smallertext{X}_\smalltext{t}}_t \big) \Big\},\; \mathrm{d}t\otimes\P\text{\rm--a.e.}
    \end{equation}
\end{theorem}

The proof can be found in \Cref{Appendix:Necessity}.

\medskip

\begin{remark}[Structure of the inconsistency adjustment]\label{remark:necessity_clarification}
    In the optimality condition above, it is important to note that the auxiliary function $\mathcal{J}(t, x, y)$ and the process $\partial Y_t^{\smallertext{X}_\smalltext{t}}$ are defined for a fixed equilibrium strategy $\alpha^\star$.
\end{remark}
\subsection{Verification theorem}

We now present the verification theorem, which states that a solution to the derived BSDE system, satisfying the Hamiltonian maximisation condition, yields an equilibrium control.
\begin{theorem}[Verification]\label{thm:verification}
    Let {\rm\Cref{Assumptions}} hold. Assume there exists a solution $(Y, Z, \partial Y, \partial Z, \partial\partial Y, \partial\partial Z)$ to the system \eqref{BSDE Weak Formulation} in the sense of {\rm\Cref{def:bsde_solution}}.

\medskip

    Define the candidate feedback control process $\alpha^\star = (\alpha^\star_t)_{t \in [0, T]}$ by the condition that it maximises the extended Hamiltonian
    \begin{equation}\label{eq:verification_argmax}
        \alpha^\star_t \in \underset{a \in A}{\mathrm{argmax}} \big\{ f (t, X_t, X_t, a) + b(t, X_t, a) \cdot \big( Z_t - \sigma(t, X_t)^\top \partial Y^{\smallertext{X}_\smalltext{t}}_t \big) \big\}, \; \d t \otimes \d \P\text{\rm--a.e.}
    \end{equation}
    
    Suppose further that
    \begin{enumerate}
        \item[$(i)$] the control process $\alpha^\star$ is admissible, \emph{i.e.}, $\alpha^\star \in \Ac;$
        \item[$(ii)$] the function $v(t, x)$ identified with $Y$ via $Y_t = v(t, X_t)$ belongs to $C^{1,2}([0,T) \times \R^n) \cap C^0([0,T] \times \R^n)$.
    \end{enumerate}
    Then, $\alpha^\star$ is an equilibrium control, and $Y_t$ is the associated value process, \emph{i.e.}, $Y_t = J(t, X_t, \alpha^\star)$.
\end{theorem}

 The proof can be found in \Cref{Appendix:verification}.

 \begin{remark}[Existence of a measurable equilibrium feedback]\label{rem:measurable_selection}
    In the statement of {\rm\Cref{thm:verification}}, we defined the candidate control $\alpha^\star$ via the maximisation of the Hamiltonian, assuming that an admissible, measurable selection of the argmax exists. Let us briefly mention why assuming this is perfectly reasonable in our setting.
    \medskip
    
    Consider the set-valued map $\Phi: [0,T] \times \R^n \times \R^d \times \R^n \rightrightarrows A$ defined by the set of maximisers
    \[
        \Phi(t, x, z, \gamma) \coloneqq \underset{a \in A}{\mathrm{argmax}} \big\{ f(t, x, x, a) + b(t, x, a) \cdot (z - \sigma(t, x)^\top\gamma) \big\}.
    \]
    Under {\rm\Cref{Assumptions}}, the coefficients $b, \sigma,$ and $f$ are continuous in all arguments. 
    Consequently, the function being maximised is jointly continuous in $((t, x, z, \gamma), a)$, which implies that the map $\Phi$ has a measurable graph and takes closed values. 
    \medskip
    
    Since the action space $A$ is a closed subset of a Polish space $($and assuming the maximum is attained, \emph{e.g.}, if $A$ is compact or under suitable coercivity conditions$)$, the \emph{Kuratowski--Ryll--Nardzewski selection theorem} $($or rather, a corollary of it, see, \emph{e.g.}, {\rm\cite[Theorem 17.18]{aliprantis2006infinite}}$)$ guarantees the existence of a Borel-measurable function $\mathcal{V}^\star: [0,T] \times \R^n \times \R^d \times \R^n \longrightarrow A$ such that $\mathcal{V}^\star(t, x, z, \gamma) \in \Phi(t, x, z, \gamma)$ for all inputs.
    Defining the process $\alpha^\star_t \coloneqq \mathcal{V}^\star(t, X_t, Z_t, \partial Y^{\smallertext{X}_\smalltext{t}}_t)$ yields an $\F$-predictable control candidate.
\end{remark}

 \subsection{Well-posedness of the solution}
 \label{sec:wp}

We finish with a result guaranteeing existence of solutions in the sense of \Cref{def:bsde_solution}. We first define the driver functions $G_1$ and $G_2$ corresponding to the second and third equations of the system \eqref{BSDE Weak Formulation}.
We denote the arguments by $(t, x, y, z, \gamma, v, \mathfrak{v})$, where $z$ represents the volatility of the value process $Z_t$, $\gamma$ represents the inconsistency term $\partial Y_t^{X_t}$, and $v, \mathfrak{v}$ represent the derivative volatilities $\partial Z^y_t$ and $\partial\partial Z^y_t$, respectively.
\begin{align*}
    G_1(t, x, y, z, \gamma, v) &\coloneqq \nabla_y f(t, y, x, \alpha^\star) + v b(t, x, \alpha^\star), \;
G_2(t, x, y, z, \gamma, \mathfrak{v}) \coloneqq \nabla_{yy}^2 f(t, y, x, \alpha^\star) + \mathfrak{v} b(t, x, \alpha^\star),
\end{align*}

where $\alpha^\star \coloneqq \mathcal{V}^\star(t, x, z, \gamma)$ (see \Cref{rem:measurable_selection}).

\begin{assumption}[Drivers integrability and  regularity]\label{ass:poly_growth}
    Let $\Theta \coloneqq (z, \mathfrak{u}, v, \mathfrak{v})$ be the vector of inputs for the drivers $($representing the $Z$, $\partial Y$, $\partial Z$, and $\partial \partial Z$ components respectively$)$. We assume there exist a constant $C > 0$ such that:
    \begin{enumerate}
        \item[$(i)$] \emph{Regularity of the Hamiltonian driver $H$.} The driver of the value process satisfies a Lipschitz-continuity condition. For any $t, x$ and inputs $\Theta$, $\Theta^\prime$
        \[
            |H(t, x, \Theta) - H(t, x, \Theta^\prime)| \le C \|\Theta - \Theta^\prime\|.
        \]

        \item[$(ii)$] \emph{Structure of the derivative drivers $G \in \{G_1, G_2\}$.} The drivers for the gradient and Hessian processes satisfy a Lipschitz-continuity condition. For any parameter $y$ and input vector $\Theta$
         \[
            |G(t, x, y, \Theta) - G(t, x, y, \Theta^\prime)| \le C \|\Theta - \Theta^\prime\|.
        \]

        \item[$(iii)$] \emph{Integrability of source terms.} The terminal conditions and the drivers evaluated at the null input vector $\Theta^0 \coloneqq 0$ satisfy the following integrability requirements
\begin{itemize}
    \item \emph{value process source:} the diagonal terminal cost and the base Hamiltonian are square-integrable
    \[
    \mathbb{E}^{\mathbb{P}}\bigg[ \big|\xi(X_T, X_T)\big|^2 + \int_0^T \big|H(t, X_t, \Theta^0)\big|^2 \mathrm{d}t \bigg] < \infty;
    \]
    \item \emph{derivative fields source:} the parameter-dependent source terms have finite weighted norms
    \[
    \sup_{y \in \mathbb{R}^n} \rho(y) \mathbb{E}^{\mathbb{P}}\bigg[ \big|\nabla_y \xi(y, X_T)\big|^2 + \big|\nabla_{yy}^2 \xi(y, X_T)\big|^2 + \int_0^T \sum_{i=1}^2 \big|G_i(t, X_t, y, \Theta^0)\big|^2 \mathrm{d}t \bigg] < \infty.
    \]
\end{itemize}

    \end{enumerate}
\end{assumption}

Note that the integrability of the state process $X$ is already guaranteed by {\rm\Cref{Assumptions}}.$(iv)$, which is essential to ensure that these polynomial bounds result in integrable random variables. Now we are able to state our uniqueness and existence result.

\begin{theorem}[Well-posedness]\label{thm:wellposedness}
    Under {\rm \Cref{Assumptions,ass:poly_growth}}, there exists a weighting parameter $\beta > 0$ such that the {\rm BSDE} system \eqref{BSDE Weak Formulation} admits a {unique solution} in the weighted space $\mathcal{K}_\beta$. Consequently, this solution also satisfies the conditions of {\rm \Cref{def:bsde_solution}}.
\end{theorem}

We remark that \Cref{ass:poly_growth} imposes strong Lipschitz-continuity requirements, and that the inconsistent linear--quadratic regulator is not covered by our result. Our point here is to present a general well-posedness result, and demonstrate the kind of techniques and spaces that are necessary to consider. We believe that a result where $H$, $G_1$ and $G_2$ have a stochastic Lipschitz coefficient proportional to $1+ \|X_t\|^2 + \|Z_t\|$ (which is exactly what is required to cover the linear--quadratic example) is achievable and we leave it as an open problem for future research. We will content ourselves here to mention that the literature on BSDEs whose generators have BMO Lipschitz-continuity constants, see \citeauthor*{imkeller2012differentiability} \cite{imkeller2012differentiability}, or quadratic BSVIEs, see \citeauthor*{hernandez2023quadratic} \cite{hernandez2023quadratic}, should be a good starting point.

\begin{remark}[Dependency of the functional spaces on the driver's growth]\label{rem:weight}
    The definition of the weighted space $\mathcal{K}_\beta$ involving the polynomial weight $\rho(y) \coloneqq (1+\|y\|^2)^{-k}$ is not intrinsic to the general theory but is a specific choice made to accommodate the polynomial growth as the one we have on the LQR case.

\end{remark}

\section{An example: the linear--quadratic time-inconsistent regulator}\label{sec:LQR}

After introducing all our results, we present a full study of a time-inconsistent problem whose inconsistency comes fully from the presence of the current state variable in the reward functional.

\subsection{Problem setting}\label{sec:LQRsetting}

We consider the linear--quadratic regulator (LQR) problem with a state-dependent terminal cost, a classical example in the literature of time-inconsistent control (see \citeauthor*{bjork2021time} \cite{bjork2021time, hu2017time}).
For simplicity, we take the dimension of the state process to be $n=d=1$.
The state process $X$ evolves according to the linear dynamics
\begin{equation}
    \d X_t = \big(\bar{a}X_t + \bar{b}\al_t\big)\d t + \sigma \d W_t, \; X_0 = x_0.
\end{equation}
The objective is to minimise the squared distance of the terminal state from the current state, penalised by the control effort.
Hence, the cost functional is given by
\begin{equation}\label{eq:lossLQR}
    J(t, x, \alpha) \coloneqq  \E^{\mathbb{P}^{\smalltext{t}, \smalltext{x}, \smalltext{\alpha}}}\bigg[\int_t^T \frac{1}{2} \al_s^2 \d s + \frac{\Gamma}{2} (X_T-x)^2\bigg].
\end{equation}
Here, we identify $f(t,  y, x, a) = \frac{1}{2}a^2$ and $\xi(y, x) = \frac{\Gamma}{2}(x-y)^2$.
The appearance of the current state $x$ in the terminal cost $\xi$ creates the time-inconsistency.

\begin{example}[Motivation: the political economy of debt management]
\label{ex:gov_debt}
Consider a government managing its national debt ratio $X$. The dynamics are governed by the interest rate gap $\bar{a}$ $($growth rate of debt$)$ and fiscal adjustments $\alpha$ $($surplus/deficit spending$)$
\[
\mathrm{d}X_t =\big (\bar{a}X_t + \bar{b}\alpha_t\big)\mathrm{d}t + \sigma \mathrm{d}W_t.
\]
The government aims at minimising the cost of fiscal interventions (tax distortions), represented by $\frac{1}{2}\alpha_t^2$.
However, the terminal objective exhibits \textit{reference point adaptation}. A government at time $t$ commits to bringing the debt $X_T$ close to their current observed level $X_t$. They penalise deviations from this inherited baseline rather than an absolute historical zero
\[
J(t, x, \alpha) =  \E^{\mathbb{P}^{\smalltext{t}, \smalltext{x}, \smalltext{\alpha}}}\bigg[ \int_t^T \frac{1}{2}\alpha_s^2 \mathrm{d}s + \frac{\Gamma}{2}(X_T - x)^2 \bigg].
\]
This creates a time-inconsistent preference structure: as the debt drifts, future administrations continuously reset the target $x$ to the new prevailing debt level, leading to the \emph{`}drifting goalpost\emph{'} phenomenon that we will analyze shortly.
\end{example}

\subsection{Equilibrium controls representation}\label{sec:LQRsolution}

Following the general theory in \Cref{sec:mainresults}, the equilibrium value function and the associated dual processes are characterised by the BSDE system \eqref{BSDE Weak Formulation}.
For the LQR problem, this system corresponds to, under $\P$

\begin{equation}\label{eq:BSDE for LQR}
    \begin{cases}
       \displaystyle Y_t =  \int_t^T H\big(r, X_r, Z_r, \partial Y_r^{\smallertext{X}_\smalltext{r}}, \partial\partial Y_r^{\smallertext{X}_\smalltext{r}}, \partial Z_r^{\smallertext{X}_\smalltext{r}}\big) \d r - \int_t^T Z_r\cdot \d W_r,\; t\in[0,T],\\[0.5em]
     \displaystyle   \partial Y_t^y = \Gamma(y - X_T) + \int_t^T \partial Z_r^y \sigma^{-1}(\bar{a}X_r + \bar{b}\alpha^\star_r) \d r - \int_t^T \partial Z_r^y\cdot \d W_r, \; t\in[0,T],\\[0.5em]
   \displaystyle     \partial \partial Y_t^y = \Gamma,\; t\in[0,T].
    \end{cases}
\end{equation}

The extended Hamiltonian $H$ corresponds to:
\begin{align*}
    H(t, x, z, \gamma, \eta, \rho) \coloneqq  \inf_{a \in \R} \bigg\{ \frac{1}{2}a^2 + (\bar{a}x + \bar{b}a)\sigma^{-1}z - (\bar{a}x + \bar{b}a)\gamma \bigg\} - \frac{1}{2}\sigma^2 \eta - \sigma \rho.
\end{align*}
The equilibrium control $\alpha^\star$ is the minimiser of this Hamiltonian.
The first-order conditions yield
\[
    a + \bar{b}\sigma^{-1}z - \bar{b}\gamma = 0 \Longleftrightarrow a = -\bar{b}(\sigma^{-1}z - \gamma).
\]

\begin{remark}[Sign convention]
    The general theory in \emph{\Cref{sec:mainresults}} is formulated as a maximisation
    problem, with the agent seeking to maximise the functional $J$.
    The linear--quadratic example studied in this section is instead
    a \emph{minimisation} problem: the agent incurs a quadratic
    running cost $\tfrac{1}{2}\alpha^2$ and a quadratic terminal
    penalty $\tfrac{\Gamma}{2}(X_T - x)^2$, both non-negative, and
    seeks to minimise their expected sum.  To embed this within the
    general framework it suffices to replace $J$ by $-J$ throughout,
    or equivalently to replace $\sup$ by $\inf$ in the Hamiltonian
    \eqref{eq:Hamiltonian_Def} and reverse the inequality in the
    equilibrium condition \ref{Def:equilibrium}.  All structural
    results---the extended \emph{DPP}, the \emph{BSDE} characterisation, the
    necessity and verification theorems---carry over verbatim under
    this sign change. In the notation of this section we, therefore
, write the Hamiltonian as an infimum and identify
    $f(t,y,x,a) = \tfrac{1}{2}a^2$ and
    $\xi(y,x) = \tfrac{\Gamma}{2}(x-y)^2$.
\end{remark}

\begin{remark}[Verification of assumptions]
    The {\rm LQR} problem fits within the framework of \emph{\Cref{Assumptions}}. Thus, \emph{\Cref{thm:necessity}} and \emph{\Cref{thm:verification}} apply to this case. In particular, all equilibria that satisfy the hypotheses of \emph{\Cref{thm:necessity}} must satisfy the above {\rm BSDE}. The fact that \emph{\Cref{thm:wellposedness}} cannot be used here simply prevents us from stating that the equilibrium we are deriving below is unique.
\end{remark}

Substituting the BSDE variables $Z_t$ and $\partial Y_t^{\smallertext{X}_\smalltext{t}}$, we obtain the feedback form
\begin{equation}\label{eq:LQRequilibrium}
    \alpha^\star_t = \bar{b}\big(\partial Y_t^{\smallertext{X}_\smalltext{t}} - \sigma^{-1}Z_t\big).
\end{equation}

To explicitly solve this system, we make use of \Cref{thm:verification} by looking for a decoupling field $\mathcal{J}(t, x, y)$ such that $\mathcal{J}(t, X_t, y) = \mathcal{Y}_t^y$.
This function must solve the following parametrised PDE
\begin{equation}\label{eq:PDE_J}
    \partial_t \mathcal{J} + (\bar{a}x + \bar{b}\alpha^\star)\partial_x \mathcal{J} + \frac{1}{2}\sigma^2 \partial_{xx} \mathcal{J} + \frac{1}{2}(\alpha^\star)^2 = 0, \; \mathcal{J}(T, x, y) = \frac{\Gamma}{2}(x-y)^2.
\end{equation}

\begin{lemma}[Derivation of the Riccati system]\label{lemma:LQR_Riccati}
 Assume that the value function admits the quadratic \emph{Ansatz}
    \begin{equation}\label{eq:Ansatz}
        \mathcal{J}(t, x, y) = A(t)x^2 + B(t)y^2 + C(t)xy + D(t)x + F(t)y + H(t).
\end{equation}
    Then, the equilibrium control is linear in $x$
    \begin{equation}\label{eq:LQR_control_sol}
        \alpha^\star(t, x) = -\bar{b}\big( (2A(t) + C(t))x  \big).
\end{equation}
    The time-dependent coefficients satisfy the following system of ordinary differential equations
    \begin{equation}\label{eq:SystemODEs}
    \begin{aligned}
        &A^{\prime} + 2\bar{a}A - 2\bar{b}^2 A(2A+C) + \frac{1}{2}\bar{b}^2(2A+C)^2 = 0,  \;A(T)=\Gamma/2, \\
        &C^{\prime} + \bar{a}C - \bar{b}^2 C(2A+C) = 0, \; C(T)=-\Gamma, \\
        &H^{\prime} - \frac{1}{2}\bar{b}^2 D^2 + \sigma^2 A = 0,  \; H(T)=0,
        \end{aligned}
    \end{equation}
    with $B(t) \equiv \Gamma/2$, $D(t) \equiv 0$ and $F(t) \equiv 0$.
    \end{lemma}

\begin{proof}
    We derive the system by substituting the \emph{Ansatz} into the equilibrium condition and the PDE.
    First, recall the identifications from the Markovian setting: $Z_t = \sigma \partial_x V(t, x)$ and $\partial Y_t^{\smallertext{X}_\smalltext{t}} = \partial_y \mathcal{J}(t, x, y)|_{y=x}$, where $V(t, x) = \mathcal{J}(t, x, x)$ is the equilibrium value function.
    Using the \emph{Ansatz} \eqref{eq:Ansatz}, the derivatives are
    \begin{align*}
        \partial_x \mathcal{J}(t, x, y) &= 2A(t)x + C(t)y + D(t), \;  \partial_y \mathcal{J}(t, x, y) = 2B(t)y + C(t)x + F(t).
\end{align*}
    The equilibrium value function is $V(t, x) = (A+B+C)x^2 + (D+F)x + H$.
 Thus,
    \[
        \partial_x V(t, x) = 2(A+B+C)x + (D+F).
\]
    Substituting these into the control formula \eqref{eq:LQRequilibrium} (noting that $\sigma^{-1}Z_t = \partial_x V(t,x)$ implies the term $\bar{b}(\partial Y - \sigma^{-1}Z)$ corresponds to $\bar{b}(\partial_y \mathcal{J} - \partial_x V)$):
    \begin{align*}
        \alpha^\star(t, x) = \bar{b} \Big( \big( 2Bx + Cx + F \big) - \big( 2(A+B+C)x + D+F \big) \Big) 
        = -\bar{b} \big( (2A + C)x + D \big).
\end{align*}
    Let us define the feedback gains $K(t) \coloneqq  \bar{b}(2A(t)+C(t))$ and $\Lambda(t) \coloneqq  \bar{b}D(t)$, so $\alpha^\star = -Kx - \Lambda$.
    Now, substitute $\mathcal{J}$ and $\alpha^\star$ into the PDE \eqref{eq:PDE_J}. We expand all terms fully
    \[
        \underbrace{(A^{\prime}x^2 + B^{\prime}y^2 + C^{\prime}xy + D^{\prime}x + F^{\prime}y + H^{\prime})}_{\partial_\smalltext{t} \mathcal{J}} + (\bar{a}x - \bar{b}Kx - \bar{b}\Lambda)\underbrace{(2Ax + Cy + D)}_{\partial_\smalltext{x} \mathcal{J}} + \frac{1}{2}\sigma^2 \underbrace{(2A)}_{\partial_{\smalltext{x}\smalltext{x}}\mathcal{J}} + \frac{1}{2}\underbrace{(K^2x^2 + 2K\Lambda x + \Lambda^2)}_{(\alpha^\smalltext{\star})^\smalltext{2}} = 0.
    \]
    Matching coefficients for each monomial term
    \begin{itemize}
        \item $x^2$: $A^{\prime} + 2A(\bar{a} - \bar{b}K) + \frac{1}{2}K^2 = 0$.
        Substituting $K$
        \[ A^{\prime} + 2\bar{a}A - 2\bar{b}^2 A(2A+C) + \frac{1}{2}\bar{b}^2(2A+C)^2 = 0. \]
        \item $xy$: $C^{\prime} + C(\bar{a} - \bar{b}K) = 0 \implies C^{\prime} + \bar{a}C - \bar{b}^2 C(2A+C) = 0$.
        \item $x$: $D^{\prime} + D(\bar{a} - \bar{b}K) - 2A\bar{b}\Lambda + K\Lambda = 0$.
        Substituting $\Lambda = \bar{b}D$
        \[ D^{\prime} + D(\bar{a} - \bar{b}K) - 2A\bar{b}^2 D + \bar{b}KD = D^{\prime} + \bar{a}D - 2\bar{b}^2 AD = 0. \]

        \item $y^2$: $B^{\prime} = 0$.
        Boundary condition $B(T) = \Gamma/2 \implies B(t) \equiv \Gamma/2$.
        \item $y$: $F^{\prime} - C\bar{b}\Lambda = 0 \implies F^{\prime} - \bar{b}^2 CD = 0$. This implies $F^{\prime} = 0 \implies F(t) \equiv 0$.
        \item constant: $H^{\prime} - D\bar{b}\Lambda + \sigma^2 A + \frac{1}{2}\Lambda^2 = 0$.
        Since $D \equiv 0 \implies \Lambda \equiv 0$, this simplifies to $H^{\prime} + \sigma^2 A = 0$.
    \end{itemize}
    Finally, note that since $D(t) \equiv 0$, the affine part of the control vanishes, and $\alpha^\star(t, x) = -K(t)x$.
\end{proof}

\subsection{Comparison of strategies}\label{sec:numsimLQR}

We compare the performance of the sophisticated (equilibrium) agent against the naive agent. More precisely, we consider

\medskip
\emph{$(i)$ equilibrium strategy:} defined by $\alpha^\star(t, x) = -K_{eq}(t)x$, where $K_{eq}(t) = \bar{b}(2A(t) + C(t))$ is derived from \Cref{lemma:LQR_Riccati}.

\medskip
\emph{$(ii)$ naive strategy:} the naive feedback law $\alpha^{\text{naive}}(t,x) = -K_{\text{naive}}(t)x$ is derived by solving a standard time-consistent LQR problem at each instant $t$, where the agent treats the current state as a fixed target $y = X_t$ for the remaining horizon $[t, T]$. By postulating a quadratic value function $V(s, x; y) = P(s)x^2 + Q(s)xy + R(s)y^2 + M(s)x + N(s)y + L(s)$, the HJB equation for a fixed parameter $y$ yields the following system for the principal coefficients
\begin{gather*}
P'(t) + 2\overline{a}P(t) - 2\overline{b}^2 P(t)^2 = 0, \; P(T) = \Gamma/2,\\
Q'(t) + (\overline{a} - 2\overline{b}^2 P(t))Q(t) = 0, \; Q(T) = -\Gamma.
\end{gather*}
Solving for $Q(t)$ via an integrating factor and evaluating the optimal control $a^* = -\overline{b}(2P(t)x + Q(t)y)$ on the diagonal where $y=x$ leads directly to:
\begin{equation}
K_{\text{naive}}(t) = \overline{b} \left( 2P(t) - \Gamma \exp \left( \int_{t}^{T} (\overline{a} - 2\overline{b}^2 P(u)) du \right) \right).
\end{equation}

\medskip

We simulate the trajectories of the state process $X$ under both strategies using an Euler--Maruyama discretisation. We use the parameters $T=1$, $\bar{a}=0.5$, $\bar{b}=1$, $\sigma=0.5$, $x_0=1$, and $\Gamma=5.0$.
\medskip

\begin{figure}[ht!]
    \centering
    \includegraphics[width=0.9\textwidth]{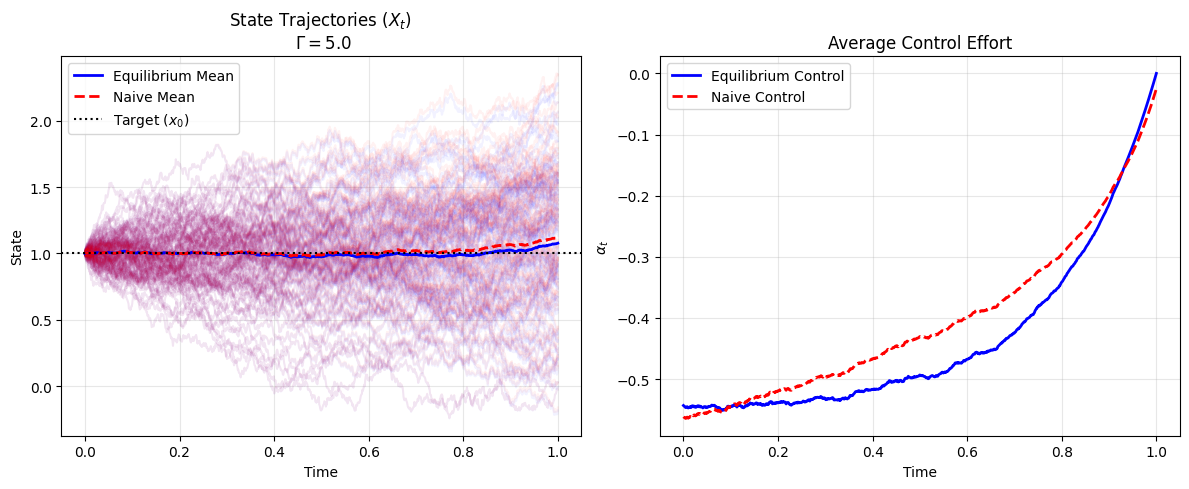}
    \caption{\footnotesize Comparison of state trajectories (left) and control effort (right). The equilibrium strategy (blue) maintains the state slightly near the target $x_0=1.0$ with moderate effort. The naive strategy (red dashed) applies more control initially.}
    \label{fig:LQR_Comparison}
\end{figure}

To rigorously quantify the performance gap, we compute the exact expected time-$0$ cost $J(0, x_0)$ for both strategies. Since both strategies are linear feedback laws of the form $\alpha(t, x) = -K(t)x$, we can derive the cost analytically.
\begin{proposition}[Exact cost]\label{prop:exact_cost}
    For a linear control $\alpha_t = -K(t)X_t$, the expected cost is
    \begin{equation}
        J(0, x_0, \alpha) = \int_0^T \frac{1}{2} K(t)^2 S(t) \d t + \frac{\Gamma}{2} \big( S(T) - 2x_0 m(T) + x_0^2 \big),
    \end{equation}
    
    where $m(t) = \E^{\P^\smalltext{\alpha}}[X_t]$ and $S(t) = \E^{\P^\smalltext{\alpha}}[X_t^2]$ are the first two moments of the state process under the controlled measure $\P^\alpha$, satisfying the {\rm ODEs}
    \begin{equation}
        m^{\prime}(t) = (\bar{a} - \bar{b}K(t))m(t), \; S^{\prime}(t) = 2(\bar{a} - \bar{b}K(t))S(t) + \sigma^2,
    \end{equation}
    with initial conditions $m(0)=x_0$, $S(0)=x_0^2$.
\end{proposition}

\medskip

\begin{proof}
    The state dynamics under the measure $\P^\alpha$ are given by $\d X_t = (\bar{a} - \bar{b}K(t))X_t \d t + \sigma \d W^\alpha_t$. Taking expectations yields the ODE for $m(t)$. Applying It\^o's formula to $X_t^2$ gives $\d (X_t^2) = (2(\bar{a}-\bar{b}K)X_t^2 + \sigma^2)\d t + 2\sigma X_t \d W^\alpha_t$. Taking expectations under $\P^\alpha$ yields the ODE for $S(t)$. Substituting $\E^{\P^\alpha}[\alpha_t^2] = K(t)^2 S(t)$ and expanding the terminal term yields the cost formula.
\end{proof}

\begin{figure}[ht!]
    \centering
    \includegraphics[width=0.5\textwidth]{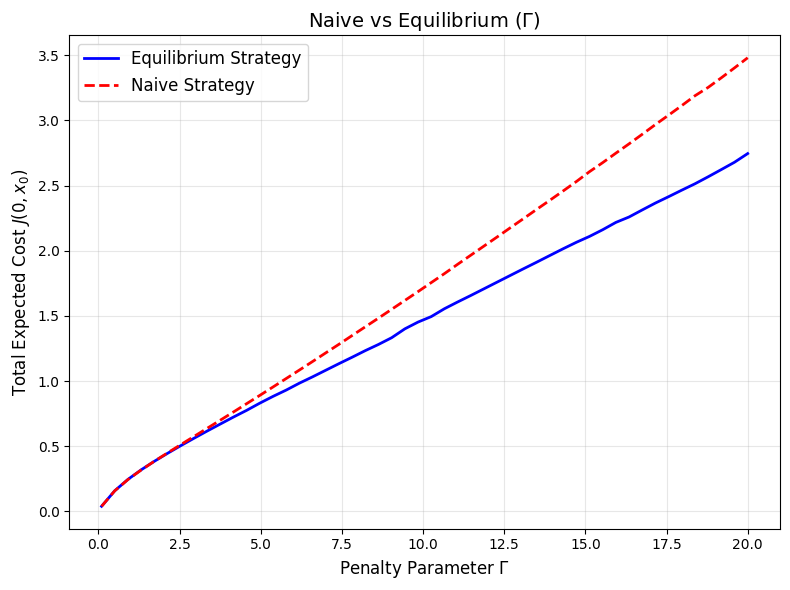}
    \caption{\footnotesize Sensitivity analysis. The total expected cost $J(0, x_0)$ is plotted against the inconsistency parameter $\Gamma$. The equilibrium strategy consistently outperforms the naive strategy as the penalty parameter increases.}
    \label{fig:Gamma_Sensitivity}
\end{figure}

The sensitivity analysis in \Cref{fig:Gamma_Sensitivity}, computed using Proposition \ref{prop:exact_cost}, confirms that the sophisticated strategy yields a strictly lower cost for all $\Gamma > 0$, with the gap widening as $\Gamma$ increases. This is coherent with the intuition that the parameter $\Gamma$ incentivises cooperation between past and future versions of the controllers by increasing the scale of the quadratic penalty.

\section{Time-dependency as a particular case of state-dependence}\label{sec:time-dependency}

The primary focus of this paper has been the dependence of preferences on the current state $x$. However, the vast majority of the literature on time-inconsistent control focuses on a different source of inconsistency: time-dependent preferences. The canonical example is non-exponential discounting (\emph{e.g.}, hyperbolic or quasi-hyperbolic discounting), where the agent's valuation of future rewards depends on the specific time $t$ at which the valuation is made. A natural question arises: is the theory developed here for state-dependent inconsistency compatible with the existing theory for time-dependent inconsistency?

\medskip
In this section, we show that our result is, in fact, a strict generalisation of \cite{hernandez2023me} in the Markovian, uncontrolled volatility case. We achieve this by viewing the initial time $t$ not as an independent parameter, but as a component of the initial state vector. By augmenting the state process, we can cover the time-dependent problem perfectly in our state-dependent framework.

\begin{remark}
    Note that the non-Markovian case is not feasible in our setting, since the presence of the current state in the reward functional compels us to look for feedback strategies that depend on the current state exclusively.
    However, we believe that the extension to controlled volatility should be possible, although technically involved.
\end{remark} 

\subsection{General problem formulation}

Let us consider a reward functional where the running cost $f$ and the terminal cost $\xi$ depend explicitly on the initialisation time $t$. We define the cost functional for an agent initialised at time $t$ with state $x$ as
\begin{align} \label{eq:Hamiltonian Time Dependency}
    \tilde J(t, x, \alpha) \coloneqq  \E^{\mathbb{P}^{\smalltext{t}, \smalltext{x}, \smalltext{\alpha}}}\bigg[\int_{t}^T f(s,   t, X_s, \al_s)\d s + \xi(t, X_T)\bigg].
\end{align}
Here, the distinction between the variable $t$ and the variable $s$ is crucial
\begin{itemize}
    \item $s \in [t, T]$ is the running time, representing the evolution of the system;
    \item $t \in [0, T]$ is the preference parameter, representing the current time from the perspective of the agent.
\end{itemize}
For example, in non-exponential discounting, one might have $f(s,   t, X_s, \alpha_s) = h(s-t) U(X_s, \alpha_s)$, where $h(\cdot)$ is the discount function. The inconsistency arises because the discount factor $h(s-t)$ changes as the initial time $t$ moves forward.

\subsection{The augmented state technique}

To apply the theory from \Cref{sec:mainresults}, we must recast the dependence on the parameter $t$ as a dependence on a state variable. We accomplish this by introducing the augmented state process. Let $\mathbf{X}$ be a process valued in $\R^{n+1}$ defined for $s \in [0,T]$ by
\[
    \mathbf{X}_s \coloneqq  \begin{pmatrix} s \\ X_s \end{pmatrix}.
\]
The dynamics of this augmented process under the control $\alpha$ are given by
\begin{equation}
    \d \mathbf{X}_s = \begin{pmatrix} 1 \\ \sigma(s, X_s)b(s, X_s, \al_s) \end{pmatrix} \d s + \begin{pmatrix} 0_{1 \times d} \\ \sigma(s, X_s) \end{pmatrix} \d W^{\alpha}_s, \;
    \text{initialised at } 
    \mathbf{X}_t = \begin{pmatrix} t \\ x \end{pmatrix} \eqqcolon  \mathbf{x}.
\end{equation}

We can now define the augmented cost functions $\tilde{f}$ and $\tilde{\xi}$ on the augmented space $\R^{n+1} \times \R^{n+1}$ (where the first coordinate represents the time component)
\begin{align*}
    \tilde{f}_s(\mathbf{x}, \mathbf{z}, a) \coloneqq  f(s,   \mathbf{x}_\smallertext{1},  \mathbf{z}_{\smallertext{2}:n\smallertext{+}\smallertext{1}}, a), \;
    \tilde{\xi}(\mathbf{x}, \mathbf{z}) \coloneqq  \xi(\mathbf{x}_\smallertext{1}, \mathbf{z}_{\smallertext{2}:n\smallertext{+}\smallertext{1}}).
\end{align*}
Using this notation, the time-dependent functional \eqref{eq:Hamiltonian Time Dependency} can be rewritten exactly in the form of our state-dependent problem
\begin{equation}
    J(\mathbf{x}, \alpha) = \E^{\P^{\smalltext{\mathbf{x}}, \smalltext{\alpha}}} \bigg[ \int_t^T \tilde{f}_s(\mathbf{x}, \mathbf{X}_s, \al_s) \d s + \tilde{\xi}(\mathbf{x}, \mathbf{X}_T) \bigg].
\end{equation}
This reformulation allows us to apply \Cref{thm:verification} directly. The parameter of the problem is now the vector $\mathbf{x} = (t, x)$.

\subsection{Sanity check: recovering the non-exponential discounting system}

We now demonstrate that applying our general BSDE system to this augmented set-up recovers the specific system derived in \cite{hernandez2023me} for the purely time-dependent case. In the augmented framework, the equilibrium value function $Y_s$ is accompanied by a gradient process $\partial Y^{\mathbf{y}}$. Since the parameter is $\mathbf{y} = (t, x)$, this gradient decomposes into two components:
\[
    \partial Y_s^{\mathbf{y}} = \begin{pmatrix} \partial Y_s^{(t)} \\ \partial Y_s^{(x)} \end{pmatrix}.
\]
Here, $\partial Y^{(t)}$ represents the sensitivity of the value to the initial time (the time-inconsistency term), while $\partial Y^{(x)}$ represents the sensitivity to the initial state (the spatial inconsistency term).

\medskip

Assume the problem's time inconsistency comes purely from the appearance of the present time (as in \cite{hernandez2023me}). This means the preferences depend on $t$, but not on $x$ as a parameter. In other words
\[
    \partial_x f(s,    t, y, a) = 0,\; \text{and} \; \partial_x \xi(t, y) = 0.
\]
Let us examine the BSDE for the gradient component (the second line of \Cref{BSDE Weak Formulation}) applied to our augmented set-up.

\medskip
\emph{$(i)$ The spatial component $(\partial Y^{(x)})$:} since the drivers $\partial_x f$ and $\partial_x \xi$ are zero, the BSDE for the spatial gradient $\partial Y^{(x)}$ becomes a homogeneous linear BSDE with zero terminal condition. By uniqueness, $\partial Y^{(x)}_s \equiv 0$. This aligns with expectation: if preferences do not depend on the initial state $x$, the inconsistency adjustment for $x$ vanishes.

\medskip

\emph{$(ii)$ The inconsistency adjustment:}
Recall that in our general framework, the driver of the BSDE for $Y$ contains the inconsistency adjustment term corresponding to the operator $\Lc^{\alpha^\star}_{(\mathbf{y})}$. For the augmented state $\mathbf{X}$, this is defined as:
\[
    \mathcal{K}_s \coloneqq b^{\mathbf{X}}_s(\mathbf{X}_s, \alpha^\star_s) \cdot \partial_{\mathbf{y}} \mathcal{J}(s, \mathbf{X}_s, \mathbf{X}_s) + \frac{1}{2}\Tr\big[ \Sigma_s(\mathbf{X}_s)\Sigma_s(\mathbf{X}_s)^\top \partial_{\mathbf{yy}}^2 \mathcal{J}(s, \mathbf{X}_s, \mathbf{X}_s) \big] + \Tr\big[ \Sigma_s(\mathbf{X}_s)\Sigma_s(\mathbf{X}_s)^\top \partial_{\mathbf{xy}}^2 \mathcal{J}(s, \mathbf{X}_s, \mathbf{X}_s) \big],
\]
where $\mathbf{y} = (t, x)$ denotes the preference parameter in the augmented set-up, and $\Sigma_s(\mathbf{X}_s)$ is the diffusion matrix of the augmented process. We compute these terms explicitly. The augmented state dynamics $\d \mathbf{X}_s = (1, b)^\top \d s + (0, \sigma)^\top \d W_s$ imply that the coefficients are vectors and matrices in $\R^{n+1}$
\[
    b^{\mathbf{X}} = \begin{pmatrix} 1 \\ \sigma(s, X_s)b(s, X_s, \alpha^\star_s) \end{pmatrix}, \;
    \Sigma = \begin{pmatrix} 0_{1 \times d} \\ \sigma(s, X_s) \end{pmatrix}, \; 
    \Sigma\Sigma^\top = \begin{pmatrix} 0 & 0_{1 \times n} \\ 0_{n \times 1} & \sigma(s, X_s)\sigma(s, X_s)^\top \end{pmatrix}.
\]

Since $\partial_x \mathcal{J} = 0$, the derivatives with respect to $\mathbf{y}$ simplify. The Jacobian $\partial_{\mathbf{y}} \mathcal{J}$ is $(\partial_t \mathcal{J}, 0)^\top$, and the Hessian matrices have zeros in all entries except potentially the top-left (time-time), which does not interact with the non-zero block of $\Sigma\Sigma^\top$. Specifically
\[
    \Tr\bigg[ \begin{pmatrix} 0 & 0 \\ 0 & \sigma^2 \end{pmatrix} \begin{pmatrix} \partial_{tt}^2 \mathcal{J} & 0 \\ 0 & 0 \end{pmatrix} \bigg] = 0.
\]
The mixed derivative term trace is similarly zero. Thus, the total inconsistency adjustment reduces to the drift term:
    \[
        \mathcal{K}_s = b^{\mathbf{X}} \cdot \partial_{\mathbf{y}} \mathcal{J} = \partial_t \mathcal{J}(s, \mathbf{X}_s, \mathbf{X}_s).
    \]
This confirms that the extra drift in the Hamiltonian is exactly the time-derivative of the value function with respect to the initial time. The BSDE for the time-derivative component $ \partial Y^{(t)}_s$ is then obtained directly from our general system \eqref{BSDE Weak Formulation}
\begin{align}\label{eq:time_derivative_BSDE}
    \d \partial Y^{(t)}_s &= - \big( \partial_t f(s,    s, X_s, \alpha^\star_s) + \mathcal{Z}_s \cdot b(s, X_s, \alpha^\star_s) \big) \d s + \mathcal{Z}_s \cdot\d W_s, \; \partial{Y}^{(t)}_T = \partial_t \xi(t, X_T).
\end{align}
This recovers the structure of the adjoint equation derived \cite{hernandez2023me}.

\bibliography{bibliographyDylan}

\begin{appendix}

\section{Proof of the extended dynamic programming principle}\label{Appendix:ddp}

In this section, we provide the detailed proof of \Cref{thm: extendedDPP}. We rely on the definition of equilibrium and the regularity of the value function with respect to the preference parameter. We define the auxiliary value function $\Psi(t, x; y)$ as the expected future reward from state $x$ at time $t$ under a fixed equilibrium strategy $\alpha^\star$, evaluated with the fixed preference parameter $y$
\begin{equation}\label{eq:Psi_def}
    \Psi(t, x; y) \coloneqq  \E^{\P^{\smalltext{t}, \smalltext{x}, \smalltext{\alpha}^\tinytext{\star}}} \bigg[ \int_t^T f(u,    y, X_u, \alpha^\star_u) \diff u + \xi(y, X_T) \bigg].
\end{equation}
By definition, the equilibrium value function corresponds to the diagonal restriction $v(t, x) = \Psi(t, x; x)$. We assume throughout this section that the regularity conditions in \Cref{Assumptions} hold.

\medskip
Before we start, let us introduce a technical lemma from the theory of stochastic calculus that turns out to be the crucial step in understanding the dynamic of the process that we are interested in.
\begin{lemma}[Itô--Kunita--Wentzell’s formula]\label{lemma:iv} Let $f(t, x)$ be a family
of $\F$-adapted and measurable stochastic processes, continuous in $(t, x) \in (\mathbb{R}_+ \times \mathbb{R}^d)$, $\P${\rm--a.s.} satisfying
\begin{enumerate}
    \item[$(i)$] for each $t \geq 0$, $\R^d\ni x \longmapsto f(t, x)\in\R$ is $C^2;$
    \item[$(ii)$] there is some $m\in\N^\star$ such that for each $x\in\R^d$, $f(t, x)$ is a continuous $(\F,\P)$--semi-martingale with
    \[
    \mathrm{d}f(t, x) = \sum_{j= 1}^m f_t^j(x) \mathrm{d}M_t^j,
    \]
    where for any $j\in\{1,\dots,m\}$ $M_j$ is a continuous $(\F,\P)$--semi-martingale, and for any $x\in\R^d$, $f_j(x)$ is an $\F$--adapted and measurable stochastic
processes continuous in $(t, x)$, such that $\R^d\ni x \longmapsto f_j(x)\in\R$ is $C^1$.
\end{enumerate}

Let $X = (X_1, \dots, X_d)$ be a continuous $(\F,\P)$--semi-martingale. Then
\begin{align}\label{Ito Wentzell Formula}
\nonumber f(t, X_t) &= f(0, X_0) + \sum_{j= 1}^m \int_0^t f_s^j(X_s) \mathrm{d}M_s^j + \sum_{i= 1}^d\int_0^t \partial_{x_\smalltext{i}}f(s, X_s) \mathrm{d}X^i_s 
    + \sum_{j= 1}^m\sum_{i= 1}^d \int_{0}^t \partial_{x_\smalltext{i}}f^j_s(   X_s) \mathrm{d} [ X^i, M^j]_s  \\
    &\quad + \frac{1}{2}\sum_{j= 1}^d\sum_{i= 1}^d \int_0^t \partial^2_{x_\smalltext{i}x_\smalltext{j}} f(s,   X_s) \mathrm{d} [ X^i, X^j]_s. 
\end{align}
\end{lemma}

\begin{remark}
    Note that, in particular, the Itô--Kunita--Wentzell's formula says that, with the assumptions of the theorem, the composition of an Itô process and a one-parametric family of Itô processes is again an Itô process, which is not such a trivial statement.
\end{remark}

We readily see that if the process $X$ is a constant $x$, we are left with the original decomposition of the process $f(t, x).$ This version of the theorem was obtained from \citeauthor*{jeanblanc2009mathematical} \cite[Theorem 1.5.3.2]{jeanblanc2009mathematical}, and we present it here without proof, referring to \citeauthor*{kunita1981some} \cite[Theorem 1]{kunita1981some}.

\medskip
Let us move now to the proof of the extended DPP. We will divide most of the work in \Cref{lemma:auxiliaryforverification,lemma:iterated_inequality,lemma:convergence}.

\begin{lemma}\label{lemma:local_comparison}
    Let $(t, x) \in [0, T] \times \R^n$. Let $\tau \in \Tc_{t, T}$ be an $\F$--stopping time bounded by $t + \delta$ for some deterministic constant $\delta > 0$. For any admissible control $\alpha \in \Ac(t, x)$, the following inequality holds
    \begin{align}\label{eq:local_comparison}
        v(t, x) \ge \E^{\P^{\smalltext{t},\smalltext{x}, \smalltext{\alpha}}} \bigg[ v(\tau, X_\tau) + \int_t^\tau f(u,    x, X_u, \alpha_u) \diff u
        + \big( \Psi(\tau, X_\tau; x) - \Psi(\tau, X_\tau; X_\tau) \big) \bigg] - r(\delta),
    \end{align}
    where $r(\delta)$ is a non-negative error term satisfying the asymptotic property $r(\delta) = o(\delta)$ as $\delta \longrightarrow 0$. Furthermore, if $\alpha = \alpha^\star$, the equality holds with $r(\delta) \equiv 0$.
\end{lemma}

\begin{proof}
    We begin by constructing a specific perturbation of the equilibrium strategy. As usual, let $\hat{\alpha} \coloneqq  \alpha \otimes_\tau \alpha^\star$ be the concatenated control defined by
    \[
        \hat{\alpha}_s(\omega) \coloneqq  \alpha_s(\omega)\mathbf{1}_{[t, \tau(\omega))}(s) + \alpha^\star_s(\omega)\mathbf{1}_{[\tau(\omega), T]}(s).
    \]
    This strategy follows the arbitrary control $\alpha$ until the stopping time $\tau$, and reverts to the equilibrium strategy $\alpha^\star$ thereafter.
    
    \medskip
    By \Cref{Def:equilibrium}, the strategy $\alpha^\star$ is optimal against local deviations up to a first-order error. In other words, for small enough $\delta$
    \[
        J(t, x, \alpha^\star) \ge J(t, x, \hat{\alpha}) - o(\delta).
    \]

    The left-hand side is, by the definition of the value function, exactly $v(t, x)$. We now analyze the right-hand side, $J(t, x, \hat{\alpha})$. By the definition of the cost functional, we have
    \[
        J(t, x, \hat{\alpha}) = \E^{\P^{\smalltext{t}, \smalltext{x}, \smalltext{\alpha}}} \bigg[ \int_t^\tau f(u,    x, X_u, \alpha_u) \diff u + \int_\tau^T f(u,    x, X_u, \alpha^\star_u) \diff u + \xi(x, X_T) \bigg].
    \]
    Note that the probability measure $\P^{t, x, \alpha}$ governs the dynamics in $[t, \tau]$, while the dynamics in $(\tau, T]$ is governed by $\alpha^\star$ given the state in $\tau$. We apply the tower property of conditional expectations, conditioning on the $\sigma$-algebra $\Fc_\tau$
    \[
        J(t, x, \hat{\alpha}) = \E^{\P^{\smalltext{t}, \smalltext{x}, \smalltext{\alpha}}} \Bigg[ \int_t^\tau f(u,    x, X_u, \alpha_u) \diff u + \E^{\P^{\smalltext{t}, \smalltext{x}, \smalltext{\hat{\alpha}}}} \bigg[ \int_\tau^T f(u,    x, X_u, \alpha^\star_u) \diff u + \xi(x, X_T) \bigg| \Fc_\tau \bigg] \Bigg].
    \]

    By the properties of the concatenated measure introduced in \Cref{Thm:Concatenated:M}, the conditional distribution of the process after $\tau$ given $\Fc_\tau$ is precisely given by the kernel $\P^{\tau, X_\smalltext{\tau}, \alpha^\smalltext{\star}}$. Consequently, the inner conditional expectation satisfies
    \[
        \E^{\P^{\smalltext{t}, \smalltext{x}, \smalltext{\hat{\alpha}}}} \bigg[ \int_\tau^T f(u,    x, X_u, \alpha^\star_u) \diff u + \xi(x, X_T) \bigg| \Fc_\tau \bigg] = \E^{\P^{\smalltext{\tau}, \smalltext{X}_\tinytext{\tau}, \smalltext{\alpha}^\tinytext{\star}}} \bigg[ \int_\tau^T f(u,    x, X_u, \alpha^\star_u) \diff u + \xi(x, X_T) \bigg].
    \]
    Comparing this to the definition in \eqref{eq:Psi_def}, we identify the right-hand side precisely as the auxiliary value function $\Psi(\tau, X_\tau; x)$.
    Substituting this back into the expansion of $J$, we obtain
    \[
        J(t, x, \hat{\alpha}) = \E^{\P^{\smalltext{t}, \smalltext{x}, \smalltext{\alpha}}} \bigg[ \int_t^\tau f(u,    x, X_u, \alpha_u) \diff u + \Psi(\tau, X_\tau; x) \bigg].
    \]    
    
    Using the initial inequality $v(t, x) \ge J(t, x, \hat{\alpha}) - o(\delta)$, we have
    \[
        v(t, x) \ge \E^{\P^{\smalltext{t}, \smalltext{x}, \smalltext{\alpha}}} \bigg[ \int_t^\tau f(u,    x, X_u, \alpha_u) \diff u + \Psi(\tau, X_\tau; x) \bigg] - o(\delta).
    \]
    Finally, we introduce the equilibrium value function at time $\tau$. Recall that $v(\tau, z) = \Psi(\tau, z; z)$. We add and subtract $v(\tau, X_\tau) = \Psi(\tau, X_\tau; X_\tau)$ inside the expectation
    \[
        \Psi(\tau, X_\tau; x) = v(\tau, X_\tau) + \big( \Psi(\tau, X_\tau; x) - \Psi(\tau, X_\tau; X_\tau) \big).
    \]
    Plugging this decomposition into the inequality yields the result \eqref{eq:local_comparison}.
\end{proof}

\begin{lemma}\label{lemma:iterated_inequality}
    Fix a time horizon $S > t$ and $N\in\N^\star$. Let $\Pi_\smallertext{N} \coloneqq \{t_0, t_1, \dots, t_\smallertext{N}\}$ be a partition of the interval $[t, S]$, where $t_0 = t$ and $t_\smallertext{N} = S$. For any admissible control $\alpha$
     \begin{align}\label{eq:iterated_sum}
        v(t, x) &\ge \E^{\P^\smalltext{\alpha}} \Bigg[ v(S, X_S) + \sum_{i=0}^{\smallertext{N}\smallertext{-}1} \int_{t_\smalltext{i}}^{t_{\smalltext{i}\smalltext{+}\smalltext{1}}} f(u,    X_{t_\smalltext{i}}, X_u, \alpha_u) \diff u + \sum_{i=0}^{\smallertext{N}\smallertext{-}1} \Big( \Psi\big(t_{i+1}, X_{t_{\smalltext{i}\smalltext{+}\smalltext{1}}}; X_{t_\smalltext{i}}) - \Psi(t_{i+1}, X_{t_{\smalltext{i}\smalltext{+}\smalltext{1}}}; X_{t_{\smalltext{i}\smalltext{+}\smalltext{1}}}) \Big) \Bigg] - o(1).
    \end{align}
\end{lemma}

\begin{proof}

We proceed by backward induction or simple iteration. Consider the interval $[t_i, t_{i+1}]$. We apply \Cref{lemma:local_comparison} conditioned on the filtration $\Fc_{t_\smalltext{i}}$, with the preference parameter frozen at the state $X_{t_\smalltext{i}}$. This gives
    \[
        v(t_i, X_{t_\smalltext{i}}) \ge \E^{\P^\smalltext{\alpha}} \bigg[ v(t_{i+1}, X_{t_{\smalltext{i}\smalltext{+}\smalltext{1}}}) + \int_{t_\smalltext{i}}^{t_{\smalltext{i}\smalltext{+}\smalltext{1}}} f(u,    X_{t_\smalltext{i}}, X_u, \alpha_u) \diff u + \Delta_{i} \bigg| \Fc_{t_i} \bigg] - o(t_{i+1}-t_i),
    \]
    where $\Delta_i \coloneqq \Psi(t_{i+1}, X_{t_{\smalltext{i}\smalltext{+}\smalltext{1}}}; X_{t_\smalltext{i}}) - \Psi(t_{i+1}, X_{t_{\smalltext{i}\smalltext{+}\smalltext{1}}}; X_{t_{\smalltext{i}\smalltext{+}\smalltext{1}}})$. Taking expectations under $\P^\alpha$ and summing these inequalities from $i=0$ to $N-1$ leads to a telescoping sum for the value function terms $v(t_i, X_{t_\smalltext{i}})$, leaving only the initial term $v(t, x)$ and the terminal term $v(S, X_S)$, plus the cumulative sums of the running costs and the adjustment terms $\Delta_i$.
\end{proof}

Now we conclude the proof of the extended dynamic programming principle in \Cref{thm: extendedDPP} by showing that the sums in \Cref{lemma:iterated_inequality} converge to the terms we expect.
\begin{lemma}[Convergence of the discrete inequality]\label{lemma:convergence}
    Let $(\Pi_N)_{N\in\N^\smalltext{\star}}$ be a sequence of partitions of $[t, S]$ whose mesh size tends to $0$. The discrete inequality in {\rm\Cref{lemma:iterated_inequality}} converges to the following integral formulation
    \begin{align*}
        v(t, x) &\ge \E^{\P^\smalltext{\alpha}} \bigg[ v(S, X_\smallertext{S}) + \int_t^\smallertext{S} f(u,    X_u, X_u, \alpha_u) \diff u - \int_t^\smallertext{S} \bigg( b(u, X_u, \alpha_u) \cdot \partial_y \Psi(u, X_u; X_u) \\ &\quad + \frac{1}{2}\Tr\big[\sigma(u, X_u) \sigma^\top(u, X_u) \partial_{yy}^2 \Psi(u, X_u; X_u)\big] 
         + \Tr\big[\sigma(u, X_u) \sigma^\top(u, X_u) \partial_{xy}^2 \Psi(u, X_u; X_u)\big] \bigg) \diff u \Bigg].
    \end{align*}
\end{lemma}

\begin{proof}
    To rigorously analyse the convergence of the discrete sums appearing in \Cref{lemma:iterated_inequality}, we introduce the time-discretisation map $\tau_\smallertext{N}: [t, S] \longrightarrow \{t_0, \dots, t_{\smallertext{N}\smallertext{-}1}\}$ defined by $\tau_N(u) \coloneqq t_i$ for $u \in [t_i, t_{i+1})$. This notation allows us to express the discrete Riemann sums as continuous stochastic integrals over the full interval $[t, S]$, facilitating the use of dominated convergence arguments.

    \medskip

    \textit{Part $1$: convergence of the running cost.}
    We consider the Riemann sum approximating the running cost
    \[
        I^{\Pi_\smalltext{N}} \coloneqq \sum_{i=0}^{\smallertext{N}\smallertext{-}1} \int_{t_\smalltext{i}}^{t_{\smalltext{i}\smalltext{+}\smalltext{1}}} f(u,    X_{t_i}, X_u, \alpha_u) \diff u.
    \]
    Using the discretisation map $\tau_\smallertext{N}$, we rewrite this sum as a single global integral
    \[
        I^{\Pi_\smalltext{N}} = \int_t^\smallertext{S} f(u,    X_{\tau_\smalltext{N}(u)}, X_u, \alpha_u) \diff u.
    \]
    We claim that $I^{\Pi_\smalltext{N}}$ converges to $\int_t^\smallertext{S} f(u,    X_u, X_u, \alpha_u) \diff u$ in $\L^1(\R,\Fc,\P^\alpha)$. Indeed, we can apply the dominated convergence theorem under the measure $\P^\alpha$
    \begin{enumerate}
        \item \textit{pointwise convergence:} the trajectories of $X$ are continuous $\P$--a.s. (and thus $\P^\alpha$--a.s.). As the mesh size $|\Pi_\smallertext{N}| \longrightarrow 0$, we have $\tau_\smallertext{N}(u) \longrightarrow u$, implying $X_{\tau_\smalltext{N}(u)} \longrightarrow X_u$ for all $u$. Since $f$ is continuous in its arguments, the integrand $f(u,    X_{\tau_\smalltext{N}(u)}, X_u, \alpha_u)$ converges pointwise to $f(u,    X_u, X_u, \alpha_u)$ for $\mathrm{d}t \otimes \P^\alpha$--almost every $(u, \omega)$;

\medskip
        \item \textit{domination:} we seek a uniform integrable bound. By the polynomial growth assumption on $f$ (\Cref{Assumptions}), there exist constants $C > 0$ and $m \ge 1$ such that for all $u \in [t, S]$
        \[
             \big|f(u,    X_{\tau_\smalltext{N}(u)}, X_u, \alpha_u)\big| \le C\big(1 + \|X_{\tau_\smalltext{N}(u)}\|^m + \|X_u\|^m\big) \le 2C\bigg(1 + \sup_{s \in [t, S]} \|X_s\|^m\bigg) \eqqcolon \mathcal{Z}.
        \]

 Since we have $\mathcal{Z} \in \L^1(\R,\Fc,\P^\alpha)$ due to \Cref{Assumptions}, we can conclude.

    \end{enumerate}

 \medskip
    \textit{Part $2$: convergence of the adjustment term.}
    We now turn to the inconsistency adjustment sum
    \[
        \mathcal{A}^{\Pi_N} \coloneqq \sum_{i=0}^{N-1} \Delta_i, \qquad \Delta_i \coloneqq \Psi(t_{i+1}, X_{t_{i+1}}; X_{t_i}) - \Psi(t_{i+1}, X_{t_{i+1}}; X_{t_{i+1}}).
    \]
    All expectations in the following are taken under $\P^{\alpha}$, the measure induced by the arbitrary control $\alpha$. Fix a partition interval $[t_i, t_{i+1}]$ and decompose
    \[
        \Delta_i = \underbrace{ \Psi(t_{i+1},X_{t_{i+1}};X_{t_i}) - \Psi(t_i,X_{t_i};X_{t_i}) }_{\text{Term I}} - \underbrace{ \big(v(t_{i+1},X_{t_{i+1}}) - v(t_i,X_{t_i})\big) }_{\text{Term II}}.
    \]

    \medskip
    \noindent\textit{Term I.}
    Apply It\^o's formula to $r \longmapsto \Psi(r, X_r; X_{t_i})$ under $\P^{\alpha}$, holding the preference parameter $X_{t_i}$ fixed. Writing $\mathcal{L}^{\alpha_r}_r = \mathcal{L}^{\alpha^\star_r}_r + (\mathcal{L}^{\alpha_r}_r - \mathcal{L}^{\alpha^\star_r}_r)$ and using the PDE $(\partial_t + \mathcal{L}^{\alpha^\star_r}_r)\Psi(\cdot, \cdot; y) = -f(\cdot, y, \cdot, \alpha^\star_r)$, we obtain
    \begin{align*}
        \Psi(t_{i+1},X_{t_{i+1}};X_{t_i}) - \Psi(t_i,X_{t_i};X_{t_i}) &= \int_{t_i}^{t_{i+1}} \Big[ -f(r, X_{t_i}, X_r, \alpha^\star_r) \\
        &\quad + \big(b(r,X_r,\alpha_r) - b(r,X_r,\alpha^\star_r)\big) \cdot \sigma(r,X_r)^{\top} \partial_x \Psi(r,X_r;X_{t_i}) \Big] \d r + M^{(i),\mathrm{I}},
    \end{align*}
    where $M^{(i),\mathrm{I}}$ is a stochastic integral against $W^{\alpha}$ and hence a true $\P^{\alpha}$--martingale by the polynomial growth of $\partial_x\Psi$ and \Cref{Assumptions}. Taking the conditional expectation $\E^{\P^{\alpha}}[\, \cdot \mid \Fc_{t_i}]$ eliminates $M^{(i),\mathrm{I}}$.

    \medskip
    \noindent\textit{Term II.}
    Since $v(r,x) = \Psi(r,x;x)$, applying the chain rule for spatial derivatives gives
    \[
        \partial_x v(r,x) = \partial_x\Psi(r,x;x) + \partial_y\Psi(r,x;x), \qquad \partial_{xx}^2 v(r,x) = \partial_{xx}^2\Psi(r,x;x) + 2\partial_{xy}^2\Psi(r,x;x) + \partial_{yy}^2\Psi(r,x;x).
    \]
    Applying \Cref{lemma:iv} (It\^o--Kunita--Wentzell formula) to $v(r,X_r)$ under $\P^{\alpha}$, substituting these identities, and using the PDE for $\Psi$ to simplify $\partial_t\Psi + \mathcal{L}^{\alpha_r}_r\Psi = -f(r,X_r,X_r,\alpha^\star_r) + (b(r,X_r,\alpha_r) - b(r,X_r,\alpha^\star_r))\cdot\sigma(r,X_r)^{\top}\partial_x\Psi(r,X_r;X_r)$, we find
    \begin{align*}
        \E^{\P^{\alpha}}\big[v(t_{i+1},X_{t_{i+1}}) - v(t_i,X_{t_i}) \mid \Fc_{t_i} \big] &= \E^{\P^{\alpha}} \bigg[ \int_{t_i}^{t_{i+1}} \Big[ -f(r,X_r,X_r,\alpha^\star_r) \\
        &\quad + \big(b(r,X_r,\alpha_r)-b(r,X_r,\alpha^\star_r)\big) \cdot\sigma(r,X_r)^{\top}\partial_x\Psi(r,X_r;X_r) \\
        &\quad + \mathcal{L}^{\alpha_r}_{r,(y)}\Psi(r,X_r;X_r) \Big] \d r \bigg| \Fc_{t_i} \bigg],
    \end{align*}
    where we define the generator acting exclusively on the $y$-variable under the \emph{arbitrary} control $\alpha$ as:
    \[
        \mathcal{L}^{\alpha_r}_{r,(y)}\Psi \coloneqq b(r,X_r,\alpha_r)\cdot\sigma(r,X_r)^{\top}\partial_y\Psi + \frac{1}{2}\Tr\big[\sigma(r,X_r)\sigma(r,X_r)^{\top} \partial_{yy}^2\Psi\big] + \Tr\big[\sigma(r,X_r)\sigma(r,X_r)^{\top} \partial_{xy}^2\Psi\big].
    \]

    \medskip
    \noindent\textit{Combining.}
    Subtracting Term II from Term I, taking the unconditional expectation $\E^{\P^{\alpha}}$, summing over $i$, and rewriting the result as a single integral via the discretisation map $\tau_N$ yields:
    \begin{align*}
        \E^{\P^{\alpha}}\big[\mathcal{A}^{\Pi_N}\big] &= \E^{\P^{\alpha}} \bigg[ \int_t^S \Big[ f(r,X_r,X_r,\alpha^\star_r) - f(r,X_{\tau_N(r)},X_r,\alpha^\star_r) \\
        &\quad + \big(b(r,X_r,\alpha_r) - b(r,X_r,\alpha^\star_r)\big) \cdot\sigma(r,X_r)^{\top} \big( \partial_x\Psi(r,X_r;X_{\tau_N(r)}) - \partial_x\Psi(r,X_r;X_r) \big) \\
        &\quad - \mathcal{L}^{\alpha_r}_{r,(y)}\Psi(r,X_r;X_r) \Big] \d r \bigg].
    \end{align*}

    \medskip
    \noindent\textit{Passage to the limit.}
    As $|\Pi_N|\to 0$, we have $X_{\tau_N(r)}\to X_r$, $\P$--a.s.\ by the continuity of the trajectories. By the continuity of $f$ and $\partial_x\Psi$ in all their arguments, the first two lines of the integrand converge pointwise to zero. Specifically:
    \begin{itemize}
        \item $f(r,X_r,X_r,\alpha^\star_r) - f(r,X_{\tau_N(r)},X_r,\alpha^\star_r) \longrightarrow 0$ pointwise;
        \item $\partial_x\Psi(r,X_r;X_{\tau_N(r)}) - \partial_x\Psi(r,X_r;X_r) \longrightarrow 0$ pointwise. Since for any fixed $(r, \omega)$ the evaluated state and controls are finite, the drift difference $(b^\alpha_r - b^{\alpha^\star}_r)$ acts as a finite multiplier, guaranteeing that the entire cross-term pointwise converges to zero.
    \end{itemize}
    Both terms are uniformly dominated by the integrable random variable $\mathcal{Z}$ constructed in Part 1 (scaled by constants depending on the Lipschitz continuity of $b$ and the polynomial growth of $\partial_x\Psi$ from \Cref{Assumptions}). Applying the dominated convergence theorem under $\P^{\alpha}$ gives:
    \[
        \lim_{|\Pi_N|\to 0} \E^{\P^{\alpha}}\big[\mathcal{A}^{\Pi_N}\big] = -\E^{\P^{\alpha}} \bigg[ \int_t^S \mathcal{L}^{\alpha_r}_{r,(y)}\Psi(r,X_r;X_r) \d r \bigg].
    \]
    Substituting the explicit form of $\mathcal{L}^{\alpha_r}_{r,(y)}\Psi$ and combining with Part 1 yields the integral inequality stated in the lemma.

    \medskip

\end{proof}

With these lemmata, we can finally conclude the proof of the extended dynamic programming principle.

\begin{proof}[Proof of Theorem \ref{thm: extendedDPP}]
   \Cref{lemma:convergence} establishes that for any admissible control $\alpha \in \Ac$, the value function satisfies the integral inequality
    \begin{align*}
        v(t, x) &\ge \E^{\P^\smalltext{\alpha}} \bigg[ v(S, X_\smallertext{S}) + \int_t^\smallertext{S} f(u,    X_u, X_u, \alpha_u) \d u - \int_t^\smallertext{S} \bigg( b(u, X_u, \alpha_u) \cdot \nabla_y \Psi(u, X_u; X_u) \\ &\quad + \frac{1}{2}\Tr\big[\sigma(u, X_u) \sigma^\top(u, X_u) \nabla_{yy}^2 \Psi(u, X_u; X_u)\big] 
         + \Tr\big[\sigma(u, X_u) \sigma^\top(u, X_u) \nabla_{xy}^2 \Psi(u, X_u; X_u)\big] \bigg) \d u \Bigg].
    \end{align*}
    To conclude the proof, we must show that equality holds when $\alpha = \alpha^\star$. Recall from \Cref{lemma:local_comparison} that if we choose the equilibrium control $\alpha^\star$, the local error term $r(\delta)$ is identically zero. This implies that the discrete-time inequality becomes an equality at every step of the iteration in \Cref{lemma:iterated_inequality}. Specifically, for $\alpha = \alpha^\star$, the telescoping sum argument holds exactly without any $o(1)$ error terms.
    
\medskip
    Consequently, passing to the limit as the mesh size $|\Pi_\smallertext{N}| \longrightarrow 0$ in the equality case proceeds identically to the inequality case, but with equalities throughout. Thus

\begin{align*}
        v(t, x) &= \E^{\P^{\smalltext{t},\smalltext{x},\smalltext{\alpha}^\tinytext{\star}}} \Bigg[ v(S, X_\smallertext{S}) + \int_t^\smallertext{S} f(u,    X_u, X_u, \alpha^\star_u) \d u - \int_t^\smallertext{S} \bigg( b_u( X_u, \alpha^\star_u) \cdot \nabla_y \Psi(u, X_u; X_u) \\
        &\quad + \frac{1}{2}\Tr\big[\sigma_u \sigma_u^\top \nabla_{yy}^2 \Psi(u, X_u; X_u)\big] + \Tr\big[\sigma_u \sigma_u^\top \nabla_{xy}^2 \Psi(u, X_u; X_u)\big] \bigg)\d u \Bigg].
    \end{align*}

    Finally, we recall the definition of the auxiliary function $\Psi(u, x; y)$ as the expected reward with fixed parameter $y$. Differentiating under the expectation sign (justified by \Cref{Assumptions}), we observe that the derivatives $\nabla_y \Psi$, $\nabla_{yy}^2 \Psi$, and $\nabla_{xy}^2 \Psi$ evaluated at $(u, X_u; X_u)$ correspond exactly to the expectation terms appearing in the theorem statement \eqref{eq:extended_dpp}, thereby concluding the proof.

\end{proof}

\section{Proof of the necessity theorem}\label{Appendix:Necessity}

Before proving the main necessity result, we establish the following consequence of the extended dynamic programming principle. 

\begin{lemma}[Martingale optimality property]\label{lemma:martingale_property}
    Let $\alpha^\star \in \Ac$ be an equilibrium control satisfying the extended {\rm DPP} identity \eqref{eq:extended_dpp}. Define the inconsistency adjustment term $\mathcal{K}_t(a)$ for any $a \in A$ by

\[
        \mathcal{K}_t(a) \coloneqq b(t, X_t, a) \cdot \sigma(t, X_t)^\top \nabla_y \mathcal{J}(t, X_t, X_t) + \Tr\bigg[ \bigg(\frac{1}{2}\nabla_{yy}^2 \mathcal{J}(t, X_t, X_t) + \nabla_{xy}^2 \mathcal{J}(t, X_t, X_t)\bigg)\sigma(t, X_t)\sigma(t, X_t)^\top \bigg],\; t\in[0,T].
    \]
    Then, the process $M^{\alpha^\smalltext{\star}}$ defined by
    \[
        M_t^{\alpha^\smalltext{\star}} \coloneqq v(t, X_t) + \int_0^t \big( f(r,  X_r, X_r, \alpha^\star_r) - \mathcal{K}_r(\alpha^\star_r) \big) \d r, \; t \in [0, T],
    \]
    is an $(\F,\P^{\alpha^\smalltext{\star}})$-martingale. Furthermore, for any arbitrary admissible control $\alpha \in \Ac$, the corresponding process $M^\alpha$ is a $(\F,\P^\alpha)$--super-martingale.
\end{lemma}

\begin{proof}
   We prove the martingale property for $\alpha^\star$. Fix $0 \le s \le t \le T$.
        We compute the conditional expectation of the increment
        \begin{align*}
            \E^{\P^{\alpha^\star}} \big[ M_t^{\alpha^\star} - M_s^{\alpha^\star} \mid \Fc_s \big] &= \E^{\P^{\alpha^\star}} \left[ v(t, X_t) - v(s, X_s) + \int_s^t \Big( f(r,  X_r, X_r, \alpha^\star_r) - \mathcal{K}_r(\alpha^\star_r) \Big) \d r \bigg| \Fc_s \right].
        \end{align*}
        By the Markov property of the state process $X$ and the feedback nature of $\alpha^\star$, we can rewrite the conditional expectation using the expectation starting at time $s$
        \[
            \E^{\P^{\alpha^\smalltext{\star}}} \big[ M_t^{\alpha^\smalltext{\star}} - M_s^{\alpha^\smalltext{\star}} \big| \Fc_s \big] = \E^{\P^{\smalltext{s}\smalltext{,} \smalltext{X}_\tinytext{s}\smalltext{,} \smalltext{\alpha}^\tinytext{\star}}} \bigg[ v(t, X_t) + \int_s^t \big( f(r,  X_r, X_r, \alpha^\star_r) - \mathcal{K}_r(\alpha^\star_r) \Big) \d r \bigg] - v(s, X_s).
        \]

 We now compare this expression with the extended DPP \eqref{eq:extended_dpp}.
        Observe that the expectation terms appearing in \eqref{eq:extended_dpp} are taken under the measure $\P^{r, X_\smalltext{r}, \alpha^\smalltext{\star}}$. 
        These terms correspond precisely to the derivatives $\nabla_y \mathcal{J}(r, X_r, X_r)$, $\nabla_{yy}^2 \mathcal{J}(r, X_r, X_r)$, and $\nabla_{xy}^2 \mathcal{J}(r, X_r, X_r)$ appearing in our definition of $\mathcal{K}_r(\alpha^\star_r)$.
        Consequently, the integral term involving $\mathcal{K}_r$ exactly cancels the inconsistency cost terms in the extended DPP, leaving the martingale difference equal to zero.

\end{proof}

With this in mind, we go on to provide a rigorous proof of \Cref{thm:necessity}. We assume the existence of a smooth equilibrium control $\alpha^\star$ and smooth value functions $V$ and $\mathcal{J}$, and we show that they necessarily induce a solution to the BSDE system \eqref{BSDE Weak Formulation} and satisfy the Hamiltonian maximisation condition.

\begin{proof}[Proof of Theorem \ref{thm:necessity}]
    The proof proceeds in three steps: first, we identify the auxiliary processes for the parameter derivatives; second, we derive the dynamics of the value function using the extended DPP; and third, we verify the Hamiltonian maximisation condition.

    \medskip
    Let us start by showing that the derivative processes satisfy \eqref{BSDE Weak Formulation}. Recall the definition of the auxiliary value function with fixed preference parameter $y \in \R^n$
    \[
        \mathcal{J}(t, x, y) \coloneqq \E^{\P^{\smalltext{t}\smalltext{,} \smalltext{x}\smalltext{,} \smalltext{\alpha}^\tinytext{\star}}}\bigg[\int_t^T f(s,   y, X_s, \alpha^\star_s) \d s + \xi(y, X_T)\bigg].
    \]
   By the classical Feynman--Kac theorem, for each fixed $y$, the function $(t, x) \longmapsto \mathcal{J}(t, x, y)$ solves the linear PDE
    \begin{equation}\label{eq:PDE_J_fixed_y}
        \partial_t \mathcal{J}(t, x, y) + \mathcal{L}_t^{\alpha^\star(t, x)} \mathcal{J}(t, x, y) + f(t, y, x, \alpha^\star(t, x)) = 0, \; (t,x,y)\in[0,T)\times\R^n\times\R^n,
    \end{equation}
    with terminal condition $\mathcal{J}(T, x, y) = \xi(y, x)$.
    By the hypothesis of \Cref{thm:necessity}, $\mathcal{J}$ is of class $C^{1,2}$ with respect to the spatial and parameter variables. We can therefore differentiate \eqref{eq:PDE_J_fixed_y} with respect to the parameter $y$. Note that the derivatives of the cost functions $\partial_y f$ and $\partial_y \xi$ exist by \Cref{Assumptions}. Let $v_y(t, x) \coloneqq \partial_y \mathcal{J}(t, x, y)$ denote the gradient with respect to $y$. It satisfies the linearised PDE
    \[
        \partial_t v_y(t,x) + \Lc_t^{\alpha^\smalltext{\star}(t,x)} v_y(t,x) + \partial_y f(t, y, x, \alpha^\star(t, x)) = 0, \; (t,x,y)\in[0,T)\times\R^n\times\R^n,\; v_y(T, x) = \partial_y \xi(y, x),\; (x,y)\in\R^n\times\R^n.
    \]
    
    This is a standard linear parabolic equation. The probabilistic representation of its solution is given by the BSDE
    \begin{align*}
        \partial Y^y_t &= \partial_y \xi(y, X_T) + \int_t^T \partial_y f(r,  y, X_r, \alpha^\star_r) \d r - \int_t^T \partial Z^y_r\cdot \d W^{\alpha^\star}_r,\; t\in[0,T],
    \end{align*}
    where we identify $\partial Y^y_t = \partial_y \mathcal{J}(t, X_t, y)$ and $\partial Z^y_t = \sigma(t, X_t)^\top \partial_{xy}^2 \mathcal{J}(t, X_t, y)$. We also let $\alpha^\star_t\coloneqq \alpha^\star(t,X_t)$, abusing notations slightly.

\medskip
    However, the system \eqref{BSDE Weak Formulation} is written under the reference measure $\P$ (where $W$ is an $(F,\P)$--Brownian motion), not $\P^{\alpha^\smalltext{\star}}$. Recall that $\d W^{\alpha^\smalltext{\star}}_r = \d W_r - b(r, X_r, \alpha^\star_r) \d r$. Substituting this change of measure into the equation above yields
    \begin{align*}
        \partial Y^y_t &= \partial_y \xi(y, X_T) + \int_t^T \big( \partial_y f(r,  y, X_r, \alpha^\star_r) + \partial Z^y_r \cdot b(r, X_r, \alpha^\star_r) \big) \d r - \int_t^T \partial Z^y_r\cdot \d W_r.
    \end{align*}
    This matches exactly the second equation of the system \eqref{BSDE Weak Formulation}. The derivation for the Hessian process $\partial\partial Y^y$  follows an identical argument by differentiating the PDE twice.

    \medskip
    Let us now address the dynamics of the process $Y$.We start by determining its driver, keeping in mind that $Y_t \coloneqq V(t, X_t)$. By It\^o's formula
    \[
        \d Y_t = \big( \partial_t V(t, X_t) + \mathcal{L}_t^{\alpha^\smalltext{\star}_\smalltext{t}} V(t, X_t) \big) \d t + \nabla_x V(t, X_t)^\top \sigma(t, X_t) \d W^{\alpha^\smalltext{\star}}_t.
    \]
    \medskip
    
    To identify the drift term $\partial_t V + \Lc^{\alpha^\smalltext{\star}} V$, we use the extended DPP (\Cref{thm: extendedDPP}). Since $\alpha^\star$ is an equilibrium control, \Cref{lemma:martingale_property} implies that the process
    \[
        M_t \coloneqq V(t, X_t) + \int_0^t \big( f(r,  X_r, X_r, \alpha^\star_r) - \mathcal{K}_r(\alpha^\star_r) \big) \d r,\; t\in[0,T],
    \]
      is an $(\F,\P^{\alpha^\smalltext{\star}})$-martingale. Thus, the drift of $M$ must vanish. Calculating it and setting it to zero gives
    \[
        \underbrace{\partial_t V(t, X_t) + \mathcal{L}_t^{\alpha^\smalltext{\star}_\smalltext{t}} V(t, X_t)}_{\text{Drift of } V} + \underbrace{f(t, X_t, X_t, \alpha^\star_t) - \mathcal{K}_t(\alpha^\star_t)}_{\text{Drift from integral}} = 0,\; \diff t\otimes\P\text{\rm--a.e.}
    \]
    Therefore, the generator of the value function is given by
    \begin{equation}\label{eq:generator_V_identified}
        \partial_t V(t, X_t) + \mathcal{L}_t^{\alpha^\smalltext{\star}_\smalltext{t}} V(t, X_t) = - f(t, X_t, X_t, \alpha^\star_t) + \mathcal{K}_t(\alpha^\star_t),\; \diff t\otimes\P\text{\rm--a.e.}
    \end{equation}
    We now define the BSDE variables for the value function. Let $Z_t \coloneqq \sigma(t, X_t)^\top \nabla_x V(t, X_t)$, $t\in[0,T]$. Under the reference measure $\P$, the dynamics of $Y$ is
    \[
        \d Y_t = \big( \partial_t V(t, X_t) + \mathcal{L}_t^{\alpha^\smalltext{\star}_\smalltext{t}} V(t, X_t) - Z_t \cdot b(t, X_t, \alpha^\star_t) \big) \d t + Z_t \cdot\d W_t.
    \]

    Substituting the generator expression from \eqref{eq:generator_V_identified} and expanding $\mathcal{K}_t$
    \begin{align*}
        \d Y_t &= \bigg(  b(t, X_t, \alpha^\star_t) \cdot \sigma(t, X_t)^\top \nabla_y \mathcal{J}(t, X_t, X_t) -f(t, X_t, X_t, \alpha^\star_t)+ \frac12\Tr\big[ \sigma(t, X_t)\sigma(t, X_t)^\top \nabla_{yy}^2 \mathcal{J}(t, X_t, X_t)\big]\\
        &\quad + \Tr\big[\sigma(t, X_t)\sigma(t, X_t)^\top \nabla_{xy}^2 \mathcal{J}(t, X_t, X_t) \bigg]- Z_t \cdot b(t, X_t, \alpha^\star_t) \bigg) \d t   + Z_t \cdot\d W_t.
    \end{align*}

    We identify the terms with the BSDE variables defined above
    \[
        \partial Y^{\smallertext{X}_\smalltext{t}}_t = \nabla_y \mathcal{J}(t, X_t, X_t), \; 
        \partial \partial Y^{\smallertext{X}_\smalltext{t}}_t = \nabla_{yy}^2 \mathcal{J}(t, X_t, X_t), \; 
        \partial Z^{\smallertext{X}_\smalltext{t}}_t = \sigma(t, X_t)^\top \nabla_{xy}^2 \mathcal{J}(t, X_t, X_t), \; t\in[0,T].
    \]
      The driver becomes
    \[
         \text{driver}_t = f(t, X_t, X_t, \alpha^\star_t) + b(t, X_t, \alpha^\star_t) \cdot (Z_t - \sigma(t, X_t)^\top \partial Y^{\smallertext{X}_\smalltext{t}}_t) - \frac{1}{2}\Tr\big[\sigma(t, X_t)\sigma(t, X_t)^\top \partial \partial Y^{\smallertext{X}_\smalltext{t}}_t\big] - \Tr\big[\sigma(t, X_t) \partial Z^{\smallertext{X}_\smalltext{t}}_t\big].
    \]
    
     This matches the drift of the first equation of \eqref{BSDE Weak Formulation}, provided that $\alpha^\star$ maximises the Hamiltonian, which is what we are left to prove.

     \medskip
    To do so, we compare the dynamics of the equilibrium value function under $\alpha^\star$ versus an arbitrary control $\alpha$. Since we know $M^{\alpha^{\smalltext{\star}}}$ is an $(\F,\P^{\alpha^\smalltext{\star}})$-martingale, we have that its drift is exactly zero
    \begin{equation}\label{eq:drift_equality}
        \partial_t V(t, X_t) + \Lc_t^{\alpha^\smalltext{\star}_\smalltext{t}} V(t, X_t) + f(t, X_t, X_t, \alpha^\star) - \mathcal{K}_t(\alpha^\star_t) = 0.
    \end{equation}

    Using \Cref{lemma:martingale_property} we also have that its drift must be non-positive
    \begin{equation}\label{eq:drift_inequality}
        \partial_t V(t, X_t) + \Lc_t^{\alpha_\smalltext{t}} V(t, X_t) + f(t, X_t, X_t, \alpha) - \mathcal{K}_t(\alpha_t) \le 0.
    \end{equation}

\medskip

    Now we simply we subtract the equality \eqref{eq:drift_equality} from the inequality \eqref{eq:drift_inequality}. Note that terms not depending on the control cancel out immediately
    \begin{itemize}
        \item the time derivative $\partial_t V$ cancels;
        \item the second-order diffusion term in $\Lc^{\alpha^\smalltext{\star}_\smalltext{t}}$ and $\Lc^{\alpha_\smalltext{t}}$ involves $\frac{1}{2}\sigma(t,x)\sigma(t, x)^\top \nabla_{xx}^2 V$. Since volatility is uncontrolled, this term is identical for both $\alpha$ and $\alpha^\star$ and cancels;
        \item the second-order trace term inside the inconsistency adjustment $\mathcal{K}_t$ (see \Cref{lemma:martingale_property}) also depends only on $\sigma(x)$ (see \Cref{remark:necessity_clarification}). It is identical in both equations and also cancels.
    \end{itemize}

  We are left with the first-order terms
    \begin{align*}
        &\big( b(t, X_t, \alpha_t) \cdot \sigma(t, X_t)^\top \nabla_x V + f(t,X_t,X_t,\alpha_t) - b(t, X_t, \alpha_t) \cdot \sigma(t, X_t)^\top \partial_y \mathcal{J}(t, X_t, X_t) \big)\\
- &\big( b(t, X_t, \alpha^\star_t) \cdot \sigma(t, X_t)^\top \nabla_x V + f(t,X_t,X_t,\alpha^\star_t) - b(t, X_t, \alpha^\star_t) \cdot \sigma(t, X_t)^\top \partial_y \mathcal{J}(t, X_t, X_t) \big) \le 0.
    \end{align*}
    Rearranging this inequality to isolate the terms dependent on $a$ and identifying $\sigma^{-1} Z_t = \nabla_x V$ and $\partial Y_t^{\smallertext{X}_\smalltext{t}} = \nabla_y \mathcal{J}(t, X_t, X_t)$, we obtain:

    \[f(t, X_t, X_t, \alpha_t) + b(t, X_t, \alpha_t) \cdot \big( Z_t - \sigma(t, X_t)^\top \partial Y^{\smallertext{X}_\smalltext{t}}_t \big)
\le
f(t, X_t, X_t, \alpha^\star_t) + b(t, X_t, \alpha^\star_t) \cdot \big( Z_t - \sigma(t, X_t)^\top \partial Y^{\smallertext{X}_\smalltext{t}}_t \big).\]
    Since this holds for any arbitrary admissible control $\alpha_t$, it implies that $\alpha^\star_t$ maximises the expression $\d t \otimes \d \P$-a.e., concluding the proof.
\end{proof}

\section{Proof of the verification theorem}\label{Appendix:verification}

In this section we present the proof of \Cref{thm:verification}. We will start by proving that the BSDE system, which was introduced informally in \Cref{sec:mainresults}, has a close relation to the problem. Let us introduce the following notation:
\[
h_t(y, x, z, a) = f(t, y,x, a) + z \cdot b(t, x,a).
\]
 
  We also denote by $\mathcal{L}_{t,(y)}^{\alpha}$ the generator associated with the control $\alpha$ but acting on the variable $y$. 
    For a function $\psi(y)$, we define:
    \[
       \mathcal{L}_{t,(y)}^{\alpha_\smalltext{t}} \psi(y) \coloneqq b(t, X_t, \alpha_t) \cdot \sigma(t, X_t)^\top \nabla_y \psi(y) + \frac{1}{2}\Tr\big[\sigma(t, X_t)\sigma(t, X_t)^\top \nabla_{yy}^2 \psi(y)\big].
    \]

In the spirit of \citeauthor{hernandez2023me} \cite{hernandez2023me}, for a control process $\alpha \in \Ac$ and an initial condition $(t, x)$ for the state process $X$, we define the following auxiliary processes $(\mathcal{Y}^{y, \alpha}, \mathcal{Z}^{y, \alpha})$. For a fixed parameter $y \in \R^n$, they solve the BSDE

\begin{equation}\label{eq:aux_verification}
    \begin{cases}
  \displaystyle      \mathcal{Y}_s^{y, \alpha} = \xi(y, X_T) + \int_s^T h_u\big(y, X_u, \mathcal{Z}_u^{y, \alpha}, \alpha_u\big) \mathrm{d}u -\int_s^T \mathcal{Z}_u^{y, \alpha}  \mathrm{d}W_u, \;
 s \in [t, T],\\[0.8em]
\displaystyle        \partial\mathcal{Y}_s^{y, \alpha} =  \nabla_{y}\xi(y, X_T) + \int_s^T \big( \nabla_{y} f(u, y, X_u, \alpha_u) + \partial\mathcal{Z}_u^{y, \alpha} b(u, X_u, \alpha_u) \big) \mathrm{d}u -\int_s^T  \partial \mathcal{Z}_u^{y, \alpha}  \mathrm{d}W_u,\;
 s\in[t,T],\\[0.8em]
    \displaystyle     \partial \partial\mathcal{Y}_s^{y, \alpha} =  \nabla_{yy}^2 \xi(y, X_T) 
+ \int_s^T \big( \nabla_{yy}^2 f(u, y, X_u, \alpha_u) + \partial\partial\mathcal{Z}_u^{y, \alpha} b(u, X_u, \alpha_u) \big) \mathrm{d}u -\int_s^T  \partial \partial \mathcal{Z}_u^{y, \alpha}  \mathrm{d}W_u,\;
 s\in[t,T].
\end{cases}
\end{equation}

We see that the structure of the system is the same as the one of \eqref{BSDE Weak Formulation}, and we will impose the same concept of solution. We start the analysis with the following lemma.
\begin{lemma} \label{lemma:auxiliaryforverification}
    We have that $\mathcal{Y}_t^{y, \alpha} = J(t, x, y, \alpha)$.

\end{lemma}

 \begin{proof}
    We work under the probability measure $\P^{t, x, \alpha}$, under which $X_t = x$ and the dynamics on $[t, T]$ are controlled by $\alpha$.
    Recall that under this measure, the Brownian motion is $W^\alpha$.
    Substituting the dynamics of $X$ into the first equation of \eqref{eq:aux_verification}, we have
    \[
        \d \mathcal{Y}_u^{y, \alpha} = -\big( h_u(y, X_u, \mathcal{Z}_u^{y, \alpha}, \alpha_u) - \mathcal{Z}_u^{y, \alpha} \cdot b(u, X_u, \alpha_u) \big) \d u + \mathcal{Z}_u^{y, \alpha} \cdot \d W^\alpha_u, \; u \in [t, T].
    \]
    The drift term simplifies to $-f(u, y, X_u, \alpha_u)$. Integrating from $t$ to $T$
    \[
        \mathcal{Y}_t^{y, \alpha} = \xi(y, X_T) + \int_t^T f(u, y, X_u, \alpha_u) \diff u - \int_t^T \mathcal{Z}_u^{y, \alpha} \cdot \diff W^\alpha_u.
    \]
    Taking expectations under $\P^{t, x, \alpha}$ eliminates the stochastic integral:
    \[
        \mathcal{Y}_t^{y, \alpha} = \E^{\P^{\smalltext{t}, \smalltext{x}, \smalltext{\alpha}}} \bigg[ \int_t^T f(u, y, X_u, \alpha_u) \diff u + \xi(y, X_T) \bigg].
    \]
    By definition, the right-hand side is exactly the cost functional $J(t, x, y, \alpha)$.
\end{proof}

In other words, the process $\mathcal{Y}_t^{y, \alpha}$ captures the dynamics of the reward functional if we fix the value $y$. Note that this could have been deduced from the PDE of $\mathcal{J}$, as it is easy to show that $\mathcal{Y}^{y, \alpha^\star}_t = \mathcal{J}(t, X_t, y)$. The idea now is to fix an equilibrium control $\alpha^\star$ and to understand the corresponding process $\mathcal{Y}^{\smallertext{X}_\smalltext{t}, \alpha^\star}_t$. One key observation is that it can be understood from two perspectives:
\medskip

\begin{enumerate}
    \item[$(i)$] from that of \eqref{eq:aux_verification}, fixing the value of $y= X_t$ and considering the resulting dynamics. This shows that $\mathcal{Y}_t^{\smallertext{X}_\smalltext{t}, \alpha^\star} = J(t, x, x, \alpha^\star) = V(t, x) = V(t, X_t);$
    \item [$(ii)$] or seen as an Itô process: we have defined a uni-parametric family of processes, and consider $\mathcal{Y}_t^{\smallertext{X}_\smalltext{t}, \alpha^\star}$ as a composition of the family with a process. In other words, we let the superscript parameter change as time advances.
\end{enumerate}

In the informal derivation of the BSDE system, we wrote $Y_t = V(t, X_t)$. The first goal of the section is that, starting from the BSDE system \eqref{BSDE Weak Formulation}, we can recover this rigorously. As we already have that $\mathcal{Y}_t^{\smallertext{X}_\smalltext{t}, \alpha^\star} = V(t, X_t)$, we must show now that $\mathcal{Y}_t^{\smallertext{X}_\smalltext{t}, \alpha^\star} = Y_t$ under suitable assumptions, which happen to be the ones introduced in \Cref{sec:mainresults}.

\medskip
\begin{proposition}\label{prop:equalityYprocess}
    Let {\rm\Cref{Assumptions}} hold. Let $(Y, Z, \partial Y, \partial Z, \partial\partial Y, \partial\partial Z)$ be a solution to \eqref{BSDE Weak Formulation} in the sense of  {\rm\Cref{def:bsde_solution}} with $\alpha^\star_t = \mathcal{V}^\star(t, X_t, Z_t, \partial Y_t^{\smallertext{X}_\smalltext{t}}).$ Then, we have that under $\mathbb{P}^{\alpha^\smalltext{\star}}$
    \[
    Y_t = \mathcal{Y}_t^{\smallertext{X}_\smalltext{t}, \alpha^\smalltext{\star}},\; t\in[0,T].
    \]
\end{proposition}

\begin{proof}
    As the equilibrium control $\alpha^\star$ maximises the Hamiltonian $H$, we substitute the optimal drift into the first equation of \eqref{BSDE Weak Formulation}. Recall that the Hamiltonian is given by:
    \[
        H(t, x, z, \gamma, \eta, \rho) = f(t, x, x, \alpha^\star_t) + b(t, x, \alpha^\star_t) \cdot (z - \sigma(t, x)^\top\gamma) - \frac{1}{2}\Tr\big[\sigma(t, x)\sigma(t, x)^\top \eta\big] - \Tr\big[\sigma(t, x) \rho^\top\big].
    \]
    Thus, the dynamics of $Y$ under the reference measure $\P$ is
    \begin{align*}
        \d Y_t &= -\bigg( f(t, X_t, X_t, \alpha^\star_t) + b(t, X_t, \alpha^\star_t) \cdot \big(Z_t - \sigma(t, X_t)^\top \partial Y_t^{\smallertext{X}_\smalltext{t}}\big) - \frac{1}{2}\Tr\big[\sigma(t, X_t) \sigma(t, X_t)^\top \partial\partial Y_t^{\smallertext{X}_\smalltext{t}}\big] - \Tr\big[\sigma(t, X_t) (\partial Z_t^{\smallertext{X}_\smalltext{t}})^\top\big] \bigg) \d t \\
        &\quad+ Z_t \d W_t.
    \end{align*}
    We change the measure to $\mathbb{P}^{\alpha^\smalltext{\star}}$ using the transformation $\d W_t = \d W^{\alpha^\smalltext{\star}}_t + b(t, X_t, \alpha^\star_t) \d t$. The term $Z_t \cdot b(t, X_t,\alpha^\star_t)$ arising from the Girsanov transformation cancels with the term $-b(t, X_t,\alpha^\star_t) \cdot Z_t$ inside the Hamiltonian driver. This yields the following dynamics for $Y$ under $\mathbb{P}^{\alpha^\smalltext{\star}}$
    \begin{align}\label{eq:bsde_Y_under_P_alpha}
       Y_t &= \xi(X_T, X_T) + \int_t^T \bigg( f(u, X_u, X_u, \alpha^\star_u) + b(u, X_u, \alpha^\star_u) \cdot \sigma(u, X_u)^\top \partial Y^{\smallertext{X}_\smalltext{u}}_u \nonumber \\
       &\quad + \frac{1}{2}\Tr\big[\sigma(u, X_u) \sigma(u, X_u)^\top \partial\partial Y^{\smallertext{X}_\smalltext{u}}_u\big] + \Tr\big[\sigma(u, X_u) \partial (Z^{\smallertext{X}_\smalltext{u}}_u)^\top\big] \bigg) \mathrm{d}u - \int_t^T  Z_u \cdot \mathrm{d}W^{\alpha^\smalltext{\star}}_u.
    \end{align}

    Now we apply the Itô--Kunita--Wentzell formula to the composed process $\mathcal{Y}_t^{\smallertext{X}_\smalltext{t}, \alpha^\smalltext{\star}}$.
    From the auxiliary system \eqref{eq:aux_verification}, for a fixed $y$, the process $\mathcal{Y}^{y, \alpha^\smalltext{\star}}$ satisfies the dynamics under $\mathbb{P}^{\alpha^\smalltext{\star}}$
    \[
        \d \mathcal{Y}^{y, \alpha^\smalltext{\star}}_u = -f(u, y, X_u, \alpha^\star_u) \d u + \mathcal{Z}^{y, \alpha^\smalltext{\star}}_u \cdot \d W^{\alpha^\smalltext{\star}}_u.
    \]
    The dynamics of the composition $\mathcal{Y}_t^{\smallertext{X}_\smalltext{t}, \alpha^\smalltext{\star}}$ is given by
    \begin{align*}
        \d \mathcal{Y}_t^{\smallertext{X}_\smalltext{t}, \alpha^\smalltext{\star}} &= \d \mathcal{Y}^{y, \alpha^\smalltext{\star}}_t \big|_{y=\smallertext{X}_\smalltext{t}} + \partial_y \mathcal{Y}^{\smallertext{X}_\smalltext{t}, \alpha^\smalltext{\star}}_t \cdot \d X_t + \frac{1}{2}\Tr\big[ \sigma(t, X_t)\sigma(t, X_t)^\top \partial_{yy}^2 \mathcal{Y}^{\smallertext{X}_t, \alpha^\star}_t \big] \d t  + \Tr\big[ \sigma(t, X_t) (\partial \mathcal{Z}^{\smallertext{X}_\smalltext{t}, \alpha^\smalltext{\star}}_t)^\top \big] \d t,
    \end{align*}
    where $\partial \mathcal{Z}^{\smallertext{X}_t, \alpha^\smalltext{\star}}$ denotes the gradient of the field $\nabla_y \mathcal{Z}^{y, \alpha^\smalltext{\star}}|_{y=\smallertext{X}_\smalltext{t}}$.
    Substituting $\d X_t = \sigma(t, X_t) b(t, X_t, \alpha^\star_t) \d t + \sigma(t, X_t) \d W^{\alpha^\smalltext{\star}}_t$
    \begin{align*}
        \d \mathcal{Y}_t^{\smallertext{X}_\smalltext{t}, \alpha^\smalltext{\star}} &= \bigg( -f(t, X_t, X_t, \alpha^\star_t) + \partial_y \mathcal{Y}^{\smallertext{X}_\smalltext{t}, \alpha^\smalltext{\star}}_t \cdot \sigma(t, X_t) b(t, X_t, \alpha^\star_t) + \frac{1}{2}\Tr\big[ \sigma(t, X_t) \sigma(t, X_t)^\top \partial_{yy}^2 \mathcal{Y}^{\smallertext{X}_\smalltext{t}, \alpha^\smalltext{\star}}_t \big] \\
        &\quad + \Tr\big[ \sigma(t, X_t) (\partial \mathcal{Z}^{\smallertext{X}_\smalltext{t}, \alpha^\smalltext{\star}}_t)^\top \big] \bigg) \d t + \big( \mathcal{Z}^{\smallertext{X}_\smalltext{t}, \alpha^\smalltext{\star}}_t + \partial_y \mathcal{Y}^{\smallertext{X}_\smalltext{t}, \alpha^\smalltext{\star}}_t \cdot \sigma(t, X_t) \big) \cdot \d W^{\alpha^\smalltext{\star}}_t.
    \end{align*}
    Rearranging the drift term $\partial_y \mathcal{Y} \cdot \sigma b = b \cdot \sigma^\top \partial_y \mathcal{Y}$, and identifying the cross-variation trace term $\Tr[\sigma (\partial \mathcal{Z})^\top]$ with the Hamiltonian term $\Tr[\sigma \partial Z]$, we observe that $\mathcal{Y}^{\smallertext{X}_\smalltext{t}, \alpha^\smalltext{\star}}$ satisfies the exact same linear BSDE as $Y$ derived in \eqref{eq:bsde_Y_under_P_alpha}.

    \medskip
    Specifically, we identify the variable $Z_t$ with the diffusion term $\mathcal{Z}^{\smallertext{X}_\smalltext{t}, \alpha^\smalltext{\star}}_t + \sigma_t^\top(X_t) \partial Y^{\smallertext{X}_\smalltext{t}}_t$, and we identify the auxiliary field derivatives $\partial_y \mathcal{Y}$ and $\partial_{yy}^2 \mathcal{Y}$ with the solution processes $\partial Y$ and $\partial\partial Y$ (which satisfy the same equations by uniqueness).
    Thus, by the uniqueness of solutions to BSDEs, we conclude $Y_t = \mathcal{Y}_t^{\smallertext{X}_\smalltext{t}, \alpha^\smalltext{\star}}$.
\end{proof}

We remark that we have managed to arrive at the BSDE system that we had deduced from the PDE system appearing in \cite{bjork2021time} from purely probabilistic arguments, namely the Itô--Kunita--Wentzell formula. With the central \Cref{prop:equalityYprocess} proven, we move on with the proof of \Cref{thm:verification}.

\begin{proof}[Proof of Theorem \ref{thm:verification}]
    Let $(t, x)$ be a fixed pair in $[0, T] \times \R^n$ and let $\alpha$ be an arbitrary admissible control in $\Ac$. We aim to verify the equilibrium condition given in \Cref{Def:equilibrium}. For a strictly positive time step $\ell > 0$, we consider the concatenated control strategy $\hat{\alpha}$ defined by $\hat{\alpha} \coloneqq  \alpha \otimes_\ell \alpha^\star$. We analyse the difference between the cost of this perturbed strategy and the cost of the equilibrium strategy, $J(t, x, \hat{\alpha}) - J(t, x, \alpha^\star)$.
   
    \medskip
     Recall that the value function is defined as $v(t, x) = J(t, x, \alpha^\star)$. We expand the cost of the perturbed strategy using the definition of the cost functional
    \begin{align*}
        J(t, x, \hat{\alpha}) &= \E^{\P^{\smalltext{t}\smalltext{,} \smalltext{x}\smalltext{,} \smalltext{\alpha}}}\bigg[ \int_t^{t+\ell} f(r,  x, X_r, \alpha_r) \diff r + \int_{t+\ell}^T f(r,  x, X_r, \alpha^\star_r) \diff r + \xi(x, X_T) \bigg] \\
        &= \E^{\P^{\smalltext{t}\smalltext{,} \smalltext{x}\smalltext{,} \smalltext{\alpha}}}\bigg[ \int_t^{t+\ell} f(r,  x, X_r, \alpha_r) \diff r + \E^{\P^{\smalltext{t}\smalltext{,} \smalltext{x}\smalltext{,}\smalltext{\alpha}}}\bigg[ \int_{t+\ell}^T f(r,  x, X_r, \alpha^\star_r) \diff r + \xi(x, X_T) \bigg| \Fc_{t+\ell} \bigg] \bigg].
    \end{align*}

    Using the concatenated measure property, we identify the conditional expectation as the auxiliary value function $\mathcal{Y}$, evaluated with the fixed preference parameter $x$ under the equilibrium control $\alpha^\star$:
    \[
        J(t, x, \hat{\alpha}) = \E^{\P^{\smalltext{t}\smalltext{,}\smalltext{ x}\smalltext{,} \smalltext{\alpha}}}\bigg[ \int_t^{t+\ell} f(r,  x, X_r, \alpha_r) \diff r + \mathcal{Y}_{t+\ell}^{x, \alpha^\smalltext{\star}} \bigg].
    \]
    We add and subtract the equilibrium value function at time $t+\ell$, which satisfies the relation $v(t+\ell, X_{t+\ell}) = \mathcal{Y}_{t+\ell}^{\smallertext{X}_{\smalltext{t}\smalltext{+}\smalltext{\ell}}, \alpha^\smalltext{\star}}$. This yields
    \[
        J(t, x, \hat{\alpha}) = \E^{\P^{\smalltext{t}\smalltext{,} \smalltext{x}\smalltext{,} \smalltext{\alpha}}}\bigg[ \int_t^{t+\ell} f(r,  x, X_r, \alpha_r) \diff r + v(t+\ell, X_{t+\ell}) \bigg] + I,
    \]
    where the term $I$ captures the cost of inconsistency due to the changing preference parameter
    \[
        I \coloneqq  \E^{\P^{\smalltext{t}\smalltext{,} \smalltext{x}\smalltext{,} \smalltext{\alpha}}}\big[ \mathcal{Y}_{t+\ell}^{x, \alpha^\smalltext{\star}} - \mathcal{Y}_{t+\ell}^{\smallertext{X}_{\smalltext{t}\smalltext{+}\smalltext{\ell}}, \alpha^\smalltext{\star}} \big].
    \]

    Since we assumed that $v(t,x)$ is in $C^{1,2}([0,T) \times \R^n)$, let us apply It\^o's formula to the process $v(s, X_s)$ on the interval $[t, t+\ell]$ under the measure $\P^{t, x, \alpha}$. Note that $\d X_r = b(r, X_r, \alpha_r)\d r + \sigma(r, X_r)\d W^\alpha_r$.
    \begin{align*}
        v(t+\ell, X_{t+\ell}) &= v(t, x) + \int_t^{t+\ell} \bigg( \partial_t v(r, X_r) + b(r, X_r, \alpha_r) \cdot \partial_x v(r, X_r) + \frac{1}{2}\Tr\big[\sigma(r,  X_r)\sigma(r, X_r)^\top \partial_{xx}^2 v(r, X_r) \big]\bigg) \diff r \\
        &\quad + \int_t^{t+\ell} \partial_x v(r, X_r) \cdot \sigma(r, X_r) \diff W^\alpha_r.
    \end{align*}

    Taking expectations under $\P^{t, x, \alpha}$ eliminates the stochastic integral. Indeed, we identify the integrand $\partial_x v(r, X_r) \cdot \sigma(r, X_r)$ as the process $Z_r$ from the BSDE governing $v$. 
    By \Cref{def:bsde_solution}, we have $Z \in \mathbb{H}^2(\R^d,\F,\P^{t, x, \alpha})$. Consequently, the stochastic integral is a true martingale with zero expectation. Substituting this into the expression for $J(t, x, \hat{\alpha})$, we obtain
    \begin{align*}
        J(t, x, \hat{\alpha}) - v(t, x) &= \E^{\P^{\smalltext{t}\smalltext{,} \smalltext{x}\smalltext{,} \smalltext{\alpha}}}\bigg[ \int_t^{t+\ell} \bigg( f(r,  x, X_r, \alpha_r) + \partial_t v(r, X_r) + b(r, X_r, \alpha_r) \cdot \partial_x v(r, X_r)\\
        &\quad  + \frac{1}{2}\Tr\big[\sigma(r, X_r)\sigma(r, X_r)^\top \partial_{xx}^2 v(r, X_r)\big] \bigg) \diff r \bigg] + I.
    \end{align*}

    To analyse $I$, we apply \Cref{Ito Wentzell Formula} to the map $y \longmapsto \mathcal{Y}_{t+\ell}^{y, \alpha^\star}$ along the process $X_r$ for $r \in [t, t+\ell]$. This yields

\begin{equation*}
        \mathcal{Y}_{t+\ell}^{\smallertext{X}_{\smalltext{t}\smalltext{+}\smalltext{\ell}}, \alpha^\smalltext{\star}} - \mathcal{Y}_{t+\ell}^{x, \alpha^\smalltext{\star}} 
        = \int_t^{t+\ell} \partial \mathcal{Y}_{t+\ell}^{\smallertext{X}_\smalltext{r}, \alpha^\smalltext{\star}} \cdot \diff X_r + \frac{1}{2} \int_t^{t+\ell} \Tr\big[ \sigma(r, X_r)\sigma(r, X_r)^\top \partial \partial \mathcal{Y}_{t+\ell}^{\smallertext{X}_\smalltext{r}, \alpha^\smalltext{\star}} \big] \diff r.
    \end{equation*}
    Substituting the dynamics $\diff X_r = b(r, X_r, \alpha_r) \diff r + \sigma(r, X_r) \diff W^\alpha_r$, we isolate the stochastic integral term:
    \[
        \int_t^{t+\ell} \partial \mathcal{Y}_{t+\ell}^{\smallertext{X}_\smalltext{r}, \alpha^\smalltext{\star}} \cdot \sigma(r, X_r) \diff W^\alpha_r.
    \]
    Taking expectations under $\P^{t, x, \alpha}$, this term vanishes.
    Indeed, for any fixed parameter $y$, the process $\partial \mathcal{Y}^y$ solves a linear BSDE whose driver $\nabla_y f$ and terminal condition $\nabla_y \xi$ have polynomial growth in $y$ and $x$ (\Cref{Assumptions}-$(iii)$).
    Standard BSDE estimates (\emph{e.g.}, \cite[Proposition 2.1]{el1997backward}) imply that the solution $\partial \mathcal{Y}^y$ inherits this polynomial growth.
    Consequently, when evaluated at $y=X_r$, the integrand $\partial \mathcal{Y} \cdot \sigma$ has polynomial growth in $X_r$.
    Given the finite moments of $X$ (\Cref{Assumptions}-$(iv)$), the integrand belongs to $\mathbb{H}^2(\R^d,\F,\P^{t, x, \alpha})$, making the integral a true martingale with zero mean.

    \medskip
    
    We are thus left with the drift terms
   \begin{align*}
I &= - \E^{\P^{\smalltext{t}, \smalltext{x}, \smalltext{\alpha}}} \bigg[ \int_t^{t+\ell} \bigg( b(r, X_r, \alpha_r) \cdot \sigma(r, X_r)^\top \partial_y \mathcal{Y}_{t+\ell}^{\smallertext{X}_\smalltext{r}, \alpha^\smalltext{\star}} + \frac{1}{2}\Tr\big[ \sigma(r, X_r)\sigma(r, X_r)^\top \partial_{yy}^2 \mathcal{Y}_{t+\ell}^{\smallertext{X}_\smalltext{r}, \alpha^\smalltext{\star}} \big] \bigg) \diff r \bigg].
    \end{align*}

    We now combine the results. We add and subtract two specific terms inside the integral
    \begin{enumerate}
        \item the running cost evaluated at the current state preference $f(r,  X_r, X_r, \alpha_r);$
        \item the generator adjustment term evaluated at the current time $\Lc^{\alpha_\smalltext{r}}_{r, (y)} \mathcal{Y}^{\smallertext{X}_\smalltext{r}, \alpha^\smalltext{\star}}_r$.
    \end{enumerate}
    Grouping these terms appropriately, we obtain the following decomposition
    \begin{align*}
        J(t, x, \hat{\alpha}) - v(t, x) &= \E^{\P^{\smalltext{t}\smalltext{,} \smalltext{x}\smalltext{,} \smalltext{\alpha}}} \Bigg[ \int_t^{t+\ell} \underbrace{ \bigg( \partial_t v(r, X_r) + \Lc_r^{\alpha_\smalltext{r}} v(r, X_r) + f(r,  X_r, X_r, \alpha_r) - \Lc^{\alpha_\smalltext{r}}_{r,(y)} \mathcal{Y}^{\smallertext{X}_\smalltext{r}, \alpha^\smalltext{\star}}_r \bigg) }_{\text{Term A: Hamiltonian gap}} \diff r \Bigg] \\
        &\quad + \E^{\P^{\smalltext{t}\smalltext{,} \smalltext{x}\smalltext{,} \smalltext{\alpha}}} \Bigg[ \int_t^{t+\ell} \underbrace{ \bigg( f(r,  x, X_r, \alpha_r) - f(r,  X_r, X_r, \alpha_r) \bigg) }_{\text{Term B: preference approximation}} \diff r \Bigg] \\
        &\quad + \E^{\P^{\smalltext{t}\smalltext{,} \smalltext{x}\smalltext{,} \smalltext{\alpha}}} \Bigg[ \int_t^{t+\ell} \underbrace{ \Big( \Lc^{\alpha_\smalltext{r}}_{r,(y)} \mathcal{Y}^{\smallertext{X}_\smalltext{r}, \alpha^\smalltext{\star}}_r - \Lc^{\alpha_\smalltext{r}}_{r,(y)} \mathcal{Y}_{t+\ell}^{\smallertext{X}_\smalltext{r}, \alpha^\smalltext{\star}} \Big) }_{\text{Term C: continuity error}} \diff r \Bigg].
    \end{align*}

    \textbf{Analysis of term A.} 
    This term measures the local sub-optimality of the control $\alpha$. Let $\mathcal{I}_r$ denote the integrand
    \[
        \mathcal{I}_r \coloneqq \partial_t v(r, X_r) + \mathcal{L}_r^{\alpha_\smalltext{r}} v(r, X_r) + f(r,  X_r, X_r, \alpha_r) - \mathcal{L}_{r,(y)}^{\alpha_\smalltext{r}} \mathcal{Y}^{\smallertext{X}_\smalltext{r}, \alpha^\smalltext{\star}}_r.
    \]

    We identify $\partial_t v$ using the first equation of the BSDE system \eqref{BSDE Weak Formulation}. Under the reference measure $\P$, the drift of the process $Y_r = v(r, X_r)$ is given by the driver $-H$. Comparing this with the drift obtained from It\^o's formula applied to $v(r, X_r)$, we establish the identity
    \begin{equation}\label{eq:partial_t_v_id}
        \partial_t v(r, X_r) = -H\big(r, X_r, Z_r, \partial Y_r^{\smallertext{X}_\smalltext{r}}, \partial\partial Y_r^{\smallertext{X}_\smalltext{r}}, \partial Z_r^{\smallertext{X}_\smalltext{r}}\big) - \frac{1}{2}\Tr\big[\sigma(r, X_r)\sigma(r, X_r)^\top \partial_{xx}^2 v(r, X_r)\big],\; \mathrm{d}r\otimes\P\text{\rm--a.e.}
    \end{equation}
    Substituting this expression into $\mathcal{I}_r$, and expanding the operators $\Lc^{\alpha_\smalltext{r}}$ and $\Lc^{\alpha_\smalltext{r}}_{(y)}$, we observe two key cancellations
    \begin{enumerate}
        \item[$(i)$] the diffusion term $\frac{1}{2}\Tr[\sigma\sigma^\top \partial_{xx}^2 v]$ from the generator $\Lc^{\alpha_\smalltext{r}} v$ cancels with the corresponding term in \eqref{eq:partial_t_v_id}$;$
        \item[$(ii)$] the inconsistency terms involving $\partial_{yy}^2 \mathcal{J}$ and $\partial_{xy}^2 \mathcal{J}$ appearing in $\Lc^{\alpha_\smalltext{r}}_{(y)} \mathcal{Y}$ depend only on the volatility $\sigma$ (which is control-independent) and cancel exactly with the inconsistency adjustment terms included in the definition of the extended Hamiltonian $H$.
    \end{enumerate}
    
    Consequently, the integrand reduces to the difference between the Hamiltonian objective evaluated at the arbitrary control $\alpha_r$ and its maximum value
    \[
        \mathcal{I}_r = \big( f(r,  X_r, X_r, \alpha_r) + b(r, X_r, \alpha_r) \cdot \big( Z_r - \sigma(r, X_r)^\top \partial Y^{\smallertext{X}_\smalltext{r}}_r \big) \big) - \sup_{a \in A} \big\{ f(r,  X_r, X_r, a) + b(r, X_r, a) \cdot \big( Z_r - \sigma(r, X_r)^\top \partial Y^{\smallertext{X}_\smalltext{r}}_r \big) \big\}.
    \]
    Thus, $\mathcal{I}_r \le 0$ almost surely, and we readily obtain
    \[
        \int_t^{t+\ell} \mathcal{I}_r \diff r \le 0,\; \P^{t, x, \alpha}\text{\rm--a.s.}
    \]

    \textbf{Analysis of term B.} We use the Lipschitz-continuity of $f$ with respect to its first parameter (\Cref{Assumptions}). Let $L$ be the Lipschitz-continuity constant. Then
    \[
        \|\text{Term B}\| = \big\|f(r,  x, X_r, \alpha_r) - f(r,  X_r, X_r, \alpha_r)\big\| \le L \|x - X_r\|.
    \]
    Taking the expectation under $\P^{t, x, \alpha}$
    \begin{align*}
        \bigg\| \E^{\P^{\smalltext{t}, \smalltext{x}, \smalltext{\alpha}}} \bigg[\int_t^{t+\ell} \text{Term B} \diff r \bigg]\bigg\| &\le L \int_t^{t+\ell} \E^{\P^{\smalltext{t}\smalltext{,} \smalltext{x}\smalltext{,} \smalltext{\alpha}}}[\|X_r - x\|] \diff r.
    \end{align*}
    Using standard moment estimates for SDEs with linear growth coefficients (see \cite[Corollary 2.5.12]{karatzas1991brownian}), we have $\E^{\P^{\smalltext{t}\smalltext{,} \smalltext{x}\smalltext{,} \smalltext{\alpha}}}[\|X_r - x\|] \le C(1+\|x\|) \sqrt{r-t}$. Thus
    \[
        \int_t^{t+\ell} \sqrt{r-t} \diff r = \bigg[ \frac{2}{3}(r-t)^{3/2} \bigg]_t^{t+\ell} = \frac{2}{3} \ell^{3/2} = o(\ell).
    \]

\textbf{Analysis of term C.} This term arises from the time-continuity of the inconsistency adjustment. We analyse the integral of the difference
    \[
       \Delta_r \coloneqq \mathcal{L}_{r,(y)}^{\alpha_r} \mathcal{Y}^{\smallertext{X}_\smalltext{r}, \alpha^\smalltext{\star}}_r - \mathcal{L}_{r,(y)}^{\alpha_r} \mathcal{Y}_{t+\ell}^{\smallertext{X}_\smalltext{r}, \alpha^\smalltext{\star}}.
    \]
    Recall that the operator $\Lc^{\alpha}_{(y)}$ is linear in the derivatives $\partial_y \mathcal{Y}^{\cdot, \alpha^\smalltext{\star}}$ and $\partial_{yy}^2 \mathcal{Y}^{\cdot, \alpha^\smalltext{\star}}$, with coefficients $b$ and $\sigma$ that satisfy linear growth conditions.
    \medskip
    
    Since the solution to the auxiliary BSDE system belongs to the space $\mathbb{S}^2(\R^d,\F,\P)$, the mappings $r \longmapsto \partial_y \mathcal{Y}^{\cdot, \alpha^\star}_r$ and $r \longmapsto \partial_{yy}^2 \mathcal{Y}^{\cdot, \alpha^\star}_r$ are continuous with respect to time in the norm of $\L^2(\P^{t,x,\alpha})$. 
    Furthermore, the state process $X$ defines a mapping $r \longmapsto X_r$ which is continuous with respect to time in $\L^p(\P^{t,x,\alpha})$ for any $p \ge 1$.
    By H\"older's inequality, the composition appearing in $\Delta_r$ is continuous with respect to time in $\L^1(\P^{t, x, \alpha})$. Therefore, we readily obtain:
    \[
       \E^{\P^{\smalltext{t}\smalltext{,} \smalltext{x}\smalltext{,} \smalltext{\alpha}}} \bigg[ \int_t^{t+\ell} \Delta_r \diff r \bigg] = \int_t^{t+\ell} o(1) \diff r = o(\ell).
    \]

Combining the non-positivity of Term A with the $o(\ell)$ estimates for Terms B and C, we obtain
    \[
        J(t, x, \hat{\alpha}) - v(t, x) \le 0 + o(\ell) + o(\ell).
    \]
    This confirms that the equilibrium strategy $\alpha^\star$ provides a higher payoff than the perturbed strategy $\hat{\alpha}$ up to first order, thereby satisfying the definition of an equilibrium control.
\end{proof}

\begin{remark}
    This result motivates {\rm\Cref{Def:equilibrium}} in the following sense: one could argue that it would make sense to allow for improvements of order $o(\ell^k)$, since the use of $k=1$ in our definition could seem arbitrary at first. However, we see here that $k=1$ is exactly the power that we need to guarantee the result.
\end{remark}

\section{Well-posedness of the BSDE system}\label{sec:wellposedness_proof}

In this section, we provide the rigorous proof for the existence and uniqueness of the solution to the system \eqref{BSDE Weak Formulation}. We adopt a fixed-point approach on the full system of three equations. To handle the linear growth of the value function derivatives (typical in linear--quadratic problems), we work in weighted spaces that allow for polynomial growth in the parameter $y$. We also show that our work implies the existence of solutions in the sense of \Cref{def:bsde_solution}.

\medskip
To ease the notation, we will denote by $C$ an arbitrary constant that may change line by line. We first introduce and prove the following standard a priori estimate (similar to \citeauthor*{el1997backward} \cite[Proposition 2.1]{el1997backward}).

\begin{lemma}[A Priori Estimates and Contraction]\label{lemma:BSDE_estimate}
Let $(\delta Y, \delta Z)$ be the solution to the linearized BSDE with driver difference $\delta f$:
\begin{equation}
    -\mathrm{d}(\delta Y_t) = \delta f_t \mathrm{d}t - \delta Z_t \mathrm{d}W_t, \quad \delta Y_T = 0.
\end{equation}
For $\beta$ sufficiently large, the following estimate holds:
\begin{equation}
    \|\delta Y\|^2_{\mathbb{S}^\smalltext{2}_\smalltext{\beta}(\R,\F,\P)} + \|\delta Z\|^2_{\mathbb{H}^2_\beta(\R^d,\F,\P)} \le \frac{C}{\beta} \|\delta f\|^2_{\mathbb{H}^2_\beta(\R,\F,\P)}.
\end{equation}
\end{lemma}

\begin{proof}

We start by applying It\^o's formula to the process $e^{\beta t}|\delta Y_t|^2$:
\begin{align*}
    \mathrm{d}(e^{\beta t}|\delta Y_t|^2) &= \beta e^{\beta t}|\delta Y_t|^2 \mathrm{d}t + e^{\beta t} \left( 2 \delta Y_t \cdot \mathrm{d}(\delta Y_t) + |\delta Z_t|^2 \mathrm{d}t \right) \\
    &= e^{\beta t} \left( \beta |\delta Y_t|^2 + |\delta Z_t|^2 - 2 \delta Y_t \cdot \delta f_t \right) \mathrm{d}t + 2e^{\beta t} \delta Y_t \cdot \delta Z_t \mathrm{d}W_t.
\end{align*}
Integrating from $0$ to $T$, taking expectations, and using the fact that $\delta Y_T = 0$ and the stochastic integral is a martingale, we get:
\begin{equation*}
    \E \bigg[ \int_0^T e^{\beta s} \left( \beta |\delta Y_s|^2 + |\delta Z_s|^2 \right) \mathrm{d}s \bigg] \leq 2 \E \bigg[ \int_0^T e^{\beta s} \delta Y_s \cdot \delta f_s \mathrm{d}s \bigg].
\end{equation*}
We now use Young's inequality, $2ab \le \frac{\beta}{2}a^2 + \frac{2}{\beta}b^2$ on the right-hand side:
\begin{equation*}
    2 \delta Y_s \cdot \delta f_s \le \frac{\beta}{2} |\delta Y_s|^2 + \frac{2}{\beta} |\delta f_s|^2.
\end{equation*}
Substituting this back into the integral equality:
\begin{equation*}
    \E \bigg[ \int_0^T e^{\beta s} \left( \beta |\delta Y_s|^2 + |\delta Z_s|^2 \right) \mathrm{d}s \bigg] \le \E \bigg[ \int_0^T e^{\beta s} \left( \frac{\beta}{2} |\delta Y_s|^2 + \frac{2}{\beta} |\delta f_s|^2 \right) \mathrm{d}s \bigg].
\end{equation*}
Subtracting the term $\frac{\beta}{2} \|\delta Y\|^2_{\mathbb{H}^2_\beta}$ from both sides yields:
\begin{equation*} \label{eq:integral_bound}
    \frac{\beta}{2} \|\delta Y\|^2_{\mathbb{H}^2_\beta} + \|\delta Z\|^2_{\mathbb{H}^2_\beta} \le \frac{2}{\beta} \|\delta f\|^2_{\mathbb{H}^2_\beta}.
\end{equation*}
This inequality immediately gives two bounds:
\begin{enumerate}
    \item $\|\delta Y\|^2_{\mathbb{H}^2_\beta} \le \frac{4}{\beta^2} \|\delta f\|^2_{\mathbb{H}^2_\beta}$.
    \item $\|\delta Z\|^2_{\mathbb{H}^2_\beta} \le \frac{2}{\beta} \|\delta f\|^2_{\mathbb{H}^2_\beta}$.
\end{enumerate}

Note that the bound for $\|\delta Z\|^2_{\mathbb{H}^2_\beta}$ is the one we need, and that we residually obtained a strong bound for $\|\delta Y\|^2_{\mathbb{H}^2_\beta}$ that will also be useful later. The latter will help us prove our desired bound for $\|\delta Y\|^2_{\mathbb{S}^2_\beta}$. 

To bound the supremum, we return to the integral form of the process $e^{\beta t}|\delta Y_t|^2$. By integrating the It\^o differential from $t$ to $T$ and using the terminal condition $\delta Y_T = 0$, we have:
\begin{equation*}
    e^{\beta t}|\delta Y_t|^2 + \int_t^T e^{\beta s} (\beta |\delta Y_s|^2 + |\delta Z_s|^2) \mathrm{d}s = \int_t^T 2 e^{\beta s} \delta Y_s \cdot \delta f_s \mathrm{d}s - \int_t^T 2 e^{\beta s} \delta Y_s \cdot \delta Z_s \mathrm{d}W_s.
\end{equation*}
The integral term on the left-hand side is non-negative. Thus:
\begin{equation*}
    e^{\beta t}|\delta Y_t|^2 \le \int_t^T 2 e^{\beta s} |\delta Y_s| |\delta f_s| \mathrm{d}s + \bigg| \int_t^T 2 e^{\beta s} \delta Y_s \cdot \delta Z_s \mathrm{d}W_s \bigg|.
\end{equation*}
We now take the supremum over $t \in [0,T]$ on both sides, followed by the expectation. For the first term on the right (the drift), we effectively bound it by the integral over $[0,T]$:
\begin{align*}
    \E \bigg[ \sup_{t \in [0,T]} e^{\beta t}|\delta Y_t|^2 \bigg] \le \underbrace{\E \int_0^T e^{\beta s} |2 \delta Y_s \cdot \delta f_s| \mathrm{d}s}_{\text{Drift part}} + \underbrace{2 \E \bigg[ \sup_{t \in [0,T]} \bigg| \int_0^t e^{\beta s} \delta Y_s \cdot \delta Z_s \mathrm{d}W_s \bigg| \bigg]}_{\text{Martingale } M_t}.
\end{align*}

Using Young's inequality for the drift:

\begin{align*}
    \E \left[ \sup_{t \in [0,T]} e^{\beta t}|\delta Y_t|^2 \right] &\le \E \int_0^T e^{\beta s} \left( \beta |\delta Y_s|^2 + \frac{1}{\beta} |\delta f_s|^2 \right) \mathrm{d}s + 2 \E \left[ \sup_{t \in [0,T]} |M_t| \right] \\
    &\le \frac{C}{\beta} \|\delta f\|^2_{\mathbb{H}^2_\beta} + 2 \E \left[ \sup_{t \in [0,T]} |M_t| \right].
\end{align*}

 Here, we used that the term $\beta \|\delta Y\|^2$ is bounded by $\frac{C}{\beta} \|\delta f\|^2$. To bound the martingale term, we apply the 
Burkholder-Davis-Gundy inequality in its $L^1$ form:

\begin{align*}
    \E \left[ \sup_{t \in [0,T]} \left| \int_0^t e^{\beta s} \delta Y_s \delta Z_s \mathrm{d}W_s \right| \right] &\le 3 \E \left[ \left( \int_0^T e^{2\beta s} |\delta Y_s|^2 |\delta Z_s|^2 \mathrm{d}s \right)^{1/2} \right] \\
    &\le 3 \E \left[ \left(\sup_{r \in [0,T]} e^{\beta r/2} |\delta Y_r|\right) \left( \int_0^T e^{\beta s} |\delta Z_s|^2 \mathrm{d}s \right)^{1/2} \right].
\end{align*}
Using again Young's inequality $ab \le \frac{1}{4}a^2 + C b^2$, and putting everything together:
\begin{equation*}
     \E \left[ \sup_{t \in [0,T]} e^{\beta t}|\delta Y_t|^2 \right] \le \frac{C}{\beta} \|\delta f\|^2_{\mathbb{H}^2_\beta} +  \frac{1}{2} \E \left[ \sup_{r \in [0,T]} e^{\beta r} |\delta Y_r|^2 \right] + C \E \left[ \int_0^T e^{\beta s} |\delta Z_s|^2 \mathrm{d}s \right].
\end{equation*}
The middle term in the RHS can be absorbed into the left-hand side of our supremum estimate. The last term is proportional to $\|\delta Z\|^2_{\mathbb{H}^2_\beta}$, which we already bounded by $\frac{2}{\beta} \|\delta f\|^2$.

\medskip

Concluding, we get:
\begin{equation*}
    \|\delta Y\|^2_{\mathbb{S}^2_\beta} + \|\delta Z\|^2_{\mathbb{H}^2_\beta} \le \frac{C}{\beta} \|\delta f\|^2_{\mathbb{H}^2_\beta},
\end{equation*}

which is what we wanted to prove.

\end{proof}

We also present this immediate corollary, which will prove useful when proving that the central map in the proof of \Cref{thm:wellposedness} is a contraction.

\medskip

\begin{corollary}\label{lemma:elkaroui_est}
    Consider two real \emph{BSDEs} $-\d Y_t = f_t \d t - Z_t \cdot \d W_t$ and $-\d Y^\prime_t = f^\prime_t \d t - Z^\prime_t \cdot \d W_t$, taking values in $\R$. Assume that $Y_T = Y'_{T}$. Let $\delta Y \coloneqq Y - Y^\prime$, $\delta Z \coloneqq Z - Z^\prime$, and $\delta f_t \coloneqq f_t - f^\prime_t$.
    \medskip
    Suppose that there exists a constant $K > 0$ and a non-negative process $\phi \in \mathbb{H}^2_\beta(\R,\F,\P)$ such that
    \[ 
    |\delta f_t| \le K\big(|\delta Y_t| + \|\delta Z_t\| + \phi_t\big), \; \d t \otimes \d \P\text{\rm--a.e.} 
    \]
    Then, for $\beta$ large enough, there exists a constant $C$ such that
    \[ 
    \|\delta Y\|^2_{\mathbb{S}^\smalltext{2}_\smalltext{\beta}(\R,\F,\P)} + \|\delta Z\|^2_{\mathbb{H}^\smalltext{2}_\smalltext{\beta}(\R^\smalltext{d},\F,\P)} \le \frac{C}{\beta} \|\phi\|^2_{\mathbb{H}^\smalltext{2}_\smalltext{\beta}(\R,\F,\P)}.
    \]

\end{corollary}

\begin{proof}

    From Lemma \ref{lemma:BSDE_estimate} we have that there exists $C$ such that: \begin{align*}
        \|\delta Y\|^2_{\mathbb{S}^\smalltext{2}_\smalltext{\beta}(\R,\F,\P)} + \|\delta Z\|^2_{\mathbb{H}^\smalltext{2}_\smalltext{\beta}(\R^\smalltext{d},\F,\P)} &\le \frac{CK}{\beta} \bigg(T\|\delta Y\|^2_{\mathbb{S}^\smalltext{2}_\smalltext{\beta}(\R,\F,\P)} + \|\delta Z\|^2_{\mathbb{H}^\smalltext{2}_\smalltext{\beta}(\R^\smalltext{d},\F,\P)}+\|\phi\|^2_{\mathbb{H}^\smalltext{2}_\smalltext{\beta}(\R,\F,\P)}\Bigg)\\
        &\le \frac{CK\max(1,T)}{\beta} \bigg(\|\delta Y\|^2_{\mathbb{S}^\smalltext{2}_\smalltext{\beta}(\R,\F,\P)} + \|\delta Z\|^2_{\mathbb{H}^\smalltext{2}_\smalltext{\beta}(\R^\smalltext{d},\F,\P)}+\|\phi\|^2_{\mathbb{H}^\smalltext{2}_\smalltext{\beta}(\R,\F,\P)}\Bigg),
    \end{align*} 
    
    where we implicitly used: 
    \begin{equation*}
    \|\delta Y\|^2_{H^2_\beta} 
    = \mathbb{E}\!\int_0^T e^{\beta t}|\delta Y_t|^2\,dt 
    \leq T\,\|\delta Y\|^2_{S^2_\beta}.
\end{equation*}
Rearranging, we get:
\[ 
    \frac{\beta - CK\max(1,T)}{\beta}\Big(\|\delta Y\|^2_{\mathbb{S}^\smalltext{2}_\smalltext{\beta}(\R,\F,\P)} + \|\delta Z\|^2_{\mathbb{H}^\smalltext{2}_\smalltext{\beta}(\R^\smalltext{d},\F,\P)}\Big) \le \frac{C}{\beta} \|\phi\|^2_{\mathbb{H}^\smalltext{2}_\smalltext{\beta}(\R,\F,\P)}.
    \]

    Assuming, for instance, that $\beta > 2CK\max(1,T)$, we have that:

\[ 
    \Big(\|\delta Y\|^2_{\mathbb{S}^\smalltext{2}_\smalltext{\beta}(\R,\F,\P)} + \|\delta Z\|^2_{\mathbb{H}^\smalltext{2}_\smalltext{\beta}(\R^\smalltext{d},\F,\P)}\Big) \le \frac{2C}{\beta} \|\phi\|^2_{\mathbb{H}^\smalltext{2}_\smalltext{\beta}(\R,\F,\P)}.
    \]

\end{proof}

Now we are ready to prove our existence and uniqueness result for weighted spaces:

\begin{proof}[Proof of Theorem \ref{thm:wellposedness}]
    We proceed by constructing a contraction mapping on the Banach space $\mathcal{K}_\beta$. We define the map $\Phi: \mathcal{K}_\beta \longrightarrow \mathcal{K}_\beta$ as follows. Let $\mathbf{w} = (y, z, u, v, \mathfrak{u}, \mathfrak{v})$ be a fixed input tuple in $\mathcal{K}_\beta$. This input serves as the background processes frozen in the drivers. We define the output $\mathbf{W} = (Y, Z, U, V, \mathcal{U}, \mathcal{V}) = \Phi(\mathbf{w})$ as the unique solution to the following decoupled system of BSDEs
      \begin{align}
        \d Y_t &= -H\big(t, X_t, Z_t, u_t^{\smallertext{X}_\smalltext{t}}, \mathfrak{u}_t^{\smallertext{X}_\smalltext{t}}, v_t^{\smallertext{X}_\smalltext{t}}\big) \d t + Z_t\cdot \d W_t, \label{eq:map_Y_detailed} \\
        \d U^y_t &= -G_1\big(t, X_t, y, \underline{z_t, v_t^y, \mathfrak{u}_t^y}\big) \d t + V_t^y\cdot \d W_t, \label{eq:map_U_detailed} \\
        \d \mathcal{U}^y_t &= -G_2\big(t, X_t, y, \underline{z_t, v_t^y, \mathfrak{v}_t^y}\big) \d t + \mathcal{V}_t^y\cdot \d W_t.
\label{eq:map_calU_detailed}
    \end{align}

    In this system, the underlined terms indicate that the drivers depend on the input $\mathbf{w}$ rather than the solution variables being solved for. Specifically, the first equation for $Y$ depends on the diagonal terms of the input fields ($u, \mathfrak{u}, v$) evaluated at the random state $X_t$. The second and third equations are parameterised by $y \in \R^n$ and depend on the input fields evaluated at that specific parameter $y$. Since the system is decoupled and the drivers satisfy the Lipschitz and growth conditions from \Cref{ass:poly_growth}, standard BSDE theory guarantees that a unique solution $\mathbf{W}$ exists for any given input $\mathbf{w}$.

\medskip
    \textit{Step 1.} Let us prove that the map $\Phi$ is well-defined, that is, that for every input $\mathbf{w}$ in $\mathcal{K}_{\beta}^{n,d}$ and $\mathbf{W} = \Phi(\mathbf{w})$, we have that $\mathbf{W} \in \mathcal{K}_{\beta}^{n,d}$.
    We must thus verify that each component of the solution vector $\mathbf{W} = (Y, Z, U, V, \mathcal{U}, \mathcal{V})$ has a finite norm in its respective weighted space.

    \begin{enumerate}
    \item[$(i)$] \textit{The value processes $(Y, Z)$.}
The pair $(Y, Z)$ solves the \emph{BSDE}
\begin{equation*}
    Y_t = \xi(X_T, X_T) + \int_t^T H\big(r, X_r, Z_r, u_r^{\smallertext{X}_\smalltext{r}}, \mathfrak{u}_r^{\smallertext{X}_\smalltext{r}}, v_r^{\smallertext{X}_\smalltext{r}}\big) \d r - \int_t^T Z_r \d W_r, \quad t \in [0,T].
\end{equation*}
By the standard \emph{a priori} estimates for BSDEs with Lipschitz continuous drivers (see, \emph{e.g.}, \citeauthor*{el1997backward} \cite[Proposition 2.1]{el1997backward}, where we take $f^2 = 0$ and $\xi^2 = 0$), the squared norm of the solution in $\mathbb{S}^2_\beta(\R,\F,\P) \times \mathbb{H}^2_\beta(\R^d,\F,\P)$ is bounded by the square-integrability of the terminal condition and the driver evaluated at zero volatility. Specifically, there exists a constant $C > 0$ depending on $T$ and the Lipschitz constant of $H$ such that
\begin{align*}
    \|Y\|^2_{\mathbb{S}^\smalltext{2}_\smalltext{\beta}(\R,\F,\P)} + \|Z\|^2_{\mathbb{H}^\smalltext{2}_\smalltext{\beta}(\R^\smalltext{d},\F,\P)} &\le C \E^\P \bigg[ \mathrm{e}^{\beta T} |\xi(X_T, X_T)|^2 + \int_0^T \mathrm{e}^{\beta r} \big| H\big(r, X_r, 0, u_r^{\smallertext{X}_\smalltext{r}}, \mathfrak{u}_r^{\smallertext{X}_\smalltext{r}}, v_r^{\smallertext{X}_\smalltext{r}}\big) \big|^2 \d r \bigg].
\end{align*}
Using the Lipschitz continuity of $H$ with respect to the inputs $\Theta \coloneqq (z, u, \mathfrak{u}, v)$ and the growth assumption on the base term $H(\cdot, 0)$, we have
\begin{align*}
    \big| H\big(r, X_r, 0, u_r^{\smallertext{X}_\smalltext{r}}, \mathfrak{u}_r^{\smallertext{X}_\smalltext{r}}, v_r^{\smallertext{X}_\smalltext{r}}\big) \big|^2 &\le 2 \big| H(r, X_r, 0, 0, 0, 0) \big|^2 + 2K^2 \Big( \|u_r^{\smallertext{X}_\smalltext{r}}\|^2 + \|\mathfrak{u}_r^{\smallertext{X}_\smalltext{r}}\|^2 + \|v_r^{\smallertext{X}_\smalltext{r}}\|^2 \Big).
\end{align*}
The base term $\E^\P[ \int_0^T |H(r, X_r, 0, 0, 0,0 )|^2 \d r ]$ is finite by \Cref{ass:poly_growth}-$(iii)$. To bound the input terms, we rely on the embedding of the weighted spaces. Recall that for any input field, say $u \in \mathbb{S}^{2,2}_{\beta, \rho}$, we have the pointwise bound $\|u_r^y\|^2 \le \rho(y)^{-1} \|u\|_{\mathbb{S}^{\smalltext{2}\smalltext{,}\smalltext{2}}_{\smalltext{\beta}\smalltext{,}\smalltext{ \rho}}(\R,\F,\P)}^2$. Substituting the random parameter $y=X_r$
\begin{align*}
    \E^\P \bigg[ \int_0^T \mathrm{e}^{\beta r} \|u_r^{\smallertext{X}_\smalltext{r}}\|^2 \d r \bigg] &\le \E^\P \bigg[ \sup_{t \in [0, T]} \rho(X_t)^{-1} \bigg] \|u\|^2_{\mathbb{S}^{\smalltext{2}\smalltext{,}\smalltext{2}}_{\smalltext{\beta}\smalltext{,} \smalltext{\rho}}(\R,\F,\P)}.
\end{align*}
Since $\rho(x)^{-1}$ has polynomial growth and $X$ admits finite moments of all orders (\Cref{Assumptions}), the expectation $\E^\P [ \sup_{t \in [0, T]} \rho(X_t)^{-1} ]$ is finite. An identical argument applies to $\mathfrak{u}$ and $v$. Consequently, the right-hand side of the \emph{a priori} estimate is finite, implying $(Y, Z) \in \mathbb{S}^2_\beta(\R,\F,\P) \times \mathbb{H}^2_\beta(\R^d,\F,\P)$.
        
\medskip
\item[$(ii)$] \textit{The gradient processes $(U, V)$.} For any fixed parameter $y \in \R^n$, the pair $(U^y, V^y)$ solves a BSDE driven by $G_1$. Applying the standard \emph{a priori} estimate (see \cite{el1997backward}) yields:
\begin{align*}
    \|U^y\|^2_{\mathbb{S}^\smalltext{2}_\smalltext{\beta}(\R^n,\F,\P)} + \|V^y\|^2_{\mathbb{H}^\smalltext{2}_\smalltext{\beta}(\R^{n\times d},\F,\P)} &\le C \E^\P \bigg[ \mathrm{e}^{\beta T} |\nabla_y \xi(y, X_T)|^2 + \int_0^T \mathrm{e}^{\beta t} \big| G_1\big(t, X_t, y, z_t, v_t^y, \mathfrak{u}_t^y\big) \big|^2 \d t \bigg].
\end{align*}
By \Cref{ass:poly_growth}.$(ii)$, the driver $G_1$ is Lipschitz continuous with respect to the input variables $\Theta \coloneqq (z, v, \mathfrak{u})$. Therefore, we can bound the squared driver by the source term (at zero input) and the norms of the inputs:
\begin{align*}
    \big| G_1\big(t, X_t, y, z_t, v_t^y, \mathfrak{u}_t^y\big) \big|^2 \le C \Big( \big| G_1(t, X_t, y, 0) \big|^2 + \|z_t\|^2 + \|v_t^y\|^2 + \|\mathfrak{u}_t^y\|^2 \Big).
\end{align*}
To verify that these processes belong to $\mathcal{K}_\beta$, we multiply the entire estimate by the weight $\rho(y)$ and take the supremum over $y \in \R^n$. The inequality splits into two parts
\begin{enumerate}
    \item \textit{source terms:} by \Cref{ass:poly_growth}-$(iii)$, the source terms have finite weighted norms. Specifically
    \[
        \sup_{y \in \R^n} \rho(y) \E^\P \bigg[ |\nabla_y \xi(y, X_T)|^2 + \int_0^T \big| G_1(t, X_t, y, 0) \big|^2 \d t \bigg] < \infty;
    \]
    \item \textit{input terms:} the inputs belong to $\mathcal{K}_\beta$, so their weighted norms are finite
    \begin{align*}
        \sup_{y \in \R^n} \rho(y) \E^\P \bigg[ \int_0^T  \mathrm{e}^{\beta t}\Big( \|z_t\|^2 + \|v_t^y\|^2 + \|\mathfrak{u}_t^y\|^2 \Big) \d t \bigg] &\le C \Big( \|z\|^2_{\mathbb{H}^\smalltext{2}_\smalltext{\beta}(\R^{d},\F,\P)} + \|v\|^2_{\mathbb{H}^{\smalltext{2}\smalltext{,}\smalltext{2}}_{\smalltext{\beta}\smalltext{,} \smalltext{\rho}}(\R^{n},\F,\P)} + \|\mathfrak{u}\|^2_{\mathbb{S}^{\smalltext{2}\smalltext{,}\smalltext{2}}_{\smalltext{\beta}\smalltext{,} \smalltext{\rho}}(\R^{n\times n},\F,\P)} \Big) < \infty.
    \end{align*}
\end{enumerate}
Combining these bounds proves that the output pair $(U, V)$ has a finite weighted norm, \emph{i.e.}, $(U, V) \in \mathbb{S}^{2,2}_{\beta, \rho}(\R^n,\F,\P) \times \mathbb{H}^{2,2}_{\beta, \rho}(\R^{n \times d},\F,\P)$.

\medskip

        \item[$(iii)$] \textit{The Hessian processes $(\mathcal{U}, \mathcal{V})$.} The argument is strictly identical to the gradient case, as the driver $G_2$ satisfies the same condition.
    \end{enumerate}
    Thus, $\mathbf{W} \in \mathcal{K}_\beta(\F,\P)$.

    \medskip

\textit{Step 2.} Next, to prove that $\Phi$ is a contraction for a sufficiently large $\beta$, let us consider two arbitrary inputs $\mathbf{w}$ and $\mathbf{w}^{\prime}$ in $\mathcal{K}_\beta(\F,\P)$. Let $\mathbf{W} = \Phi(\mathbf{w})$ and $\mathbf{W}^{\prime} = \Phi(\mathbf{w}^{\prime})$ be their corresponding outputs. We denote the differences by $\delta \mathbf{w} = \mathbf{w} - \mathbf{w}^{\prime}$ and $\delta \mathbf{W} = \mathbf{W} - \mathbf{W}^{\prime}$. Our goal is to derive an estimate for $\|\delta \mathbf{W}\|_{\mathcal{K}_\smalltext{\beta}(\F,\P)}$ in terms of $\|\delta \mathbf{w}\|_{\mathcal{K}_\smalltext{\beta}(\F,\P)}$.
\medskip

\begin{enumerate}

\item[$(i)$]\textit{Estimation of the value process $(Y, Z)$.}
Consider the first equation for the value process $Y$, which is scalar-valued. The difference in the drivers, denoted by $\delta H_t$, is bounded pointwise by the differences in the solution components and the inputs. Let us define the scalar aggregate error process $\phi_t$ for the inputs as:
\[
    \phi_t \coloneqq \|\delta z_t\|+\|\delta u_t^{\smallertext{X}_\smalltext{t}}\| + \|\delta \mathfrak{u}_t^{\smallertext{X}_\smalltext{t}}\| + \|\delta v_t^{\smallertext{X}_\smalltext{t}}\|.
\]
By the Lipschitz continuity of the Hamiltonian $H$ (Assumption \ref{ass:poly_growth}), we have the pointwise bound $|\delta H_t| \le C(|\delta Y_t| + \|\delta Z_t\| + \phi_t)$. Applying \Cref{lemma:elkaroui_est} to the scalar BSDE for $\delta Y$, we obtain the following bound in the standard weighted spaces:
\[
    \|\delta Y\|^2_{\mathbb{S}^\smalltext{2}_\smalltext{\beta}(\R,\F,\P)} + \|\delta Z\|^2_{\mathbb{H}^\smalltext{2}_\smalltext{\beta}(\R^\smalltext{d},\F,\P)} \le \frac{C}{\beta} \|\phi\|^2_{\mathbb{H}^\smalltext{2}_\smalltext{\beta}(\R,\F,\P)} = \frac{C}{\beta} \E^\P \bigg[ \int_0^T \mathrm{e}^{\beta t} \phi_t^2 \d t \bigg].
\]
To relate this integral to the norms of the random fields in $\mathcal{K}_\beta$, we use the inequality $(a+b+c+d)^2 \le 4(a^2+b^2+c^2+d^2)$ to separate the components of $\phi_t$. We then bound the integral of each term using the moment constant $M_\smallertext{X} \coloneqq \E^\P[\sup_{t \in [0, T]} (1+\|X_t\|^2)^k]$. For instance, for the gradient term $\delta u$, we have:
\begin{align*}
    \E^\P \bigg[ \int_0^T \mathrm{e}^{\beta t} \|\delta u_t^{\smallertext{X}_\smalltext{t}}\|^2 \d t \bigg] &= \E^\P \bigg[ \frac{\rho(X_t)}{\rho(X_t)}\int_0^T \mathrm{e}^{\beta t}  \|\delta u_t^{\smallertext{X}_\smalltext{t}}\|^2 \d t \bigg] \\
    &\le C   \E^\P\bigg[\sup_{t \in [0, T]} (1+\|X_t\|^2)^k \bigg] \|\delta u\|^2_{\mathbb{S}^{\smalltext{2}\smalltext{,}\smalltext{2}}_{\smalltext{\beta}\smalltext{,} \smalltext{\rho}}(\R^\smalltext{n},\F,\P)} \le C M_\smallertext{X} \|\delta u\|^2_{\mathbb{S}^{\smalltext{2}\smalltext{,}\smalltext{2}}_{\smalltext{\beta}\smalltext{,} \smalltext{\rho}}(\R^\smalltext{n},\F,\P)},  
\end{align*}
Applying identical estimates for the Hessian term $\delta \mathfrak{u}$ (in the weighted $\mathbb{S}^{2,2}$ space) and the volatility gradient term $\delta v$ (in the weighted $\mathbb{H}^{2,2}$ space), and bounding the integrals, we arrive at the final estimate for the value process:
\begin{equation}\label{eq:est_Y_detailed}
    \|\delta Y\|^2_{\mathbb{S}^\smalltext{2}_\smalltext{\beta}(\R,\F,\P)} + \|\delta Z\|^2_{\mathbb{H}^\smalltext{2}_\smalltext{\beta}(\R^\smalltext{d},\F,\P)} \le \frac{3C M_\smallertext{X}}{\beta} \Big( \|\delta z\|^2_{\mathbb{H}^\smalltext{2}_\smalltext{\beta}(\R^{\smalltext{d}},\F,\P)}+\|\delta u\|^2_{\mathbb{S}^{\smalltext{2}\smalltext{,}\smalltext{2}}_{\smalltext{\beta}\smalltext{,} \smalltext{\rho}}(\R^\smalltext{n},\F,\P)} + \|\delta \mathfrak{u}\|^2_{\mathbb{S}^{\smalltext{2}\smalltext{,}\smalltext{2}}_{\smalltext{\beta}\smalltext{,} \smalltext{\rho}}(\R^{\smalltext{n}\smalltext{\times}\smalltext{n}},\F,\P)} + \|\delta v\|^2_{\mathbb{H}^{\smalltext{2}\smalltext{,}\smalltext{2}}_{\smalltext{\beta}\smalltext{,} \smalltext{\rho}}(\R^{\smalltext{n}\smalltext{\times}\smalltext{d}},\F,\P)} \Big).
\end{equation}

\medskip
    \item[$(ii)$]\textit{Estimation of the gradient process $(U, V)$.}
    Next, we consider the system for the gradient \eqref{eq:map_U_detailed}. For a fixed parameter $y \in \R^n$, the difference in the driver $G_1$ satisfies the Lipschitz condition stated in \Cref{ass:poly_growth}
    \[
        \|\Delta G_1(t, X_t, y)\| \le C \big( \|\delta V_t^y\| + \|\delta z_t\| + \|\delta v_t^y\| + \|\delta \mathfrak{u}_t^y\| \big).
    \]
    We apply the stability estimate from \Cref{lemma:elkaroui_est} for this fixed $y$ (or rather, a lifted version of it to $\R^n$), and follow the same reasoning as in Step 1. We obtain:
    \[
        \|\delta U^y\|^2_{\mathbb{S}^\smalltext{2}_\smalltext{\beta}(\R^\smalltext{n},\F,\P)} + \|\delta V^y\|^2_{\mathbb{H}^\smalltext{2}_\smalltext{\beta}(\R^{\smalltext{n} \smalltext{\times} \smalltext{d}},\F,\P)} \le \frac{C}{\beta} \E^\P \bigg[ \int_0^T \mathrm{e}^{\beta t}  \big( \|\delta z_t\|^2 + \|\delta v_t^y\|^2 + \|\delta \mathfrak{u}_t^y\|^2 \big) \d t \bigg].
    \]

We now lift this pointwise estimate to the functional space norm. We multiply the entire inequality by the fixed weight $\rho(y)$, and take the supremum. Using that $2\sup(a^2 + b^2) \geq \sup(a^2) + \sup(b^2)$, and potentially changing the constants, we obtain:

    \begin{equation}\label{eq:est_U_detailed}
        \|\delta U\|^2_{\mathbb{S}^{\smalltext{2}\smalltext{,}\smalltext{2}}_{\smalltext{\beta}\smalltext{,} \smalltext{\rho}}(\R^\smalltext{n},\F,\P)} + \|\delta V\|^2_{\mathbb{H}^{\smalltext{2}\smalltext{,}\smalltext{2}}_{\smalltext{\beta}\smalltext{,} \smalltext{\rho}}(\R^{\smalltext{n} \smalltext{\times} \smalltext{d}},\F,\P)} \le \frac{C  C_\rho}{\beta} \Big( \|\delta z\|^2_{\mathbb{H}^\smalltext{2}_\smalltext{\beta}(\R^{\smalltext{d}},\F,\P)} + \|\delta v\|^2_{\mathbb{H}^{\smalltext{2}\smalltext{,}\smalltext{2}}_{\smalltext{\beta}\smalltext{,} \smalltext{\rho}}(\R^{\smalltext{n} \smalltext{\times} \smalltext{d}},\F,\P)} + \|\delta \mathfrak{u}\|^2_{\mathbb{S}^{\smalltext{2}\smalltext{,}\smalltext{2}}_{\smalltext{\beta}\smalltext{,} \smalltext{\rho}}(\R^{\smalltext{n}\smalltext{\times} \smalltext{n}},\F,\P)} \Big).
    \end{equation}
    
    \medskip
    \item[$(iii)$]\textit{Estimation of the Hessian process $(\mathcal{U}, \mathcal{V})$.}
    The analysis for the Hessian system \eqref{eq:map_calU_detailed} mirrors that of the gradient exactly. The driver $G_2$ satisfies the same Lipschitz condition. Multiplying by $\rho(y)$, taking the supremum, and using the large $\beta$ estimate yields
    \begin{equation}\label{eq:est_calU_detailed}
        \|\delta \mathcal{U}\|^2_{\mathbb{S}^{\smalltext{2}\smalltext{,}\smalltext{2}}_{\smalltext{\beta}\smalltext{,} \smalltext{\rho}}(\R^{\smalltext{n}\smalltext{\times} \smalltext{n}},\F,\P)} + \|\delta \mathcal{V}\|^2_{\mathbb{H}^{\smalltext{2}\smalltext{,}\smalltext{2}}_{\smalltext{\beta}\smalltext{,} \smalltext{\rho}}(\R^{\smalltext{n}\smalltext{\times}\smalltext{n}\smalltext{\times}\smalltext{d}},\F,\P)} \le \frac{C  C_\rho}{\beta} \Big( \|\delta z\|^2_{\mathbb{H}^\smalltext{2}_\smalltext{\beta}(\R^\smalltext{d},\F,\P)} + \|\delta v\|^2_{\mathbb{H}^{\smalltext{2}\smalltext{,}\smalltext{2}}_{\smalltext{\beta}\smalltext{,} \smalltext{\rho}}(\R^{\smalltext{n}\smalltext{\times} \smalltext{d}},\F,\P)} + \|\delta \mathfrak{v}\|^2_{\mathbb{H}^{\smalltext{2}\smalltext{,}\smalltext{2}}_{\smalltext{\beta}\smalltext{,} \smalltext{\rho}}(\R^{\smalltext{n}\smalltext{\times}\smalltext{n}\smalltext{\times}\smalltext{d}},\F,\P)} \Big).
    \end{equation}

\medskip
    \item[$(iv)$]\textit{Conclusion.}
    We sum the inequalities \eqref{eq:est_Y_detailed}, \eqref{eq:est_U_detailed}, and \eqref{eq:est_calU_detailed}. Let $\|\delta \mathbf{W}\|^2_{\mathcal{K}_\smalltext{\beta}}$ denote the total squared norm of the difference in the output, which is the sum of the squared norms of all components. Similarly, let $\|\delta \mathbf{w}\|^2_{\mathcal{K}_\smalltext{\beta}}$ denote the norm of the input difference. Combining the estimates, we find
   \begin{align*}
        \|\delta \mathbf{W}\|^2_{\mathcal{K}_\smalltext{\beta}(\F,\P)} &= \Big( \|\delta Y\|^2_{\mathbb{S}^\smalltext{2}_\smalltext{\beta}(\R,\F,\P)} + \|\delta Z\|^2_{\mathbb{H}^\smalltext{2}_\smalltext{\beta}(\R^\smalltext{d},\F,\P)} \Big) + \Big( \|\delta U\|^2_{\mathbb{S}^{\smalltext{2}\smalltext{,}\smalltext{2}}_{\smalltext{\beta}\smalltext{,} \smalltext{\rho}}(\R^\smalltext{n},\F,\P)} + \|\delta V\|^2_{\mathbb{H}^{\smalltext{2}\smalltext{,}\smalltext{2}}_{\smalltext{\beta}\smalltext{,} \smalltext{\rho}}(\R^{\smalltext{n} \smalltext{\times} \smalltext{d}},\F,\P)} \Big) + \\ 
        &\quad+ \Big( \|\delta \mathcal{U}\|^2_{\mathbb{S}^{\smalltext{2}\smalltext{,}\smalltext{2}}_{\smalltext{\beta}\smalltext{,} \smalltext{\rho}}(\R^{\smalltext{n} \smalltext{\times} \smalltext{n}},\F,\P)} + \|\delta \mathcal{V}\|^2_{\mathbb{H}^{\smalltext{2}\smalltext{,}\smalltext{2}}_{\smalltext{\beta}\smalltext{,} \smalltext{\rho}}(\R^{\smalltext{n} \smalltext{\times} \smalltext{n} \smalltext{\times} \smalltext{d}},\F,\P)} \Big) \\
        &\le \frac{\tilde C}{\beta} \Big( \|\delta z\|^2_{\mathbb{H}^\smalltext{2}_\smalltext{\beta}(\R^\smalltext{d},\F,\P)} + \|\delta u\|^2_{\mathbb{S}^{\smalltext{2}\smalltext{,}\smalltext{2}}_{\smalltext{\beta}\smalltext{,} \smalltext{\rho}}(\R^\smalltext{n},\F,\P)} + \|\delta v\|^2_{\mathbb{H}^{\smalltext{2}\smalltext{,}\smalltext{2}}_{\smalltext{\beta}\smalltext{,} \smalltext{\rho}}(\R^{\smalltext{n} \smalltext{\times} \smalltext{d}},\F,\P)} + \|\delta \mathfrak{u}\|^2_{\mathbb{S}^{\smalltext{2}\smalltext{,}\smalltext{2}}_{\smalltext{\beta}\smalltext{,} \smalltext{\rho}}(\R^{\smalltext{n}\smalltext{\times} \smalltext{n}},\F,\P)} + \|\delta \mathfrak{v}\|^2_{\mathbb{H}^{\smalltext{2}\smalltext{,}\smalltext{2}}_{\smalltext{\beta}\smalltext{,} \smalltext{\rho}}(\R^{\smalltext{n}\smalltext{\times}\smalltext{n}\smalltext{\times}\smalltext{d}},\F,\P)} \Big) \\
        &\le \frac{\tilde{C}}{\beta} \|\delta \mathbf{w}\|^2_{\mathcal{K}_\beta(\F,\P)},
    \end{align*}

    where $\tilde{C}$ depends only on the Lipschitz-continuity constants, the weight parameter $k$, the maturity $T$ and the moments of $X$. By choosing $\beta > \tilde{C}$, the factor $\frac{\tilde{C}}{\beta}$ becomes strictly less than 1. This proves that the map $\Phi$ is a contraction on the Banach space $\mathcal{K}_\beta(\F,\P)$ when $\beta > \tilde{C}$. Consequently, by the Banach fixed-point theorem, there exists a unique fixed point $\mathbf{W}^{\star} \in \mathcal{K}_\beta(\F,\P)$ such that $\mathbf{W}^{\star} = \Phi(\mathbf{W}^{\star})$. This fixed point is the unique solution to the BSDE system \eqref{BSDE Weak Formulation} in $\mathcal{K}_\beta(\F,\P)$.
    \end{enumerate}
    \medskip
    
  The second part of the theorem is a direct consequence of Lemma \ref{lemma:embedding_spaces}.
\end{proof}

\end{appendix}

\end{document}